%% file: main.tex
\documentclass[reqno,a4paper,12pt]{amsart}
\usepackage{enumerate,amsmath,amssymb,stmaryrd} 
\usepackage{pstricks,pst-node,pst-coil,pst-plot} 
\usepackage{hyperref}
\usepackage{esint}
\usepackage{xfrac}
\usepackage{accents}
\input{def}

\begin{document} 

%%%%%%%%%%%%%%%%%%%%%%%%%%%%%%%%%%%%%%%%%%%%%%%%%%%%%%%%%%%%%%%%%%%%%%%%%%%%%%%%%%%%%%%%%%%%%%%

\title[Center of mass and STCMC-foliations]
{On center of mass and foliations\\ by constant spacetime mean curvature surfaces\\ for isolated systems in General Relativity}

%%%%%%%%%%%%%%%%%%%%%%%%%%%%%%%%%%%%%%%%%%%%%%%%%%%%%%%%%%%%%%%%%%%%%%%%%%%%%%%%%%%%%%%%%%%%%%%
 
\author{Carla Cederbaum}
\address{Carla Cederbaum: Mathematics Department, University of T\"ubingen, Germany}
\email{cederbaum@math.uni-tuebingen.de}

\author{Anna Sakovich}
\address{Anna Sakovich: Department of Mathematics, Uppsala University, Sweden}
\email{anna.sakovich@math.uu.se}
\begin{abstract}
We propose a new foliation of asymptotically Euclidean initial data sets by $2$-spheres of constant spacetime mean curvature (STCMC). The leaves of the foliation have the STCMC-property regardless of the initial data set in which the foliation is constructed which asserts that there is a plethora of STCMC $2$-spheres in a neighborhood of spatial infinity of any asymptotically flat spacetime. The STCMC-foliation can be understood as a covariant relativistic generalization of the CMC-foliation suggested by Huisken and Yau \cite{HY}.

We show that a unique STCMC-foliation exists near infinity of any asymptotically Euclidean initial data set with non-vanishing energy which allows for the definition of a new notion of total center of mass for isolated systems. This STCMC-center of mass transforms equivariantly under the asymptotic Poincar\'e group of the ambient spacetime and in particular evolves under the Einstein evolution equations like a point particle in Special Relativity. The new definition also remedies subtle deficiencies in the CMC-approach to defining the total center of mass suggested by Huisken and Yau \cite{HY} which were described by Cederbaum and Nerz~\cite{CN}.
\end{abstract}

\subjclass[2000]{53C21, (83C05, 83C30)}
% 53C21 Methods of Riemannian geometry, including PDE methods; 
% curvature restrictions
%
% 83C05 Einstein's equations (general structure, canonical formalism, 
% Cauchy problems) 
%
% 83C30 Asymptotic procedures (radiation, news functions, H-spaces, etc.) 

\date{\today}

\maketitle

%%%%%%%%%%%%%%%%%%%%%%%%%%%%%%%%%%%%%%%%%%%%%%%%%%%%%%%%%%%%%%%%%%%%%%%%
\section{Introduction and goals}
In General Relativity, isolated (gravitating) systems are individual or clusters of stars, black holes, or galaxies that do not interact with any matter or gravitational radiation outside the system under consideration. Intuitively, they should have a total center of mass which should in a suitable sense behave as a point particle in Special Relativity. In this paper, we suggest a definition of total center of mass for suitably isolated systems and argue that this center of mass notion indeed behaves as a point particle in Special Relativity in a suitable sense (meaning it transforms equivariantly under the asymptotic Poincar\'e group of the ambient spacetime). In particular, we will show that the center of mass notion we suggest evolves in time under the Einstein evolution equations like a point particle in Special Relativity.

The main idea of our approach is to modify the definition of center of mass given by Huisken and Yau \cite{HY} for asymptotically Euclidean Riemannian manifolds --- using an asymptotic foliation by $2$-spheres of constant mean curvature (CMC), see Section~\ref{secPrelim}~--- by staging it in a Lorentzian (spacetime) setting or in other words by staging it in asymptotically Euclidean initial data sets. More specifically, we will prove existence and uniqueness of an asymptotic foliation by $2$-spheres of constant spacetime mean curvature under optimal asymptotic decay assumptions. Here, ``spacetime constant mean curvature (STCMC)'' means that the co-dimen\-sion $2$ mean curvature vector $\vec{\mathcal{H}}$ of each $2$-sphere has constant Lorentzian length $\mathcal{H}$.

It is well known that this STCMC-condition can be reformulated in terms of initial data sets, namely as the product of the inner and outer ``expansions'' (or ``null mean curvatures'', see Remark \ref{PseudoStab}) with respect to any given null frame along a $2$-surface. On the other hand, the STCMC-condition is naturally independent of the initial data set in which the foliation is constructed. Our result thus asserts that there is a plethora of STCMC-surfaces in a neighborhood of spatial infinity of any asymptotically flat spacetime.

Furthermore, the new construction of a center of mass will be shown to remedy the subtle deficiencies of the Huisken and Yau approach \cite{HY} described by Cederbaum and Nerz \cite{CN}. Last but not least, we will provide an asymptotic flux integral formula for the center of mass extending that of Beig and \'O Murchadha~\cite{BeigOMurchadha}. The analytic techniques in our proofs rely on those developed by Metzger~\cite{Metzgerthesis} and Nerz~\cite{NerzCMC,NerzCE}.\\

Concluding this introduction, we would like to point out that the notion of spacetime mean curvature of $2$-surfaces in initial data sets has independently been considered in other contexts, both before and after the results of this paper had been announced. For example, the inverse spacetime mean curvature flow has been studied by Bray, Hayward, Mars, and Simon in \cite{BrayHaywardMarsSimon} and by Hangjun Xu in his thesis \cite{Xu}.

The STCMC-condition is (trivially) satisfied by marginally outer/inner trapped surfaces (MOTS/MITS), extremal surfaces (see e.g. \cite{EngelhardtWall}), and generalized apparent horizons (see e.g. \cite{Mars}, \cite{BrayKhuri}), with spacetime mean curvature $\mathcal{H}=0$ in all those cases. More generally, 2-surfaces with constant spacetime mean curvature are critical points for the area functional inside the future-directed null-cone, with mean curvature vector pointing in the direction in which the expansion of the surface is extremal. The aforementioned generalized apparent horizons have outer area minimizing property which is appealing in the view of spacetime Penrose Inequality. We would like to point the reader to the interesting work by Carrasco and Mars \cite{CarrascoMars} giving insights into the (over-)generality of $\mathcal{H}=0$ as a condition for a horizon. In a recent paper of Cha and Khuri \cite{ChaKhuri}, the area $A$ of the outermost STCMC-surface with $\mathcal{H}=2$ appears in the conjectured Penrose Inequality $m\geq \sqrt{\sfrac{A}{16\pi}}$ expected to hold for an asymptotically anti-de Sitter initial data set of mass $m$ satisfying the dominant energy condition. 

Because of the spacetime geometry nature of the STCMC-condition, we expect that STCMC-surfaces and STCMC-foliations will have a number of applications beyond the definition of a center of mass of an isolated system as well as beyond the setting of asymptotically Euclidean initial data sets. For example, a special subfamily of STCMC-surfaces foliating a null hypersurface implicitly appears in recent work by Klainerman and Szeftel \cite{KlainermanSzeftel}, where they arise as surfaces with both constant outer and constant inner expansion.\\

\paragraph*{\emph{Structure of the paper.}} In Section \ref{secPrelim}, we will summarize the necessary definitions and notations as well as more details on the background and existing work on the total center of mass of isolated systems. In Section \ref{secMain}, we will state our main results and very briefly explain the strategy of our proofs. The following sections will be dedicated to the more technical components of the proof with Section \ref{secHypersurfaces} focusing on a priori estimates for STCMC-surfaces, Section \ref{secLinearization} discussing the linearization of spacetime mean curvature, Section \ref{secExistence} asserting existence of the STCMC-foliation, Section~\ref{secADMstyleexpr} introducing the coordinate expression of the center of mass associated with the STCMC-foliation, and Section \ref{secEvolution} proving the claimed law of time evolution under the Einstein evolution equations. Appendix \ref{sesApp} collects results such as Sobolev Inequalities on $2$-surfaces, while Appendix \ref{apFermi} studies STCMC-surfaces in normal geodesic coordinates. Finally, in Section~\ref{secHYCN}, we will discuss an exemplary initial data set highlighting the differences between the newly suggested notion of center of mass and the existing one suggested by Huisken and Yau.\\

\paragraph*{\emph{Acknowledgements}}
We would like to thank Julien Cortier who was involved in this project at early stages and contributed with many insightful discussions.

The authors would like to extend thanks to the Erwin Schr\"odinger Institute, the Institut Henri Poincar\'e, the Mathematisches Forschungsinstitut Oberwolfach, and the Max Planck Institute for Mathematics for allowing us to collaborate in stimulating environments.

CC is indebted to the Baden-W\"urttemberg Stiftung for the financial support of this research project by the Eliteprogramme for Postdocs. The work of CC is supported by the Institutional Strategy of the University of T\"ubingen (Deutsche Forschungsgemeinschaft, ZUK 63). CC also thanks the Fondation des Sciences Math\'ematiques de Paris for generous support. The work of AS was supported by Knut and Alice Wallenberg Foundation and Swedish Research Council (Vetenskapsr{\aa}det). We also thank the Crafoord Foundation for generous support.

\section{Preliminaries}\label{secPrelim}
\input{preliminaries}

%%%%%%%%%%%%%%%%%%%%%%%%%%%%%%%%%%%%%%%%%%%%%%%%%%%%%%%%%%%%%%%%%%%%%%%%

\section{Main results, motivation, and the strategy of the proof}\label{secMain}
\input{result}

%%%%%%%%%%%%%%%%%%%%%%%%%%%%%%%%%%%%%%%%%%%%%%%%%%%%%%%%%%%%%%%%%%%%%%%%

\section{A priori estimates on STCMC-surfaces}\label{secHypersurfaces}
\input{on-center}

%%%%%%%%%%%%%%%%%%%%%%%%%%%%%%%%%%%%%%%%%%%%%%%%%%%%%%%%%%%%%%%%%%%%%%%%%%
\newpage
\section{The linearization of spacetime mean curvature}\label{secLinearization}
\input{linearization}

%%%%%%%%%%%%%%%%%%%%%%%%%%%%%%%%%%%%%%%%%%%%%%%%%%%%%%%%%%%%%%%%%%%%%%%%%

\section{Existence and uniqueness of the STCMC-foliation}\label{secExistence}
\input{existence}

%%%%%%%%%%%%%%%%%%%%%%%%%%%%%%%%%%%%%%%%%%%%%%%%%%%%%%%%%%%%%%%%%%%%%%%%%%

\section{The coordinate center of the STCMC-foliation}\label{secADMstyleexpr}
\input{CoM}

%%%%%%%%%%%%%%%%%%%%%%%%%%%%%%%%%%%%%%%%%%%%%%%%%%%%%%%%%%%%%%%%%%%%%%%%%

\section{Time evolution and Poincar\'e covariance\\ of the STCMC-center of mass}\label{secEvolution}
\input{evolution}

%%%%%%%%%%%%%%%%%%%%%%%%%%%%%%%%%%%%%%%%%%%%%%%%%%%%%%%%%%%%%%%%%%%%%%%%%

\section{A concrete graphical example in the Schwarzschild spacetime}\label{secHYCN}
\input{HY-CN}

%%%%%%%%%%%%%%%%%%%%%%%%%%%%%%%%%%%%%%%%%%%%%%%%%%%%%%%%%%%%%%%%%%%%%%%%%%

\appendix
\section{Round surfaces in asymptotically Euclidean manifolds}\label{sesApp}
\input{appendix}

%%%%%%%%%%%%%%%%%%%%%%%%%%%%%%%%%%%%%%%%%%%%%%%%%%%%%%%%%%%%%%%%%%%%%%%%%

\section{The STCMC-condition in normal geodesic coordinates}\label{apFermi}
\input{appendix2}

%%%%%%%%%%%%%%%%%%%%%%%%%%%%%%%%%%%%%%%%%%%%%%%%%%%%%%%%%%%%%%%%%%%%%%%%%
\bibliographystyle{amsplain}
\bibliography{biblio}

%%%%%%%%%%%%%%%%%%%%%%%%%%%%%%%%%%%%%%%%%%%%%%%%%%%%%%%%%%%%%%%%%%%%%%%%%
\end{document}

%% file: def.tex
\theoremstyle{plain}

\newtheorem{thm}{Theorem}[section]
\newtheorem{theorem}[thm]{Theorem}

\newtheorem{corollary}[thm]{Corollary}

\newtheorem{lemma}[thm]{Lemma}

\newtheorem{proposition}[thm]{Proposition}
\newtheorem{conjecture}[thm]{Conjecture}

\theoremstyle{remark}
\newtheorem{example}[thm]{Example}
\newtheorem{remark}[thm]{Remark}

\theoremstyle{definition}

\newtheorem{definition}[thm]{Definition}

\newcounter{mnotecount}[section]

\newcommand{\definedas}{\mathrel{\raise.095ex\hbox{\rm :}\mkern-5.2mu=}}
 \newcommand{\asdefined}{\mathrel{=\mkern-5.2mu\raise.095ex\hbox{\rm :}}}

\def\epsilon{{\varepsilon}}
\def\phi{{\varphi}}

\let\<\langle 
\let\>\rangle 

\newcommand{\IDS}{(M^3,g,K,\mu,J)}
\newcommand{\IDStau}{(M^3,g,\tau K,\mu_\tau,\tau J)}

\newcommand{\R}{\mathbb{R}}

\newcommand{\bR}{\mathbb{R}}

\renewcommand{\hbar}{\overline{h}}

\DeclareMathOperator{\graph}{graph}
\DeclareMathOperator{\tr}{tr}
\DeclareMathOperator{\divg}{div}

\DeclareMathOperator{\id}{Id}

\DeclareMathOperator{\Span}{Span}

\newcommand{\ric}{\mathrm{Ric}}
\newcommand{\mric}{\mathfrak{{Ric}}}

\newcommand{\scal}{\mathrm{Scal}}
\newcommand{\mscal}{\mathfrak{Scal}}

\DeclareMathOperator{\hess}{Hess}

%% file: preliminaries.tex
% !TEX root = main.tex
Recall that an \emph{initial data set} for the Einstein equations is a tuple $\IDS$ where $(M^{3},g)$ is a smooth Riemannian manifold and $K$ is a smooth symmetric $(0,2)$-tensor field on $M^3$ playing the role of the second fundamental form of $M^{3}$ in an ambient Lorentzian spacetime. The (scalar) \emph{local energy density} $\mu$ and the ($1$-form) \emph{local momentum density} $J$ defined on $M^{3}$ can be read off from the \emph{constraint equations}\\[-2ex]
\begin{subequations}\label{CE}
\begin{align}
 \scal - \vert K\vert ^2 + (\tr K)^2 \label{eqHamiltonian}&=2\mu\\
\divg (K - (\tr K) g)&=J. \label{eqMomentum}
\end{align}
\end{subequations}
Here, $\tr$, $\divg$, and $\vert\cdot\vert$ denote the trace, the divergence, and the tensor norm with respect to $g$, respectively, and $\scal$ denotes its scalar curvature. Sometimes we will find it convenient to use the \emph{conjugate momentum tensor} $\pi \definedas (\tr K) g -K$. 

The constraint equations \eqref{CE} arise as a consequence of the Gauss--Codazzi--Mainardi equations from the Einstein equations $\mric-\tfrac{1}{2}\,\mscal\,\mathfrak{g}=\mathfrak{T}$ satisfied by a given spacetime $(\mathfrak{M}^{1,3},\mathfrak{g})$ with \emph{energy-momentum tensor} $\mathfrak{T}$, where $g$ is the Riemannian metric induced by the Lorentzian metric $\mathfrak{g}$ on the spacelike hypersurface $M^{3}$ and $K$ is the induced second fundamental form. Letting $\eta$ denote the timelike future unit normal to the initial data set $\IDS$, the energy and momentum density are derived from $\mathfrak{T}$ via $\mu=\mathfrak{T}(\eta,\eta)$, and $J=\mathfrak{T}(\eta,\cdot)$, and the \emph{stress tensor $S$} on $M^3$ is defined by $S=\mathfrak{T}(\cdot,\cdot)$\label{textmuJ}. The constraint equations \eqref{CE} thus necessarily hold on any spacelike hypersurface (or "initial data set") $\IDS$ in the spacetime $(\mathfrak{M}^{1,3},\mathfrak{g})$.

In order to model an "isolated system", we will assume that the ambient spacetime $(\mathfrak{M}^{1,3},\mathfrak{g})$ with its energy-momentum tensor $\mathfrak{T}$ and the choice of initial data set $\IDS$ are such that the initial data set is "asymptotically Euclidean", a notion made precise in the following standard definition.
\begin{definition}[Asymptotically Euclidean initial data sets]\label{defAEdata}\label{defAE}
Let $\varepsilon \in (0,\tfrac{1}{2}]$
 and let $\IDS$ be a smooth initial data set. Assume there is a smooth coordinate chart $\vec{x}\colon M^{3} \setminus \mathcal{B} \to \mathbb{R}^3 \setminus \overline{B_R(0)}$ defined in the region exterior to a compact set $\mathcal{B} \subset M^{3}$. We say that $\mathcal{I}\definedas\IDS$ is a \emph{$C^2_{\sfrac{1}{2}+\varepsilon}$-asymptotically Euclidean initial data set (with respect to $\vec{x}$)} if there is a constant $C=C(\mathcal{I},{\vec{x}})$ such that, in the coordinates $\vec{x}=(x^1, x^2, x^3)\in\R^3\setminus\overline{B_R(0)}$, we have the pointwise estimates
\begin{subequations}\label{AE}
\begin{align}
\vert g_{ij}-\delta_{ij}\vert  + \vert\vec{x}\,\vert \vert \partial_k g_{ij}\vert  + \vert\vec{x}\,\vert ^2 \vert \partial_k\partial_l g_{ij}\vert  & \leq   C \vert\vec{x}\,\vert ^{-\frac{1}{2}-\varepsilon} \label{eqOptDecay1}\\
\vert K_{ij}\vert  +  \vert\vec{x}\,\vert \vert \partial_k K_{ij}\vert   & \leq  C \vert\vec{x}\,\vert ^{-\frac{3}{2}-\varepsilon} \label{eqOptDecay2} \\
\vert \mu\vert  + \vert J_i\vert  & \leq C \vert\vec{x}\,\vert ^{-3-\varepsilon}  \label{eqOptDecay3}
\end{align}
\end{subequations}
for all $\vec{x}\in\mathbb{R}^{3}\setminus \overline{B_R(0)}$ and for all $i,j,k,l \in \{1,2,3\}$. Here, slightly abusing notation, we have silently pushed forward all tensor fields on $M^{3}$ (including scalars) and written $g_{ij}\definedas(\vec{x}_{*}g)_{ij}$ as well as $K_{ij}\definedas(\vec{x}_{*}K)_{ij}$, etc. The Kronecker delta $\delta_{ij}$ denotes the components of the Euclidean metric with respect to the coordinates $\vec{x}$. By another slight abuse of notation, we will refer to the above constant $C$ as $C_{\mathcal{I}}$, suppressing the dependence on the chart $\vec{x}$.
\end{definition}

Asymptotically Euclidean initial data sets are well-known to possess well-defined total energy, linear momentum, and mass. More precisely, if $\mathcal{I}=(M,g,K,\mu,J)$ is a $C^2_{\sfrac{1}{2}+\varepsilon}$-asymptotically Euclidean initial data set for any $\varepsilon>0$ (naturally extending the definition to $\varepsilon>\tfrac{1}{2}$), its \emph{(ADM-)energy} $E$ and its \emph{(ADM-)linear momentum} $\vec{P}=(P^{1},P^{2},P^{3})$ are given by
\begin{align}\label{eqADME}
E & \definedas  \frac{1}{16\pi} \lim_{r\to\infty} \int\limits_{\vert\vec{x}\,\vert =r} \sum_{i,j} (\partial_i g_{ij} - \partial_j g_{ii}) \frac{x^j}{r} \, d\mu^\delta, \\\label{eqADMP}
P^{j} & \definedas  \frac{1}{8\pi} \lim_{r\to\infty} \int\limits_{\vert\vec{x}\,\vert =r} \sum_{i} \pi_{ij} \frac{x^i}{r} \, d\mu^\delta,
\end{align}
respectively, where $d\mu^{\delta}$ denotes the area measure induced on the coordinate sphere $\lbrace{\vert\vec{x}\,\vert=r\rbrace}$ by the Euclidean metric $\delta$ and ADM stands for Arnowitt--Deser--Misner~\cite{ADM}. The quantities $E$ and $\vec{P}$ are well-defined under the asymptotic conditions imposed here for arbitrary $\varepsilon>0$~\cite{Bartnik,Chrusciel} --- meaning the expressions converge and $E$ is asymptotically independent of the chart $\vec{x}$ while $\vec{P}$ is asymptotically covariant under chart deformations in a suitable way. From them, one defines the \emph{(ADM-)mass}~by
\begin{align}
m&\definedas\sqrt{E^{2}-\vert\vec{P}\vert^{2}}
\end{align}
whenever this expression makes sense, that is whenever the \emph{energy-momentum $4$-vector $(E,\vec{P})$} is causal with respect to the Minkowski metric of Special Relativity.

\begin{remark}[Bounds on $\varepsilon$]
For $\varepsilon\leq0$ in the above definition, one can find an asymptotic chart $\vec{x}$ (meaning a coordinate transformation outside a compact set) on the canonical Euclidean initial data set $\mathcal{I}_{\text{Eucl.}}=(\R^{3},\delta,K\equiv0,\mu\equiv0,J\equiv0)$ with respect to which $\mathcal{I}_{\text{Eucl.}}$ is $C^{2}_{\sfrac{1}{2}+\varepsilon}$-asymptotically Euclidean but the expression $E$ does not vanish as it should for Euclidean space, see Denisov and Soloviev \cite{DS}. This explains the suggestive notation of the decay order as $\sfrac{1}{2}+\varepsilon$. 

On the other hand, if $\varepsilon>\tfrac{1}{2}$ for an initial data set $\mathcal{I}$, a simple computation shows that it has $E=\vec{P}=0$ which is non-desirable in the context of discussing the center of mass and asymptotic foliations by constant mean curvature. This explains why we exclude this case in Definition \ref{defAEdata}.
\end{remark}

\begin{remark}[Asymptotically Euclidean Riemannian manifolds]\label{remAERiem}
If a Riemannian manifold $(M^{3},g)$ (with asymptotic chart $\vec{x}$) satisfies \eqref{eqOptDecay1} and its scalar curvature satisfies $\vert\scal\vert \leq C \vert\vec{x}\,\vert ^{-3-\varepsilon}$ for all $\vec{x}\in\mathbb{R}^{3}\setminus \overline{B_R(0)}$,  we say that $(M^{3},g)$ is a \emph{$C^2_{\sfrac{1}{2}+\varepsilon}$-asymptotically Euclidean manifold}. This is called the ``Riemannian case'', the reason being that one can reinterpret this as saying that the ``trivially extended initial data set'' $(M^{3},g,K\equiv0,\mu=\tfrac{1}{2}\,\scal,J\equiv0)$ satisfies \eqref{eqOptDecay1}-\eqref{eqOptDecay3}. In the Riemannian case, the notions ``mass" $m$ and ``energy" $E$ can be and are used interchangeably. 
\end{remark}

\subsection{Center of mass} We now proceed to discussing the total center of mass of an asymptotically Euclidean initial data set $\mathcal{I}=\IDS$ with energy $E\neq0$. The assumption $E\neq0$ is both technical (as many definitions of center of mass explicitly divide by $E$) and physically reasonable when considering the center of mass.

First, let us remark that our field knows many definitions of center of mass for isolated systems. The first definitions were given in terms of asymptotic flux integral expressions in coordinates, similar to those of energy and linear momentum above, see \eqref{BOMRT} below and the text surrounding it. In 1996, Huisken and Yau~\cite{HY} proved existence and uniqueness of a foliation by constant mean curvature $2$-spheres near infinity of an asymptotically Euclidean Riemannian manifold of positive energy $E>0$ and related it to a definition of center of mass in a way described below and in more detail in Section \ref{secADMstyleexpr}. More recently, Chen, Wang, and Yau~\cite{CWY} suggested a new definition of center of mass for isolated systems which is constructed from optimal isometric embeddings into the flat Minkowski spacetime of Special Relativity. For a brief, non-complete summary of other definitions of center of mass, please see \cite{CN}.\\

\emph{Flux integral definition.} The most prominent flux integral notion of center of mass $\vec{C}_\text{B\'OM}=(C^{1}_\text{B\'OM},C^{2}_\text{B\'OM},C^{3}_\text{B\'OM})$ for asymptotically Euclidean initial data sets was introduced by Beig and \'O Murch\-adha~\cite{BeigOMurchadha} as the asymptotic flux integral
\begin{align}\label{BOMRT}
C^{\,l}_\text{B\'OM} &\definedas \frac{1}{16\pi E} \lim_{r\to\infty} \int\limits_{\vert\vec{x}\,\vert =r} \left[x^l \sum_{i,j} (\partial_i g_{ij} - \partial_j g_{ii})\frac{x^j}{r} - \sum_i \left(g_{il}\frac{x^i}{r}-g_{ii}\frac{x^l}{r}\right)\right]\, d\mu^\delta,
\end{align}
a definition going back in parts to Regge and Teitelboim~\cite{RT}. See Szabados~\cite{Szabados} for valuable critical comments on this definition, and see Section \ref{secADMstyleexpr} for a covariant generalization of this formula following from our work.

The center of mass integral $\vec{C}_{\text{B\'OM}}$ will in general \emph{not} converge for initial data sets $\mathcal{I}=\IDS$ which are merely $C^2_{\sfrac{1}{2}+\varepsilon}$-asymptotically Euclidean with respect to some chart $\vec{x}$ and have $E\neq0$.  It will however converge once one assumes that the initial data set satisfies certain asymptotic symmetry conditions in the given chart $\vec{x}$, as for example the Regge--Teitelboim conditions introduced in~\cite{RT}, see \cite{BeigOMurchadha,ChruscielDelay,Huang} and Definition \ref{defRT} below. We also point out that the expression for $\vec{C}_\text{B\'OM}$ does not explicitly depend on the second fundamental form $K$ of the initial data set.\\

\emph{Definitions via foliations.} Several authors define the center of mass of an initial data set $\mathcal{I}=\IDS$ via a foliation by $2$-spheres near infinity. Following Cederbaum and Nerz~\cite{CN}, we will call such definitions ``abstract'' in contrast to the more explicit ``coordinate definitions'' of center of mass, see below.

The first abstract definition of center of mass was given in 1996 by Huisken and Yau~\cite{HY}, who proved existence and uniqueness of a foliation near the asymptotic end of an asymptotically Euclidean Riemannian manifold by closed, stable $2$-spheres of constant mean curvature, the \emph{CMC-foliation}. This goes back to an idea of Christo\-doulou and Yau~\cite{christodoulou}. In 2006, Metzger~\cite{Metzgerthesis} considered a foliation by $2$-spheres of constant null mean curvature (also called constant expansion) and concluded that this foliation is not fully suitable for defining a center of mass. For a more detailed review of foliations suggested to study in this context and of recent progress in terms of necessary and sufficient asymptotic decay conditions, please see \cite{CN}. 

Huisken and Yau~\cite{HY} also assign a coordinate center to their foliation. It is constructed from the abstract CMC-center as a ``Euclidean center'' of the CMC-foliation as follows: First, any closed, oriented $2$-surface $\Sigma\hookrightarrow\mathbb{R}^3$ has a \emph{Euclidean coordinate center} $\vec{c}\,(\Sigma)$ defined by
\begin{align}
\vec{c}\,(\Sigma)&\definedas \fint\limits_{\Sigma} \vec{x}\, d\mu^\delta \definedas \frac1{\vert\Sigma\vert_{\delta}} \int\limits_{\Sigma} \vec{x}\, d\mu^\delta.
\end{align}
Picking a fixed asymptotically flat coordinate chart $\vec{x}\colon M^3\setminus\mathcal{B}\to\mathbb{R}^3\setminus \overline{B_R(0)}$, this definition can naturally be extended to closed, oriented $2$-surfaces $\Sigma\hookrightarrow M^3\setminus \mathcal{B}$ by pushing $\Sigma$ forward to $\R^3$ and identifying $\vec{c}\,(\Sigma)\definedas\vec{c}\,(\vec{x}(\Sigma))$, slightly abusing notation. We will also call this center \emph{Euclidean center of $\Sigma$ (with respect to $\vec{x}$)}. This naturally extends to asymptotic foliations:

\begin{definition}[Coordinate center of a foliation]\label{folicoordcenter of mass}
Let $\mathcal{I}=\IDS$ be a $C^2_{\sfrac{1}{2}+\epsilon}$-asymptotically Euclidean initial data set for a chart $\vec{x}\colon M^3\setminus\mathcal{B}\to\R^3\setminus \overline{B_R(0)}$. Let $\{\Sigma^\sigma\}_{\sigma>\sigma_0}$ be a foliation of the asymptotic end $M^3\setminus\mathcal{B}$ of $M^3$ with area radius $r(\Sigma^\sigma)=\sqrt{\sfrac{\vert\Sigma^{\sigma}\vert}{4\pi}}$ of $\Sigma^\sigma$ diverging to $\infty$ as $\sigma\to\infty$. Denote by $\vec{c}\,(\Sigma^\sigma)$ the Euclidean coordinate center of the leaf $\Sigma^\sigma$ with respect to~$\vec{x}$. Then the \emph{(Euclidean) coordinate center $\vec{C}=(C^1,C^2,C^3)$ of the foliation $\{\Sigma^\sigma\}_{\sigma>\sigma_0}$ (with respect to the asymp\-totic chart $\vec{x}$)} is given by 
\begin{align}
\vec{C}\definedas\lim_{\sigma\to\infty}\vec{c}\,(\Sigma^\sigma),
\end{align}
in case the limit exists. Otherwise, we say that the coordinate center of the foliation $\{\Sigma^\sigma\}_{\sigma>\sigma_0}$ \emph{diverges (with respect to the asymptotic chart $\vec{x}$)}. 
\end{definition}

The vector $\vec{C}$ can be pictured to describe a point in the target $\R^3$ of the asymptotically flat coordinate chart $\vec{x}\colon M^3\setminus\mathcal{B}\to\R^3\setminus\overline{B_R(0)}$, but it need not lie in the image of $\vec{x}$, and indeed will often lie inside $\overline{B_R(0)}$. This means it cannot necessarily be pulled back into $M^3$. Moreover, $\vec{C}$ depends on the choice of asymptotic chart $\vec{x}$ --- at least a priori.\\

Coming back to the CMC-foliation constructed by Huisken and Yau \cite{HY}, it is well-known that the coordinate center $\vec{C}_\text{HY}$ of the CMC-foliation of a suitably asymptotically flat Riemannian manifold $(M^{3},g)$ or initial data set $\IDS$ of non-vanishing energy $E$ with respect to a given asymptotic chart $\vec{x}$ coincides with the Beig--\'O Murchadha center of mass vector $\vec{C}_\text{B\'OM}$ defined by \eqref{BOMRT}, provided that some additional symmetry assumptions are satisfied, see Huang~\cite{HuangstableCMC}, Eichmair and Metz\-ger~\cite{Metzger-Eichmair}, and Nerz~\cite{Nerzevo}. The most optimal result to date \cite[Theorem 6.3]{NerzCMC} states that for $C^2_{\sfrac{1}{2}+\varepsilon}$-asymptotically Euclidean Riemannian manifolds with $E\neq0$ satisfying the $C^2_{1 + \varepsilon}$-Regge--Teitelboim condition (see Definition \ref{defRT} below),  we have $\vec{C}_\text{HY}=\vec{C}_\text{B\'OM}$ whenever both definitions converge, and that divergence of one implies divergence of the other. Again, let us point out that the construction of $\vec{C}_{HY}$ does not explicitly depend on the second fundamental form $K$ of the initial data set under consideration. We furthermore note that the product $E\,\vec{C}_\text{B\'OM}$ is sometimes referred to as the ``center of mass charge" in the literature, even when $E=0$. We will not follow this usage here.\\

In this paper, we construct a novel geometric foliation $\{\Sigma^\sigma\}_{\sigma>\sigma_0}$ of the asymptotically flat end $M^3\setminus\mathcal{B}$ of a given $C^2_{\sfrac{1}{2}+\varepsilon}$-asymptotically Euclidean initial data set $\IDS$ with non-vanishing energy $E\neq0$, namely a foliation with ``constant spacetime mean curvature (STCMC)"-leaves, see Section \ref{secMain}. The general approach to define the coordinate center of a foliation $\{\Sigma^\sigma\}_{\sigma>\sigma_0}$ described above will then be applied to this new foliation to obtain a new definition of the coordinate center of mass of an initial data set as well as a coordinate expression analogous to and extending~\eqref{BOMRT}, see Section \ref{secADMstyleexpr}.

\subsection{Miscellannea} Here we collect some other definitions for future reference.\\[-1ex]

\emph{Regge--Teitelboim condition for initial data sets.} With the exception of the later part of Section \ref{secADMstyleexpr}, we will \emph{not} assume that the initial data sets under consideration satisfy any asymptotic symmetry assumptions, in particular the Regge--Teitelboim conditions. However, it will be useful in our discussion to refer to those conditions which is why we define them here.

\begin{definition}[Regge--Teitelboim conditions for initial data sets]\label{defRT}
We say that a $C^2_{\sfrac{1}{2}+\varepsilon}$-asymp\-tot\-ic\-ally Euclidean initial data set $\mathcal{I}=\IDS$ satisfies the \emph{$C^2_{\gamma + \varepsilon}$-Regge--Teitelboim conditions} for $\gamma>\tfrac{1}{2}$ (with respect to the given chart $\vec{x}$ with respect to which it is $C^2_{\sfrac{1}{2}+\varepsilon}$-asymptotically Euclidean) if there is a constant $C=C(\mathcal{I},{\vec{x}},\gamma)$ such that
\begin{subequations}
\begin{align}
\left\vert g_{ij}^{\mathrm{odd}}\right\vert  + \vert\vec{x}\,\vert \left\vert  \partial_k (g_{ij}^{\mathrm{odd}})\right\vert  + \vert\vec{x}\,\vert ^2 \left\vert \partial_k\partial_l (g_{ij}^{\mathrm{odd}})\right\vert  & \leq    C \vert\vec{x}\,\vert ^{-\gamma-\varepsilon} \label{eqRT1} \\
\vert \pi_{ij}^{\mathrm{even}}\vert  +  \vert\vec{x}\,\vert \vert \partial_k (\pi_{ij}^{\mathrm{even}})\vert  &\leq  C \vert\vec{x}\,\vert ^{-1-\gamma-\varepsilon} \label{eqRT2}\\
\vert \mu^{\mathrm{odd}}\vert  + \vert (J^j)^{\mathrm{odd}}\vert  & \leq C \vert\vec{x}\,\vert ^{-\frac{5}{2}-\gamma-\varepsilon}  \label{eqRT3}
\end{align}
\end{subequations}
holds for all $\vec{x}\in\mathbb{R}^{3}\setminus \overline{B_R(0)}$ and for all $i,j,k,l\in\{1,2,3\}$. Here, as usual, we have denoted the even and odd parts of any continuous function $f\colon \mathbb{R}^{3}\setminus \overline{B_R(0)}\to\mathbb{R}$ by
\begin{align}
f^{\mathrm{odd}}(\vec{x})\definedas \tfrac{1}{2}(f(\vec{x})-f(-x)), \qquad f^{\mathrm{even}}(\vec{x}) \definedas \tfrac{1}{2}(f(\vec{x})+f(-\vec{x})).
\end{align}
\end{definition}

\begin{remark}[Regge--Teitelboim conditions for Riemannian manifolds]\label{remRTRiemannian}
We say that a $C^2_{\sfrac{1}{2}+\varepsilon}$-asympt\-ot\-ic\-ally Euclidean Riemannian manifold $(M^3,g)$ satisfies the \emph{$C^2_{\gamma + \varepsilon}$-(Riemannian) Regge--Teitelboim conditions} on $\mathbb{R}^{3}\setminus \overline{B_R(0)}$ for $\gamma>\tfrac{1}{2}$ if the above inequalities are satisfied for $\pi\equiv K\equiv0$, i.e.~if \eqref{eqRT1} holds and if
$\vert \scal^{\mathrm{odd}}\vert  \leq C \vert\vec{x}\,\vert ^{-\frac{5}{2}-\gamma-\varepsilon}$ for all $\vec{x}\in\mathbb{R}^{3}\setminus \overline{B_R(0)}$.
\end{remark}

\emph{Weighted Sobolev spaces.} In this paper, we use the following definition of Sobolev spaces, which is well-suited for keeping track of fall-off rates of different quantities associated with our foliation. Suppose that $(\Sigma, g^\Sigma)$ is a closed (compact without boundary), oriented $2$-surface in an asymptotically Euclidean $3$-manifold $(M^{3},g)$ of suitable regularity.
For $p\in [1,\infty)$, the Lebesgue space $L^p(\Sigma)$ is defined as the set of all measurable functions $f\colon\Sigma\to\R$ such that their $L^p$-norm 
\begin{align*}
\Vert f\Vert _{L^p(\Sigma)}\definedas \left(\int_\Sigma \vert f\vert ^p \, d\mu\right)^{\frac{1}{p}}
\end{align*}
is finite. Recall also that the $L^\infty$-norm of a measurable $f\colon\Sigma\to\R$ is defined by $\Vert f\Vert _{L^\infty(\Sigma)} \definedas \mathrm{ess}\sup_\Sigma \vert f\vert $. Then for $p\in [1,\infty]$, the Sobolev norms are defined as follows:
\begin{align*}
\Vert f\Vert _{W^{0,p}(\Sigma)} \definedas \Vert f\Vert _{L^p(\Sigma)}, \qquad \Vert f\Vert _{W^{k+1,p}(\Sigma)} \definedas \Vert f\Vert _{L^p(\Sigma)}+ r \left\Vert \nabla^\Sigma f\right\Vert _{W^{k,p}(\Sigma)}, \quad k=0,1, \ldots
\end{align*}
where $r\definedas\sqrt{\sfrac{\vert \Sigma\vert}{4\pi}}$ is the area radius of $\Sigma$. The Sobolev space $W^{k,p}(\Sigma)$ is the set of all functions with finite $W^{k,p}$-norm. This definition naturally extends to the case of tensor fields on $\Sigma$. Appendix \ref{sesApp} in particular collects some Sobolev Inequalities for functions on $2$-surfaces $(\Sigma,g^\Sigma)$ embedded in Euclidean space.

%% file: result.tex
% !TEX root = main.tex
Given a $2$-dimensional surface $\Sigma$ in an initial data set $\IDS$, we denote its mean curvature inside the Riemannian manifold $(M^3,g)$ with respect to the outward pointing unit normal\footnote{Please note that we use the convention for the sign of the second fundamental form ensuring that $H=2$ with respect to the outward pointing unit normal for the unit round sphere in $\R^3$.} by $H$ and set $P \definedas \tr_\Sigma K$, as usual. The \emph{spacetime mean curvature (STMC)} of $\Sigma$ is defined by the length of the \emph{spacetime mean curvature vector} $\vec{\mathcal{H}}$
\begin{align}
\mathcal{H} = \sqrt{H^2 - P^2}.
\end{align} 
We will suggestively write $\mathcal{H}^\sigma_\tau$ to denote the spacetime mean curvature of a surface called $\Sigma^\sigma_\tau$ etc, $\widetilde{\mathcal{H}}$ to denote the spacetime mean curvature of a surface called $\widetilde{\Sigma}$ etc., whenever the initial data set inducing the intrinsic and extrinsic geometry on the surface is clear from context.

In this paper we prove the following theorems.
\setcounter{section}{6}
\setcounter{thm}{1}
\begin{theorem}[Existence of STCMC-foliation]
Let $\mathcal{I}=\IDS$ be a $C^2_{\sfrac{1}{2} + \varepsilon}$-asymptotically Euclidean initial data set with non-vanishing energy $E \neq 0$. Then there is a constant $\sigma_{\mathcal{I}}>0$ depending only on $\varepsilon$, $C_{\mathcal{I}}$, and $E$, a compact set $\mathcal{K} \subset M^{3}$, and a bijective $C^{1}$-map $\Psi\colon (\sigma_\mathcal{I}, \infty) \times \mathbb{S}^2 \to M^{3} \setminus \mathcal{K}$ such that each of the surfaces $\Sigma^\sigma \definedas \Psi(\sigma, \mathbb{S}^2)$ has constant spacetime mean curvature $\mathcal{H}( \Sigma^\sigma) \equiv\sfrac{2}{\sigma}$ provided that $\sigma>\sigma_{\mathcal{I}}$. 
\end{theorem}

\setcounter{thm}{11}
\begin{theorem}[Uniqueness of STCMC-foliation] 
Let $a \in [0,1)$, $b \geq 0$, and $\eta\in (0,1]$ be constants and let $\mathcal{I}=\IDS$ be a
$C^2_{\sfrac{1}{2} + \varepsilon}$-asymptotically Euclidean initial data set with
non-vani\-shing energy $E \neq 0$. Then there is a constant $\sigma_{\mathcal{I}}$  depending only on $\varepsilon$, $a$, $b$, $\eta$, $C_\mathcal{I}$, and $E$, such that for all $\sigma > \sigma_{\mathcal{I}}$, there is a unique surface $\Sigma^{\sigma} \in \mathcal{A}(a,b,\eta)$
with constant spacetime mean curvature $\mathcal{H}(\Sigma^{\sigma}) \equiv\sfrac{2}{\sigma}$ with respect to~$\mathcal{I}$.
\end{theorem}

Here, $\mathcal{A}(a,b,\eta)$ is an a priori class of ``asymptotically centered'' spheres introduced in Section \ref{secHypersurfaces}. It has been shown in particular by Brendle and Eichmair \cite{BrendleEichmair} that such an a priori condition is necessary to obtain uniqueness of CMC-surfaces in general, see the discussion in Section~\ref{subsecUnique}. As STCMC-surfaces generalize CMC-surfaces, their observation applies here, too.

\setcounter{section}{7}
\setcounter{thm}{4}
We also obtain a coordinate expression $\vec{C}_{\text{STCMC}}$ for the STCMC-center of mass, see below. It differs from the Beig--\'O Murchadha formula $\vec{C}_{\text{B\'OM}}$ given in \eqref{BOMRT} by a term $\vec{Z}$, as stated in the following theorem.

\begin{theorem}[STCMC-coordinate expression]
Let $\mathcal{I}=\IDS$ be a $C^2_{\sfrac{1}{2}+\varepsilon}$-asymptotically Euclidean initial data set with respect to an asymptotic coordinate chart $\vec{x}\colon M^{3}\setminus \mathcal{B}\to\R^{3}\setminus\overline{B_{R}(0)}$ and decay constant $C_{\mathcal{I}}$, with non-vanishing energy $E \neq 0$. Assume in addition that 
\begin{align*}
\vert K\vert \leq C_{\mathcal{I}} \vert\vec{x}\,\vert ^{-2}
\end{align*}
for all $\vec{x}\in\R^{3}\setminus\overline{B_{R}(0)}$ and that $g$ satisfies the Riemannian $C^2_{\sfrac{3}{2}+\varepsilon}$-Regge--Teitelboim condition. Then the coordinate center $\vec{C}_\text{STCMC}$ of the unique foliation by surfaces of constant spacetime mean curvature is well-defined if and only if the \emph{correction term} 
\begin{align*}
Z^i \definedas \frac{1}{32\pi E} \lim_{r\to\infty}\int_{\mathbb{S}^{2}_{r}} \frac{x^i \left(\sum_{k,l}\pi_{kl}x^k x^l\right)^2}{r^3} \, d\mu^\delta
\end{align*}
limits exist for $i=1,2,3$. In this case, we have 
\begin{align*}
\vec{C}_\text{STCMC} = \vec{C}_\text{B\'OM} +  \vec{Z},
\end{align*}
where $\vec{C}_\text{B\'OM}$ is the Beig--\'O Murchadha center of mass and $\vec{Z}=(Z^1,Z^2,Z^3)$, or equivalently
\begin{align*}
\begin{split}
C^i_\text{STCMC} = \frac{1}{16\pi E}\lim_{r\to \infty} \left[\int_{\mathbb{S}^{2}_{r}}   \left(x^i \sum_{k,l}(\partial_k g_{kl} - \partial_l g_{kk})\frac{x^l}{r} - \sum_k \left(g_{ki}\frac{x^k}{r}-g_{kk}\frac{x^i}{r}\right) \right) d\mu^\delta \right.\\ + \left.\int_{\mathbb{S}^{2}_{r}} \frac{x^i \left(\sum_{k,l}\pi_{kl}x^k x^l\right)^2}{2r^3} \, d\mu^\delta\right], \quad  i=1,2,3.
\end{split}
\end{align*} 
\end{theorem}

An example of an initial data set with $\vec{C}_{\text{STCMC}}\neq\vec{C}_{\text{B\'OM}}$ or in other words with $\vec{Z}\neq0$ will be discussed in Section \ref{secHYCN}. The above formula for $\vec{C}_\text{STCMC}$ allows to compute the STCMC-center of mass of an initial data set explicitly, once an asymptotic chart $\vec{x}$ has been picked. However, as the assumptions of Theorem \ref{thCoordinateExpression} suggest, this formula cannot be expected to always converge. See Conjecture \ref{conj:convergence} and the text above of it for a discussion of when the coordinate expression for $\vec{C}_{\text{STCMC}}$ should converge, without reference to $\vec{C}_{\text{B\'OM}}$ and $\vec{Z}$ and without any Regge--Teitelboim conditions nor additional decay assumptions on $K$.\\

\setcounter{section}{8}
\setcounter{thm}{0}
We get the following theorem on the time evolution of the STCMC-foliation and center of mass. The full covariance of the STCMC-foliation under the Poincar\'e group is discussed in Section \ref{secPoincare}.
\begin{theorem}[Time evolution of STCMC-foliation]
Let $(\R\times M^{3},\mathfrak{g})$ be a smooth, globally hyperbolic Lorentzian spacetime satisfying the Einstein equations with energy momentum tensor $\mathfrak{T}$. Suppose that, outside a set of the form $\R\times \mathcal{K}$, $\mathcal{K}\subset M^{3}$ compact, there is a diffeomorphism $\id_{\R}\times\,\vec{x}\colon \R\times(M^{3}\setminus\mathcal{K})\to \R\times(\R^{3}\setminus\overline{B_{R}(0)})$ which gives rise to asymptotic coordinates $(t,\vec{x}\,)$ on $\R\times(M^{3}\setminus\mathcal{K})$. 

Assume that $\mathcal{I}_{0}=(\{0\}\times M^{3},g,K,\mu,J)\hookrightarrow(\R\times M^{3},\mathfrak{g})$ is a $C^2_{\sfrac{1}{2}+\varepsilon}$-asymptotically Euclidean initial data set with respect to the coordinate chart $\vec{x}$ and with $E\neq0$, and suppose additionally that $K=O_{1}(\vert\vec{x}\,\vert^{-2})$ with constant $C_{\mathcal{I}}$ as $\vert\vec{x}\,\vert\to\infty$. Now consider the $C^{1}$-parametrized family of $C^2_{\sfrac{1}{2}+\varepsilon}$-asymptotically Euclidean initial data sets 
\begin{align*}
\mathcal{I}(t)=(\{t\}\times M^{3},g(t),K(t),\mu(t),J(t))\hookrightarrow(\R\times M^{3},\mathfrak{g})
\end{align*}
with respect to $\vec{x}$ which starts from $\mathcal{I}(0) = \mathcal{I}_{0}$, and which exists for all $t\in(-t_{*},t_{*})$ for some $t_{*}>0$. Assume furthermore that the constants $C_{\mathcal{I}(t)}$ are uniformly bounded on $(-t_{*},t_{*})$, without loss of generality such that $C_{\mathcal{I}(t)}\leq C_{\mathcal{I}_{0}}$.

Assume the foliation $\mathcal{I}(t)$ has initial lapse $N=1+O_2(\vert\vec{x}\,\vert^{-\frac{1}{2}-\varepsilon})$ as $\vert\vec{x}\,\vert\to\infty$ with decay measuring constant denoted by $C_{N}$ and initial shift $X=0$, and suppose furthermore that the initial stress tensor $S$ of $\mathcal{I}_{0}$ satisfies $S = O(\vert\vec{x}\,\vert^{-\frac{5}{2}-\varepsilon})$ as $\vert\vec{x}\,\vert\to\infty$. There is a constant $\overline{t}>0$, depending only on $\varepsilon$, $C_{\mathcal{I}_{0}}$, $C_{N}$, and $E(0)$ such that the following holds: If the initial data set $\mathcal{I}_{0}$ has well-defined STCMC-center of mass $\vec{C}_{\text{STCMC}}\,(0)$ then the STCMC-center of mass $\vec{C}_{\text{STCMC}}\,(t)$ of $\mathcal{I}(t)$ is also well-defined for $\vert t\vert<\overline{t}$. Furthermore, the initial velocity at $t=0$ is given by
\begin{align*}
\left.\frac{d}{dt}\right\vert_{t=0} \vec{C}_{STCMC} = \frac{\vec{P}}{E}.
\end{align*}
Moreover, we have that $\left.\frac{d}{dt}\right\vert_{t=0} E=0$ and $\left.\frac{d}{dt}\right\vert_{t=0} \vec{P}=\vec{0}$.
\end{theorem}
\vspace{2ex}
\setcounter{section}{3}
\subsection{Strategy of the proofs of Theorems \ref{thExistence} and \ref{thUniqueness}}\label{ssecStrategy}
The underlying structure of the proofs of Theorems \ref{thExistence} and \ref{thUniqueness} presented in Section \ref{secExistence} and several of the lemmas proved in the same section is a method of continuity inspired by Metzger \cite{Metzgerthesis,MetzgerCMC}. Given an initial data set $\mathcal{I}=\IDS$, we will consider the one-parameter family of initial data sets $\mathcal{I}_\tau=\IDStau$, $\tau\in [0,1]$, with $\mu_\tau$ given through the constraint equations \eqref{CE} as
\begin{align}\label{CEtau}
2\mu_\tau&\definedas  \scal - |\tau K|^2 + (\tr (\tau K))^2.
\end{align}
For $\tau=0$, we thus consider the Riemannian manifold $(M^3,g)$ with $2\mu=2\mu_0=\scal$ while for $\tau=1$, we study the original initial data set $\mathcal{I}=\IDS$ with $\mu=\mu_1$. It is straightforward to see that if the original initial data set $\mathcal{I}$ is $C^2_{\sfrac{1}{2} + \varepsilon}$-asymptotically Euclidean with respect to an asymptotical chart $\vec{x}\colon M^3\setminus\mathcal{B}\to\R^3\setminus\overline{B_R(0)}$ then all initial data sets $\mathcal{I}_\tau$ are also $C^2_{\sfrac{1}{2} + \varepsilon}$-asymptotically Euclidean with respect to the same chart and comparable constants. In particular, the Riemannian manifold $(M^3,g)$ is $C^2_{\sfrac{1}{2} + \varepsilon}$-asymptotically Euclidean in this chart. This is what will allows us to drop the explicit mention of the chart in the proofs. Moreover, we note that the energy $E_\tau$ computed for the initial data set $\mathcal{I}_\tau=\IDStau$ does in fact not depend on $\tau$ and can and will thus be called $E$. We globally assume in this paper that $E\neq0$ and we will fix the background Riemannian manifold $(M^3,g)$ once and for all.

For second fundamental form $K=0$, the desired STCMC-foliation coincides with the classical CMC-foliation. From Nerz' work \cite{NerzCMC}, we thus know that the theorems and lemmas we will prove for initial data sets hold in the Riemannian setting under the Riemannian version of our assumptions, see also Remark \ref{remAERiem}. In other words, we know that our claims hold for $\tau=0$ in the method of continuity approach described above. In Section \ref{secExistence}, we will recall Nerz' corresponding theorems in our notation.

As usual, we will appeal to the Implicit Function Theorem in order to show openness of the interval in the method of continuity. Closedness follows from a standard convergence argument.

%% file: on-center.tex
% !TEX root = main.tex
When deforming the foliation by $2$-surfaces of constant mean curvature to the foliation by $2$-surfaces of constant spacetime mean curvature, we need to keep track of how the geometry of the leaves changes. For this, following \cite{MetzgerCMC} and \cite{NerzCMC}, we will now introduce an a priori class of closed, oriented $2$-surfaces having the properties that their ``area radius'', ``coordinate radius'', and ``mean curvature radius'' as defined below are comparable in a certain sense.

In this section, we will not make explicit reference to the asymptotic coordinate chart $\vec{x}\colon M^{3}\setminus \mathcal{B}\to \R^{3}\setminus \overline{B_{R}(0)}$ in most estimates, however the asymptotic coordinates $\vec{x}$ will be used in order to compute the coordinate radius and the center of mass of a given $2$-surface~$\Sigma\hookrightarrow M^{3}$ (or ``$\,\Sigma\hookrightarrow\mathcal{I}\,$''). We will always and mostly tacitly assume that $\Sigma\hookrightarrow M^{3}\setminus \mathcal{B}$ so that it lies in the domain of the asymptotic coordinate chart.
												
\begin{definition}\label{defConcentric}
Let $(M^{3},g)$ be a $C^2_{\sfrac{1}{2}+\varepsilon}$-asymptotically flat manifold with asymptotic coordinate chart $\vec{x}\colon M^{3}\setminus \mathcal{B}\to \R^{3}\setminus \overline{B_{R}(0)}$. Given any closed, oriented $2$-surface $\Sigma \hookrightarrow M^{3}\setminus\mathcal{B}$, we define its \emph{area radius} $r=r(\Sigma)$ and \emph{(Euclidean) coordinate center} $\vec{z} = \vec{z}\,(\Sigma)$, $\vec{z}=(z^1,z^2,z^3)$, by
\begin{align}
r \definedas \sqrt{\frac{\vert \Sigma\vert_{g}}{4\pi}}, \quad \text{ and } \quad z^i\definedas \frac{1}{\vert \Sigma\vert_{\delta}} \int_\Sigma x^i \, d\mu^\delta, \qquad i=1,2,3,
\end{align}
respectively, where $d\mu^{\delta}$ denotes the area element on $\Sigma$ induced by the Euclidean metric~$\delta$. Given constants $a\in [0,1)$, $b \geq 0$, and $\eta\in (0,1]$, we say that $\Sigma$ belongs to the a priori class of $(M^{3},g)$-\emph{asymptotically centered} surfaces,
\begin{align}
\Sigma\in{\mathcal{A}(a,b,\eta)},
\end{align}
if its area radius $r$, coordinate center $\vec{z}$, \emph{coordinate radius} $\vert\vec{x}\,\vert$, and mean curvature $H$ satisfy the following estimates 
\begin{align}\label{eqDefConcentric}
\vert \vec{z}\,\vert  \leq a r + b r^{1-\eta}, \quad r^{2+\eta} \leq \min_\Sigma \vert \vec{x}\,\vert ^{\frac{5}{2} + \varepsilon}, \quad \int_\Sigma H^2 \, d\mu - 16 \pi (1-\gamma) \leq \frac{b}{r^\eta},
\end{align}
where $\gamma$ denotes the genus of $\Sigma$. 
\end{definition}

\begin{remark}
We will use the same a priori classes in the context of asymptotically Euclidean initial data sets $\mathcal{I}=\IDS$, where the definition of $\mathcal{A}(a,b,\eta)$ only depends on the Riemannian manifold part $(M^{3},g)$. This will later be important when we consider families of initial data sets of the form $\mathcal{I}_{\tau}=(M^{3},g,\tau K, \mu_{\tau}, \tau J)$, see Section \ref{secExistence} and \eqref{CEtau}.
\end{remark}
\medskip

\begin{remark}
Note that for $r>1$, $0\leq a \leq \overline{a}<1$, $0 \leq b \leq \overline{b}$, and $0< \overline{\eta} \leq \eta \leq 1$, we have  $\mathcal{A}(a,b,\eta) \subseteq \mathcal{A}(\overline{a},\overline{b},\overline{\eta})$.
\end{remark}

\begin{example}\label{exCMC}
Let $(M^{3},g)$ be a $C^2_{\sfrac{1}{2} + \varepsilon}$-asymptotically Euclidean manifold with non-vanishing energy~$E\neq0$. Then the unique leaves of the constant mean curvature foliation $\{\Sigma^\sigma \}_{\sigma > \sigma_0}$ constructed in \cite{NerzCMC} are asymptotically centered in this sense. More specifically, there are constants $b > 0$ and $\sigma_0>0$ depending only on $C_{\mathcal{I}}$ such that $\Sigma^\sigma \in \mathcal{A}(a=0,b,\eta=\varepsilon)$ for $\sigma > \sigma_0$. See \cite[Section 5]{NerzCMC} for details.
\end{example}

\begin{proposition}\label{propRegularity}
Suppose that $\overline{a} \in [0,1)$, $\overline{b} \geq 0$, $\overline{\eta} \in (0,1]$, and assume that $0\leq a \leq \overline{a}$, $0 \leq b \leq \overline{b}$, and $\overline{\eta} \leq \eta \leq 1$. Let $\mathcal{I}=\IDS$ be a $C^2_{\sfrac{1}{2}+\varepsilon}$-asymptotically Euclidean initial data set. Then there exist constants $\overline{\sigma}$ and $C$ depending only on $\varepsilon$, $\overline{a}$, $\overline{b}$, $\overline{\eta}$,  and $C_{\mathcal{I}}$ such that the following a priori conclusions hold for any closed, oriented $2$-surface $\Sigma\hookrightarrow\mathcal{I}$ with $\Sigma\in \mathcal{A}(a,b,\eta)$: Suppose that $\Sigma$ has constant spacetime mean curvature $\mathcal{H} \equiv \sfrac{2}{\sigma}$ in $\mathcal{I}$ for some $\sigma > \overline{\sigma}$. Then $\Sigma$ is a topological sphere and the tracefree part $\mathring{A}$ of its second fundamental form satisfies
\begin{align}\label{eq2FFtraceless}
r^{-1} \Vert \mathring{A}\Vert _{W^{1,2}(\Sigma)} + \Vert \mathring{A}\Vert _{L^\infty(\Sigma)} \leq C r^{-\frac{3}{2}-\varepsilon}.
\end{align}  
Furthermore, there exists a function $f\colon \mathbb{S}^2_r (\vec{z}\,) \to \mathbb{R} $ such that  $\Sigma$ is the graph of $f$ and
\begin{align}\label{eqGraphFunction}
\Vert f\Vert _{W^{2,\infty}(\mathbb{S}^2_r (\vec{z}\,))}\leq C r^{\frac{1}{2}-\varepsilon},
\end{align}
as well as a conformal parametrization $\psi\colon \mathbb{S}^2_r (\vec{z}\,) \to \Sigma$ which satisfies
\begin{align}\label{eqConfPar}
\Vert \psi-\id\Vert _{W^{2,2}(\mathbb{S}^2_r (\vec{z}\,))}\leq C r^{\frac{3}{2}-\varepsilon},
\end{align}
where $\id$ denotes the trivial embedding $(\mathbb{S}^2 _r (\vec{z}\,),g^{\mathbb{S}^2_r(\vec{z}\,)} )\hookrightarrow (\mathbb{R}^3,\delta)$. The conformal factor $u\colon \mathbb{S}^{2}_{r}(\vec{z}\,)\to\R^{+}$ such that $\psi^* g^\Sigma = u^2 g^{\mathbb{S}^2_r(\vec{z}\,)}$ satisfies
\begin{align}\label{eqConfFact}
\Vert u^2-1\Vert _{W^{2,2}(\mathbb{S}^2_r (\vec{z}\,))}\leq C r^{\frac{1}{2}-\varepsilon}.
\end{align}
Finally, the \emph{Euclidean distance to the coordinate origin $\vert\vec{x}\,\vert $ (on $\Sigma$)}, the area radius $r$, and the \emph{spacetime mean curvature radius $\sigma$} are comparable in the following sense:
\begin{align}\label{eqDistArea}
(1 - \overline{a}) r - C r^{\max\{\frac{1}{2}-\varepsilon, 1-\overline{\eta}\}} \leq \vert\vec{x}\,\vert &\leq (1 + \overline{a}) r + C r^{\max\{\frac{1}{2}-\varepsilon, 1-\overline{\eta}\}},\\\label{eqMCArea}
\vert r - \sigma \vert  &\leq C r^{\frac{1}{2}-\varepsilon}.
\end{align}
\end{proposition}
\vspace{1ex}
\begin{remark}\label{remInsuff}
The conclusions of this theorem are mostly the same as those in \cite[Proposition 4.4]{NerzCMC}, only for STCMC- rather than CMC-surfaces. However, we cannot directly refer to this result because, roughly speaking, it assumes that the mean curvature $H$ of $\Sigma$ falls off like
$H -  \tfrac{2}{\sigma}= O(r^{-\frac{3}{2}-\varepsilon})$, whereas  the relation
\begin{align*}
H^2 = \mathcal{H}^2 + P^2 =\left(\tfrac{2}{\sigma}\right)^2  + P^2,
\end{align*}
recalling $P=\tr_{\Sigma}K$, and the definition of the a priori class $\mathcal{A}(a,b,\eta)$ --- which coincides with that in~\cite{NerzCMC} ---, only ensure via the second inequality in \eqref{eqDefConcentric} that
\begin{align}\label{eqCompMC2}
\vert H - \tfrac{2}{\sigma}\vert  = \left\vert \sqrt{\left(\tfrac{2}{\sigma}\right)^2  + P^2} - \tfrac{2}{\sigma} \right\vert  \leq \vert P\vert  \leq C (\min_\Sigma \vert\vec{x}\,\vert )^{-\frac{3}{2}-\varepsilon}  \leq C r^{-1 - \frac{\eta}{2}}
\end{align}
which does not a priori give us $H -  \tfrac{2}{\sigma}= O(r^{-\frac{3}{2}-\varepsilon})$. We will thus need to extend the result and its proof to our setting.
\end{remark}

\begin{proof}
Within this proof, $C$ will always be a generic constant depending only on $\overline{\sigma}$, $\overline{a}$, $\overline{b}$, $\overline{\eta}$,  and $C_{\mathcal{I}}$. With Remark \ref{remInsuff} in mind, we need to improve the estimate in \eqref{eqCompMC2}. For this purpose, we first note that by the definition of $(M^{3},g)$ being $C^{2}_{\sfrac{1}{2}+\varepsilon}$-asymptotically Euclidean and by the second inequality in \eqref{eqDefConcentric}, we have
\begin{align*}
\vert \scal\vert  \leq C\vert\vec{x}\,\vert ^{-\frac{5}{2}-\varepsilon} \leq C r^{-2-\eta},
\end{align*}
which implies $\Vert \scal\Vert _{L^1(\Sigma)}\leq Cr^{-\eta}$. Similarly, with $\nu$ denoting the unit normal of $\Sigma$ in $(M^{3},g)$, we get $\Vert \ric(\nu,\nu)\Vert _{L^1(\Sigma)} \leq C r^{-\eta}$. Combining this with the last inequality of \eqref{eqDefConcentric}, we conclude by the Gauss equation and the Gauss--Bonnet Theorem that 
\begin{align*}
\begin{split}
\int_{\Sigma}\vert \mathring{A}\vert ^2 \, d\mu 
& = \int_{\Sigma} \left(\scal - \scal_{\Sigma} - 2 \ric(\nu,\nu) + \tfrac{1}{2} H^2\right) \, d\mu \\
& = \tfrac{1}{2}\int_{\Sigma}  H^2 \, d\mu - 4\pi (2-2\gamma) + O(r^{-\eta}) \\
& = O(r^{-\eta}),
\end{split}
\end{align*}
hence $\Vert \mathring{A}\Vert _{L^2(\Sigma)}\leq C r^{-\frac{\eta}{2}}$. Then by Lemma \ref{lemComparison} we also have $\Vert \mathring{A}^\delta\Vert _{L^2(\Sigma,\delta^\Sigma)}\leq C r^{-\frac{\eta}{2}}$, where $\delta^\Sigma$ is the induced metric of the embedding $(\Sigma,\delta^{\Sigma}) \hookrightarrow (\mathbb{R}^3,\delta)$. We are now in a position to apply the result of De Lellis and M\"uller \cite[Theorem 1.1]{DeLellisMueller} (see also \cite[Section 2.3]{MetzgerCMC} where this result is reformulated in a scale invariant form) to conclude that $\Sigma$ is a topological sphere, with a conformal parametrization $\psi\colon \mathbb{S}^2_r(\vec{z}\,) \to \Sigma$ and the conformal factor $u\colon\mathbb{S}^2_r(\vec{z}\,) \to\R^{+}$ such that $\psi^* \delta^\Sigma = u^2 g^{\mathbb{S}^2_r(\vec{z}\,)}$ satisfying
  \begin{subequations}\label{eqConformal}
\begin{align}
\Vert \psi-\id\Vert _{W^{2,2}(\mathbb{S}^2_r (\vec{z}\,))} &\leq C r^2 \Vert \mathring{A}^\delta\Vert _{L^2(\Sigma,\delta^\Sigma)} \leq C r^{2-\frac{\eta}{2}}, \label{eqConfPar1}\\
\Vert u^2-1\Vert _{W^{2,2}(\mathbb{S}^2_r(\vec{z}\,))} & \leq C r \Vert \mathring{A}^\delta\Vert _{L^2(\Sigma,\delta^\Sigma)} \leq Cr^{1-\frac{\eta}{2}}. \label{eqConfFact1}
\end{align}
\end{subequations}

In order to prove that $\sigma$ and $r$ are comparable, we estimate
\begin{align*}
\begin{split}
2 \sqrt{\pi} r \vert \tfrac{1}{r} - \tfrac{1}{\sigma}\vert 
&=\tfrac{1}{\sqrt{2}} \Vert (\tfrac{1}{r}-\tfrac{1}{\sigma})g^\Sigma \Vert _{L^2(\Sigma)}\\
&\leq \tfrac{1}{\sqrt{2}}\left(\Vert  \tfrac{1}{r} \delta^\Sigma -\tfrac{1}{\sigma}g^\Sigma \Vert _{L^2(\Sigma)}+\Vert \tfrac{1}{r}(\delta^\Sigma-g^\Sigma)\Vert _{L^2(\Sigma)}\right)\\
& \leq\tfrac{1}{\sqrt{2}}\left(\Vert \tfrac{1}{r}\delta^\Sigma-A^\delta\Vert _{L^2(\Sigma)}+\Vert A^\delta-A\Vert _{L^2(\Sigma)}+\Vert A-\tfrac{1}{\sigma}g^\Sigma\Vert _{L^2(\Sigma)}\right) 
+ O(r^{-\frac{\eta}{2}}),  
\end{split}
\end{align*}
where we have used \eqref{AE} and the second inequality in \eqref{eqDefConcentric} in the last line. Here, we have 
\begin{align*}
\Vert \tfrac{1}{r} \delta^\Sigma - A^\delta \Vert _{L^2(\Sigma)}\leq C \Vert \mathring{A}^\delta\Vert _{L^2(\Sigma,\delta^\Sigma)} = O(r^{-\frac{\eta}{2}})
\end{align*}  
by \cite[Theorem 1.1]{DeLellisMueller} (see also (2.4) in \cite{MetzgerCMC}), and 
\begin{align*}
\begin{split}
\Vert A^\delta - A\Vert _{L^2(\Sigma)} 
&\leq C r^{-\frac{\eta}{2}}(1+ \Vert A\Vert _{L^2(\Sigma)}) \\
& = C r^{-\frac{\eta}{2}} \left(1 + \sqrt{\Vert \mathring{A}\Vert ^2_{L^2(\Sigma)} + 
\tfrac{1}{2} \Vert  H\Vert ^2_{L^2(\Sigma)} \,} \right)\\
& \leq C r^{-\frac{\eta}{2}},
\end{split}
\end{align*}
by Lemma \ref{lemComparison} combined with \eqref{eqDefConcentric}, 
and
\begin{align*}
\begin{split}
\Vert  A - \tfrac{1}{\sigma} g^\Sigma \Vert _{L^2(\Sigma)} 
& \leq \Vert A - \tfrac{H}{2} g^\Sigma\Vert _{L^2(\Sigma)} + \Vert (\tfrac{1}{\sigma} - \tfrac{H}{2}) g^\Sigma\Vert _{L^2(\Sigma)} \\
& \leq \Vert \mathring{A}\Vert _{L^2(\Sigma)} + \Vert H - \tfrac{2}{\sigma}\Vert _{L^2(\Sigma)} \\
& = O(r^{-\frac{\eta}{2}})
\end{split}
\end{align*}
by \eqref{eqCompMC2}.

Summing up, we conclude that 
\begin{align}\label{eqMCArea1}
\vert \tfrac{r}{\sigma}-1\vert  \leq C r^{-\frac{\eta}{2}}.
\end{align}  

To prove that $\vert\vec{x}\,\vert $ and $r$ are comparable, note that by the first inequality in \eqref{eqDefConcentric} and because $0<1-\overline{a}\leq 1-a$, the coordinate origin $\vec{0}$ lies inside $\mathbb{S}^2_r (\vec{z}\,)$ for $r>\left(\frac{\overline{b}}{1-\overline{a}}\right)^{\frac{1}{\overline{\eta}}}$. For such large radii, we thus elementarily find
\begin{eqnarray*}
\min_{\mathbb{S}^{2}_r(\vec{z}\,)}\vert\vec{x}\,\vert   & \geq & (1 - \overline{a}) r - \overline{b} r^{1-\overline{\eta}}, \\
\max_{\mathbb{S}^{2}_r(\vec{z}\,)}\vert\vec{x}\,\vert   & \leq & (1 + \overline{a}) r + \overline{b} r^{1-\overline{\eta}},
\end{eqnarray*}
with the help of the first inequality in \eqref{eqDefConcentric}.
By \eqref{eqConfPar1} and the Sobolev Inequality in the form of Lemma~\ref{lemSobolevEmbed}, it follows that $\vert \psi-\id\vert \leq C \vert\vec{x}\,\vert ^{1-\frac{\eta}{2}}$. Combining this with the above inequalities, we conclude that, on $\Sigma$, we have
\begin{align}\label{eqDistArea1}
(1 - \overline{a}) r - C r^{1-\frac{\overline{\eta}}{2}} \leq \vert\vec{x}\,\vert  \leq (1 + \overline{a}) r + C r^{1-\frac{\overline{\eta}}{2}},
\end{align}
provided that the area radius $r$ of $\Sigma$ is sufficiently large. Via \eqref{eqMCArea1}, we can alternatively state that \eqref{eqDistArea1} holds if the spacetime mean curvature radius $\sigma$ satsifies $\sigma>\overline{\sigma}$ for a suitably large $\overline{\sigma}$ only depending on $\varepsilon$, $\overline{a}$, $\overline{b}$, $\overline{\eta}$, and $C_{\mathcal{I}}$.\\

\paragraph*{\emph{Bootstrapping}.} With these new bounds \eqref{eqMCArea1} and \eqref{eqDistArea1} at hand, we can apply \cite[Proposition 4.1]{NerzCMC} with $\kappa$ chosen as $\tfrac{3}{2}+\varepsilon>1$, $\eta$ chosen as our $\tfrac{\eta}{2}>0$, and $c_{1},c_{2}$ chosen as our generic constant $C$. As all the estimates going into verifying the assumptions from \cite[Proposition 4.1]{NerzCMC} hold pointwise in our case, the assumptions are indeed satisfied for any $p>2$. Note that the existence of the uniform Sobolev Inequality assumed in \cite[Proposition 4.1]{NerzCMC} is well-established in our setting, and goes back to \cite[Proposition 5.4]{HY} which holds for surfaces in asymptotically Euclidean manifolds with general asymptotics as described in Section \ref{secPrelim}. Again via \eqref{eqMCArea1}, this gives us \eqref{eq2FFtraceless} for $\sigma>\overline{\sigma}$, with suitably enlarged $\overline{\sigma}$ only depending on $\varepsilon$, $\overline{a}$, $\overline{b}$, $\overline{\eta}$, and $C_{\mathcal{I}}$.

As a consequence of \eqref{eq2FFtraceless}, the estimates \eqref{eqConfPar1}-\eqref{eqConfFact1} improve, and we get \eqref{eqConfPar} and \eqref{eqConfFact}. Similarly, repeating the above argument that we used to derive \eqref{eqMCArea1} and \eqref{eqDistArea1}, we obtain the improved radius comparison \eqref{eqDistArea} and \eqref{eqMCArea}.  

Finally, now that we have a pointwise bound on the tracefree part of the second fundamental form $\mathring{A}$ accompanying the pointwise estimate \eqref{eqCompMC2} for the mean curvature~$H$, it follows that $\Sigma$ is the graph of a function $f\in W^{2,\infty}(\mathbb{S}^2_r (\vec{z}\,))$ such that \eqref{eqGraphFunction} holds for $\sigma>\overline{\sigma}$, for again suitably enlarged $\overline{\sigma}$ only depending on $\varepsilon$, $\overline{a}$, $\overline{b}$, $\overline{\eta}$, and $C_{\mathcal{I}}$, see e.g. \cite[Corollary E.1]{NerzCMC}, which adapts \cite[Theorem 1.1]{DeLellisMueller} to our setting. To be more precise, \cite[Corollary E.1]{NerzCMC} is only not stated invariantly under scaling but with $\vert\Sigma\vert=4\pi$, but it is straightforward to adapt it to include the area radius for our purposes. This finishes the proof of Proposition \ref{propRegularity}.
\end{proof}  

%% file: linearization.tex
% !TEX root = main.tex
In this section, we will introduce the spacetime mean curvature map $\mathcal{H}$ in a given initial data set $\mathcal{I}=(M^3,g,K,\mu,J)$. We will analyze its properties in a neighborhood of a given $2$-surface $\Sigma$ having constant spacetime mean curvature. We will show that the linearization of the map $\mathcal{H}$ is invertible when the linearization is computed with respect to normal variations within the given initial data set $\mathcal{I}$. This will later be used to ensure that the CMC-foliation of $(M^3,g)$ constructed in \cite{NerzCMC} can be pushed via a method of continuity to an STCMC-foliation of $\mathcal{I}$.

Throughout this section, we will assume that $\mathcal{I}=(M^3,g,K,\mu, J)$ is a $C^{2}_{\sfrac{1}{2}+\varepsilon}$-asympto\-tically Euclidean initial data set with non-vanishing energy $E\neq 0$ and with fixed asymptotic coordinates~$\vec{x}$. Furthermore, it will be assumed that  $\Sigma$ is a fixed $2$-surface of constant spacetime mean curvature $\mathcal{H}(\Sigma) \equiv \sfrac{2}{\sigma}$ which has sufficiently large mean curvature radius $\sigma$ and which for some fixed $a \in [0,1)$, $b \geq 0$, and $\eta \in (0,1]$ belongs to the a priori class $\mathcal{A}(a,b,\eta)$, see Definition \ref{defConcentric}. In this setting, we know from Proposition \ref{propRegularity} that $\Sigma$ is a topological sphere, and that its coordinate radius, area radius, and mean curvature radius are comparable as stated in \eqref{eqDistArea}, \eqref{eqMCArea}. This in particular implies that
\begin{align}\label{eqFallOff}
P = \tr_\Sigma K = O_{1}(\sigma^{-\frac{3}{2} - \varepsilon}), \; H = \sqrt{\mathcal{H}^2 + P^2} = \tfrac{2}{\sigma} + O_{1}(\sigma^{-2-2\varepsilon}), \;\tfrac{P}{H}=O_{1}(\sigma^{-\frac{1}{2}-\varepsilon}),
\end{align} 
for $\sigma>\overline{\sigma}$, where $\overline{\sigma}$ and the constants hidden in the $O$-notation only depend on $\varepsilon$, $a$, $b$, $\eta$, $C_{\mathcal{I}}$.

\subsection{Stability operators associated with prescribed (spacetime) mean curvature surfaces}\label{subsecStab}
In a neighborhood of $\Sigma$, we introduce normal geodesic coordinates $y\colon \Sigma \times (-\xi,\xi) \to M^{3}$ for some $\xi>0$, such that $y(\cdot, 0) = \id_\Sigma$, and  $\tfrac{\partial y}{\partial t} = \nu^{\Sigma_t}$, with $\nu^{\Sigma_{t}}$ being the outward unit normal to $\Sigma_t \definedas y(\Sigma, t)$. For a function $f\in C^\infty(\Sigma)$ with $\vert f\vert  < \xi$, we define the \emph{graph of $f$ over $\Sigma$} as 
\begin{align}
\graph f =  \{y(q,f(q)) : q \in \Sigma\}.
\end{align}
Then, slightly abusing notation, let $\mathcal{H}\colon  C^\infty(\Sigma) \to  C^\infty(\Sigma)$ be the operator which assigns to a function $f$ the spacetime mean curvature $\mathcal{H} (f)$ of $\graph f$ (with respect to the fixed initial data set~$\mathcal{I}$). The linearization of this map $\mathcal{H}$ is computed in the following lemma.

\begin{lemma}\label{lemLinearizationCMC}
 Let $\Sigma \hookrightarrow M^{3}$ be a closed, oriented $2$-surface. Let $\mathcal{V}\colon\Sigma \times (-\xi,\xi) \to M^{3}$ be the normal variation with $\mathcal{V}(\cdot,0)=\id_\Sigma$ and $\left.\frac{\partial \mathcal{V}}{\partial t}\right\vert _{t=0} = f\nu$ for $f\in C^\infty (\Sigma)$. Then the linearization $L^{\mathcal{H}}$ of the spacetime mean curvature map at $\Sigma$ is given by
 \begin{align*}
 & L^\mathcal{H} f \definedas \left.\frac{\partial \mathcal{H}(\mathcal{V}(\cdot, t))}{\partial t} \right\vert _{t=0} \\[0.5ex]
 &\;\; = \frac{H \left( -\Delta^\Sigma f - (\vert A\vert ^2 + \ric (\nu, \nu)) f \right)- P \left((\nabla_\nu \tr K - \nabla_\nu K(\nu, \nu)) f - 2 K (\nabla^\Sigma f,\nu)\right)}{\sqrt{H^2 - P^2}},
 \end{align*}
 where $\triangle^{\Sigma}$, $\nabla^{\Sigma}$ denote the Laplacian and covariant gradient on $(\Sigma,g^\Sigma)$, respectively.
\end{lemma}
\begin{proof}
 This follows from the definition of spacetime mean curvature $\mathcal{H} = \sqrt{H^2-P^2}$ and the well-known formulas for $\left.\frac{\partial H (\mathcal{V}(\cdot, t))}{\partial t}\right\vert _{t=0}$ and $\left.\frac{\partial P  (\mathcal{V}(\cdot, t))}{\partial t}\right\vert _{t=0}$,
see Metzger \cite[Lemma 5.1]{MetzgerCMC}.
\end{proof}

The map $L^\mathcal{H}$ naturally extends to a bounded mapping $L^\mathcal{H}\colon  W^{2,2}(\Sigma) \to L^{2} (\Sigma)$. In Section \ref{ssecInvertibility}, we will prove that this mapping has a bounded inverse, for which it is convenient to rewrite the above expression for $L^\mathcal{H}$ in the form
\begin{align*}
L^\mathcal{H} f = \frac{\mathcal{L} f}{\sqrt{1 - \left(\frac{P}{H}\right)^2}} ,
\end{align*}
where 
\begin{align}\label{eqLinearization}
\begin{split}
\mathcal{L} f\definedas & - \Delta^\Sigma f - (\vert A\vert ^2 + \ric (\nu, \nu)) f \\
&- \tfrac{P}{H} \left((\nabla_\nu \tr K - \nabla_\nu K(\nu, \nu)) f - 2 K (\nabla^\Sigma f,\nu) \right).
\end{split}
\end{align}
Since the denominator is clearly bounded and bounded away from zero by our assumptions on $\Sigma$, the (bounded) invertibility of $L^\mathcal{H}\colon W^{2,2}(\Sigma) \to L^{2} (\Sigma)$ will follow once we show that $\mathcal{L}\colon  W^{2,2}(\Sigma) \to L^{2} (\Sigma)$ is invertible with bounded inverse. 

\begin{remark}\label{PseudoStab}
Recall that the $H\pm P$-stability operator $L^{H\pm P} $ of the map $H\pm P$ (surfaces of constant expansion or null mean curvature) is given by
\begin{align*}
L^{H\pm P} f & = \left.\frac{\partial (H \pm P)(\mathcal{V}(\cdot, t))}{\partial t}\right\vert _{t=0} \\
 & = - \Delta^\Sigma f - (\vert A\vert ^2 + \ric (\nu, \nu)) f \pm \left((\nabla_\nu \tr K - \nabla_\nu K(\nu, \nu)) f - 2 K (\nabla^\Sigma f,\nu) \right),
\end{align*}
see \cite{MetzgerCMC}. As it turns out, the analytic properties of $L^{H\pm P}$ imply that constant expansion foliations do not provide an adequate notion of center of mass, in contrast to the STCMC-foliation studied here. The main difference is that the contribution of the second fundamental form $K$ in the $H\pm P$-stability operator is large, while it is rescaled by a factor $\sfrac{P}{H}$ in the STCMC-stability operator. The largeness of the contribution of $K$ in the $H\pm P$-stability operator will cause the geometric centers of the surfaces of the foliation to drift away in the direction of the linear momentum $\vec{P}$  in general, see Metzger \cite[Section 7]{MetzgerCMC}. This can only be avoided by imposing very fast fall-off conditions on $K$ to ensure that $\vec{P}=0$. Furthermore, a certain smallness assumption on $K$ is also directly required to ensure the invertibility of $L^{H\pm P}$, and hence the existence of the constant expansion foliation, see~\cite[Theorem 3.1]{NerzCE}. 

As a consequence of the factor $\sfrac{P}{H}$ in the STCMC-stability operator, no smallness assumption on $K$ will be needed to ensure the existence of the foliation by surfaces of constant spacetime mean curvature. Furthermore, we will see that the leaves of this foliation do not translate as their spacetime mean curvature approaches zero, provided that the standard asymptotic symmetry conditions are imposed.
\end{remark}

As we will see, the operator $\mathcal{L}$, and consequently the operator $L^{\mathcal{H}}$, is in many respects similar to the standard (CMC-)stability operator of $\Sigma$, namely to
\begin{align*}
L^H f = \left.\frac{\partial H (\mathcal{V}(\cdot, t))}{\partial t}\right\vert _{t=0} = - \Delta^\Sigma f - (\vert A\vert ^2 + \ric (\nu, \nu)) f. 
\end{align*} 
This operator has been intensively studied, see e.g.~\cite{DoCarmoBarbosa}.

\subsection{Eigenvalues and eigenfunctions of $-\Delta^\Sigma$}\label{secEigenLaplace}
In preparation for proving the invertibility of the operator $\mathcal{L}$, we summarize the spectral properties of the operator $-\Delta^\Sigma$, the Laplacian calculated with respect to the metric $g^\Sigma$ induced by $(\Sigma,g^\Sigma) \hookrightarrow (M^3,g)$. For this, let us first consider the operator $-\Delta^{\mathbb{S}^2_r}$, the Laplacian calculated with respect to the standard round metric $\delta^{\mathbb{S}^2_r}$ on $\mathbb{S}^2_r = \mathbb{S}^2_r(\vec{0}\,) \hookrightarrow (\bR^3,\delta)$. The eigenvalues of $-\Delta^{\mathbb{S}^2_r}$ are $\sfrac{l(l+1)}{r^2}$ for $l\geq0$, and the eigenspace corresponding to $\sfrac{l(l+1)}{r^2}$ is given by the space of homogeneous harmonic polynomials of degree $l$ restricted to $\mathbb{S}^2_r$, see e.g.~\cite[Chapter II.4]{ChavelEigenvalues}. In particular, the first non-zero eigenvalue of $-\Delta^{\mathbb{S}^2_r}$ is $\sfrac{2}{r^2}$, the corresponding eigenspace is spanned by the restrictions to $\mathbb{S}^2_r$ of the coordinate functions $x^1$, $x^2$, $x^3$ on $\mathbb{R}^{3}$. In the following, we enumerate the eigenvalues of $-\Delta^{\mathbb{S}^2_r}$ counting their multiplicity by
\begin{align}
 0=\lambda_{0}^{\delta}<\lambda_{i}^{\delta}\leq\lambda_{i+1}^{\delta}, \qquad i=1,2,\ldots,
\end{align}
and we denote the associated complete $L^2(\mathbb{S}^{2}_{r})$-orthonormal system of eigenfunctions by $\lbrace f_{i}^{\delta}\rbrace_{i=0}^{\infty}$. Without loss of generality, we may assume that the chosen enumeration is such that
\begin{align}
f_{i}^{\delta}=\sqrt{\frac{3}{4\pi r^{4}}}\,x^{i} \quad\text{ for }i=1,2,3.
\end{align}
Note that the tracefree part of the Hessian of each of these functions vanishes, and that we have 
\begin{align*}
\langle \nabla^{\mathbb{S}^2_r} f^\delta_i, \nabla^{\mathbb{S}^2_r} f^\delta_j \rangle - \frac{3\delta_{ij}}{4\pi r^4} + \frac{f^\delta_i f^\delta_j}{r^2}=0\quad\text{ for }i=1,2,3,
\end{align*}
where $\nabla^{\mathbb{S}^{2}_{r}}$ denotes the gradient with respect to $\delta^{\mathbb{S}^2_r}$. 

In order to describe the eigenvalues and eigenfunctions of the operator $-\Delta^\Sigma$, note that by Proposition  \ref{propRegularity} there is a vector $\vec{z}\in\R^{3}$ and a conformal parametrization $\psi\colon \mathbb{S}^2_r(\vec{z}\,) \to \Sigma$ such that
\begin{align*}
\Vert \psi^* g^\Sigma - \delta^{\mathbb{S}^2_r(\vec{z}\,)}\Vert _{W^{2,2}(\mathbb{S}^2_r(\vec{z}\,))} \leq C r^{\frac{1}{2}-\varepsilon},
\end{align*}
where $r$ is the area radius of $(\Sigma,g^\Sigma)$. As all spheres of radius $r$ in Euclidean space are isometric, we can easily ``translate'' $\psi$ to a conformal parametrization $\overline{\psi}\colon  \mathbb{S}^2_r \to \Sigma$
 such that
\begin{align}\label{eqConfPar2}
\Vert \overline{\psi}^{\,*} g^\Sigma - \delta^{\mathbb{S}^2_r}\Vert _{W^{2,2}(\mathbb{S}^2_r)} \leq C r^{\frac{1}{2}-\varepsilon},
\end{align}
where $r$ still denotes the area radius of $(\Sigma,g^\Sigma)$.

We will now describe a complete orthonormal system in $L^2(\Sigma)$ consisting of the eigenfunctions $\{f_i\}_{i=0}^\infty$ such that $-\Delta^\Sigma f_i = \lambda_i f_i$, with $0=\lambda_{0}< \lambda_i \leq \lambda_{i+1}$, $i=1,2,\ldots$, again counted with multiplicity. The eigenfunctions $f_i$ will be chosen so that $\overline{\psi}^{\,*}f_i$ is asymptotic to $f_i^\delta$ for each $i=1,2,\dots$. For simplicity of notation, in what follows we will identify $f_i\colon\Sigma \to \mathbb{R}$ with its pullback $\overline{\psi}^{\,*} f_i$ without further ado. This enumeration and identification will also allow us to prove useful  estimates for the eigenvalues and eigenfunctions of $-\Delta^\Sigma$. 

\begin{lemma}\label{lemEigenLaplace2}
Let $\mathcal{I} = \IDS$ be a $C^{2}_{\sfrac{1}{2}+\varepsilon}$-asymptotically Euclidean initial data set. Let $a \in [0,1)$, $b \geq 0$, $\eta \in (0,1]$, and consider a $2$-surface $\Sigma\hookrightarrow\mathcal{I}$ such that $\Sigma\in\mathcal{A}(a,b,\eta)$ with respect to $\mathcal{I}$. Then there exist constants $C>0$ and $\overline{\sigma}>0$ depending only on $\varepsilon$, $a$, $b$, $\eta$, and $C_{\mathcal{I}}$ such that if $\Sigma$ has constant spacetime mean curvature $\mathcal{H}\equiv \sfrac{2}{\sigma}$ for $\sigma > \overline{\sigma}$, then there is a complete orthonormal system in $L^2(\Sigma)$ consisting of the eigenfunctions $\{f_i\}_{i=0}^\infty$ such that 
\begin{align*}
-\Delta^\Sigma f_i = \lambda_i f_i
\end{align*}
 with $0=\lambda_{0}< \lambda_i \leq \lambda_{i+1}$, $i=1,2,\ldots$, counted with multiplicity,
 
  and such that for $i=1, 2, 3$ the following estimates hold
\begin{align}\label{eqEigenEstimates1a}                                
\vert \lambda_i - \lambda_i^\delta\vert  &\leq C\sigma^{-\frac{5}{2}-\varepsilon},\\[0.5ex]\label{eqEigenEstimates1b} 
\Vert f_i - f_i^\delta\Vert_{W^{2,2}(\Sigma)} &\leq C \sigma^{-\frac{1}{2}-\varepsilon},\\[0.5ex]\label{eqTracelessHessian}
\Vert \accentset{\circ}{\hess}^{\Sigma} f_i\Vert _{L^2(\Sigma)}& \leq C \sigma^{-\frac{5}{2}-\varepsilon},\\[0.5ex]\label{eqEigenEstimates2}
\left\Vert \langle \nabla^\Sigma f_i, \nabla^\Sigma f_j \rangle - \frac{3\delta_{ij}}{\sigma^2\vert \Sigma\vert } + \frac{f_i f_j}{\sigma^2}\right\Vert _{L^{1}(\Sigma)} &\leq \frac{C}{\sigma^{\frac{5}{2}+\varepsilon}}.
\end{align}
Furthermore, $\lambda_{0}=0$ and
\begin{align}\label{eqBigEigenvalues}
\lambda_i > \frac{5}{\sigma^2} \; \text{ for }\; i = 4,5,\ldots. 
\end{align}
\end{lemma}
\begin{proof}
By \eqref{eqConfPar2} and Lemma \ref{lemSobolevEmbed} we have
\begin{align*}
\Vert \overline{\psi}^{\,*} g^\Sigma - \delta^{\mathbb{S}^2_r}\Vert _{L^{\infty}(\mathbb{S}^2_r)} \leq C r^{-\frac{1}{2}-\varepsilon}.
\end{align*}
Applying the Rayleigh Theorem (see e.g.~\cite[Chapter II.5]{ChavelEigenvalues}), we see from the above estimate and \eqref{eqMCArea} that 
\begin{align*}
\lambda_i = \inf_{f} \int_\Sigma \vert \nabla^\Sigma f\vert ^2 \, d\mu = \lambda^\delta_i (1 + O(r^{-\frac{1}{2}-\varepsilon})),  \qquad i=1,2,3,
\end{align*}
where the infimum is taken over all $f\in W^{1,2}(\Sigma)$ with $\int_\Sigma f \, d\mu = 0$ and $\Vert f \Vert_{L^2(\Sigma)}=1$. Of course, the $O$-term constant and the lower bound on $\sigma$ coming from this calculation only depend on $\varepsilon$, $a$, $b$, $\mu$, and $C_{\mathcal{I}}$. We will now construct the respective eigenfunctions $f_i$, for which we will use the fact that these functions are solutions to the equation 
\begin{align}\label{eq:neattrick}
- \Delta^{\mathbb{S}^2_r} (f_i - f_i^\delta) - \lambda^\delta_i (f_i - f_i^\delta) = (\Delta^\Sigma - \Delta^{\mathbb{S}^2_r}) f_i + (\lambda_i-\lambda_i^\delta) f_i,  \qquad i=1,2,3,
\end{align}
where $\lambda^\delta_i = \sfrac{2}{r^2}$ for $i=1,2,3$. Noting that the right hand side of the equation equals $- \Delta^{\mathbb{S}^2_r} f_i -\lambda_i^\delta f_i$, and using integration by parts it is straightforward to check that it is orthogonal in $L^2(\mathbb{S}^2_r)$ to any element in the kernel of the self-adjoint differential operator in the left hand side.
Thus, by the Fredholm Alternative \cite[Appendix~I, Theorem 31]{Besse}, for every $i=1,2,3$ there is a unique solution $f_i - f_i^\delta \in W^{2,2}(\mathbb{S}^2_r)$ orthogonal in $L^2(\mathbb{S}^2_r)$ to the linear space spanned by $ f_i^\delta$, $i=1,2,3$. Note that we may without loss of generality assume that $\Vert f_i\Vert_{L^2(\Sigma)} =1$ so that $\Vert\Delta^\Sigma f_i\Vert_{L^2(\Sigma)} = \lambda_i = O(r^{-2})$. Since $\scal^\Sigma = \tfrac{2}{\sigma^2}+O(\sigma^{-\tfrac{5}{2}-\varepsilon})$ as a consequence of the Gauss equation (see e.g.~\eqref{eqGaussConsequences} below), in view of \eqref{eqMCArea} and Lemma \ref{lemCalderon} we have $\Vert f_i \Vert_{W^{2,2}(\Sigma)} = O(1)$ as $\overline{f}_i = 0$. With the above estimates at hand, it is now straightforward to check that
\begin{align*}
\Vert (\Delta^\Sigma - \Delta^{\mathbb{S}^2_r}) f_i + (\lambda_i-\lambda_i^\delta) f_i\Vert _{L^2(\mathbb{S}^2_r)} \leq C r^{-\frac{5}{2}-\varepsilon},
\end{align*}
so by standard elliptic regularity (see e.g.~\cite[Appendix H, Theorem 27]{Besse}\footnote{We cite this result for the unit sphere and apply rescaling to extend it to spheres of radius $r>0$.}) applied to the operator on the left hand side of \eqref{eq:neattrick}, we have 
\begin{align}\label{eqBasisFallOff}
\Vert f_i - f_i^\delta\Vert _{W^{2,2}(\mathbb{S}^2_r)} \leq C  r^{-\frac{1}{2}-\varepsilon},   \qquad i=1,2,3,
\end{align}
whenever $\sigma>\overline{\sigma}$, for suitably large $C>0$ and $\overline{\sigma}>0$ depending only on $\varepsilon$, $a$, $b$, $\nu$, and $C_{\mathcal{I}}$. This defines the eigenfunctions $f_i$, $i=1,2,3$, up to applying the Gram-Schmidt process in the case of multiple eigenvalues. Note that \eqref{eqBasisFallOff} with \eqref{eqMCArea}, \eqref{eqConfPar2}, and the properties of the functions $f_i^\delta$ implies \eqref{eqTracelessHessian} and \eqref{eqEigenEstimates2}. 

We have $\lambda_0 = 0$ by definition. Using again the Rayleigh Theorem and  \eqref{eqMCArea}, it is also straightforward to check that $\lambda_i > \sfrac{5}{\sigma^2}$ for $i=4,5,\ldots$\,, since the respective eigenvalues of $-\Delta^{\mathbb{S}^2_r}$ satisfy $\lambda_i^\delta \geq \sfrac{6}{r^2}$, whenever $\sigma>\overline{\sigma}$ for suitably large $\overline{\sigma}$ only depending on $\varepsilon$, $a$, $b$, $\mu$, and $C_{\mathcal{I}}$. This concludes the proof.
\end{proof}

We can now give a more detailed characterization of the lowest eigenvalues $\lambda_i$, $i=1,2,3$. More specifically, in the following lemma we show that these eigenvalues are computed in terms of the \emph{Hawking mass} 
\begin{align}\label{eqHawkingMass}
m_\mathcal{H}(\Sigma) \definedas \sqrt{\frac{\vert \Sigma\vert }{16\pi}}\left(1-\frac{1}{16\pi}\int_\Sigma \mathcal{H}^2 \, d\mu\right)
\end{align}
of $\Sigma$ in the initial data set $\mathcal{I}$. We will drop the explicit reference to $\Sigma$ later and will write $m_{\mathcal{H}}$ instead of $m_{\mathcal{H}}(\Sigma)$. This lemma and its proof are very similar to \cite[Lemma 4.5]{NerzCMC}, but rephrased in the spacetime setting. 

\begin{lemma}\label{lemLaplace}                                    
Let $\mathcal{I}=\IDS$ be a $C^{2}_{\sfrac{1}{2}+\varepsilon}$-asymptotically Euclidean initial data set with the energy $E$. Suppose that $a \in [0,1)$, $b \geq 0$, $\eta \in (0,1]$, and that $\Sigma\in\mathcal{A}(a,b,\eta)$ with respect to $\mathcal{I}$ is a surface with Hawking mass $m_{\mathcal{H}}(\Sigma)$. Then there exist constants $C>0$ and $\overline{\sigma}>0$, depending only on $\varepsilon$, $a$, $b$, $\eta$, $\vert E\vert $, and $C_{\mathcal{I}}$ such that if $\Sigma$ has constant spacetime mean curvature $\mathcal{H} \equiv \sfrac{2}{\sigma}$ for $\sigma >  \overline{\sigma}$ then the following estimates hold
\begin{align}\label{eqIntRicci1}
\left\vert \lambda_i - \left( \frac{2}{\sigma^2} + \frac{6m_\mathcal{H}(\Sigma)}{\sigma^3} + \int_\Sigma \ric(\nu, \nu) f_i^2 d\mu\right) \right\vert  \leq \frac{C}{\sigma^{3+\varepsilon}} \quad \text{for} \quad i=1,2,3,
\end{align}
and 
\begin{align}\label{eqIntRicci2}
\left\vert \int_\Sigma \ric(\nu, \nu) f_i f_j d\mu \right\vert  \leq \frac{C}{\sigma^{3+\varepsilon}}\quad\quad \text{for} \quad i\neq j, \quad\text{ with }\quad i,j =1,2,3.
\end{align}
\end{lemma}

\begin{proof}
A polarized version of the standard Bochner formula (see e.g.~Proposition~33(3) in \cite[Chapter 3]{Petersen}) in dimension 2 applied to the eigenfunctions $f_i$ and $f_j$ for $i,j=1,2,3$ reads 
\begin{align*}
\tfrac{1}{2}\Delta^\Sigma &\left\langle \nabla^{\Sigma} f_i, \nabla^\Sigma f_j \right\rangle\\
 &= \,\left\langle \hess^\Sigma f_i,  \hess^\Sigma f_j \right\rangle \\ &  \quad + \tfrac{1}{2} \left( \left\langle\nabla^\Sigma f_i , \nabla^\Sigma \Delta^\Sigma f_j\right\rangle + \left\langle\nabla^\Sigma \Delta^\Sigma f_i, \nabla^\Sigma f_j \right\rangle \right) + \tfrac{1}{2}\scal^\Sigma \left\langle\nabla^\Sigma f_i, \nabla^\Sigma f_j \right\rangle \\
&=\, \left\langle \accentset{\circ}{\hess}^{\Sigma} f_i,  \accentset{\circ}{\hess}^{\Sigma} f_j \right\rangle + \tfrac{1}{2} \lambda_i\lambda_j f_i f_j - \tfrac{1}{2}\left(\lambda_i+\lambda_j-\scal^\Sigma\right) \left\langle\nabla^\Sigma f_i, \nabla^\Sigma f_j \right\rangle.
\end{align*}
\smallskip
Integrating this identity, using the Divergence Theorem on the closed surface $\Sigma$, integrating by parts, and recalling \eqref{eqTracelessHessian}, we obtain
\begin{align}\label{eqLambdaScal}
\left\vert \lambda_i^2 \delta_{ij} - \int_{\Sigma} \scal^\Sigma \left\langle\nabla^\Sigma f_i, \nabla^\Sigma f_j \right\rangle d\mu \right\vert \leq C\sigma^{-5-2\varepsilon}
\end{align}
for some constant $C>0$ and all $\sigma>\overline{\sigma}>0$, with $C$ and $\overline{\sigma}$ only depending on $\varepsilon$, $a$, $b$, $\eta$, and $C_{\mathcal{I}}$. Next, the Gauss equation combined with \eqref{eq2FFtraceless} and \eqref{eqFallOff} gives us
\begin{align}\label{eqGaussConsequences}
\begin{split}
\scal^\Sigma & = \scal - 2\ric(\nu,\nu) - \vert \mathring{A}\vert ^2 + \tfrac{H^2}{2}\\
             & = \scal - 2\ric(\nu,\nu) +\tfrac{2}{\sigma^2} + O(\sigma^{-3-2\varepsilon}),
\end{split}
\end{align}
possibly enlarging $C>0$ and $\overline{\sigma}>0$ without introducing new dependencies.

Substituting this into \eqref{eqLambdaScal} and using \eqref{eqEigenEstimates2}, \eqref{eqEigenEstimates1a} together with the fact that our initial data set is $C^2_{\sfrac{1}{2}+\varepsilon}$-asymptotically Euclidean we conclude that, by partial integration, we get
\begin{align}\label{eqIntermediateEV}
\left\vert \left(\lambda_i^2 - \frac{2}{\sigma^2}\lambda_i\right)\delta_{ij} - \int_\Sigma (\scal - 2 \ric(\nu,\nu))\left(\frac{3\delta_{ij}}{\sigma^2\vert \Sigma\vert }-\frac{f_i f_j}{\sigma^2}\right)\, d\mu\right\vert  \leq \frac{C}{\sigma^{5+2\varepsilon}}.
\end{align}

When $i\neq j$, $i,j =1,2,3$, this gives us \eqref{eqIntRicci2} once we recall that $\scal=O(\sigma^{-3-\varepsilon})$ as a consequence of Definition \ref{defAE} with possibly enlarged $C>0$ and $\overline{\sigma}>0$. In the case $i=j$, $i,j=1,2,3$, one arrives at \eqref{eqIntRicci1} by combining \eqref{eqIntermediateEV}, \eqref{eqMCArea}, \eqref{eqEigenEstimates1a}, and the fact that our initial data set is $C^2_{\sfrac{1}{2}+\varepsilon}$-asymptotically Euclidean, as well as using
\begin{align}\label{eqIntermediateEV2}\nonumber
& \left\vert \frac{r}{16\pi}\int_\Sigma (\scal-2\ric(\nu,\nu))\,d\mu - m_{\mathcal{H}}(\Sigma)\right\vert  \\\nonumber
& \quad = \left\vert \frac{r}{16\pi}\int_\Sigma \left(\scal^\Sigma - \frac{H^2}{2} + \Vert \mathring{A}\Vert ^2\right)\,d\mu -\sqrt{\frac{\vert \Sigma\vert }{16\pi}}\left(1-\frac{1}{16\pi}\int_\Sigma (H^2-P^2) \, d\mu\right)\right\vert \\
& \quad \leq \frac{C}{\sigma^{2 \varepsilon}}.
\end{align}
This last inequality follows from \eqref{eqGaussConsequences}, the Gauss-Bonnet Theorem, and the definition of $r$, $\sigma$, and $m_{\mathcal{H}}(\Sigma)$ with possibly enlarged $C>0$ and $\overline{\sigma}>0$. This proves the claims of the lemma.
\end{proof}

\begin{remark}
Since $P^2=O(\sigma^{-3-2\varepsilon})$ in \eqref{eqIntermediateEV2}, this lemma remains valid if we replace the Hawking mass $m_{\mathcal{H}}(\Sigma)$ by the Geroch mass $m_{H}(\Sigma)$ (also sometimes referred to as ``(Riemannian) Hawking mass'') given by
 \begin{align}\label{eqHawkingMassRiemannian}
 m_{H} (\Sigma) = \sqrt{\frac{\vert \Sigma\vert }{16\pi}}\left(1-\frac{1}{16\pi}\int_\Sigma H^2 \, d\mu\right).
\end{align}
The same remark will hold true for the subsequent results. However, we choose to use $m_{\mathcal{H}} (\Sigma)$, and not $m_H (\Sigma)$, throughout to emphasize the spacetime nature of our result.
\end{remark}

\subsection{Invertibility of the operator $\mathcal{L}$}\label{ssecInvertibility}
Section \ref{secEigenLaplace} above provides the following description of the eigenvalues of the Laplacian $-\Delta^\Sigma$:
\begin{itemize}
\item $\lambda_0 = 0$,\\[-1.5ex]
\item for $i=1,2,3$ the eigenvalues $\lambda_i$ are characterized by formula \eqref{eqIntRicci1},\\[-1.5ex]
\item for $i=4,5,\ldots$ we have $\lambda_i > \sfrac{5}{\sigma^2}$.
\end{itemize}
It turns out to be useful to decompose functions $h\in L^{2}(\Sigma)$ with respect to the $L^{2}(\Sigma)$-complete orthonormal system $\{f_{0},f_{1},f_{2},f_{3},\dots\}$ of eigenfunctions corresponding to $-\Delta^{\Sigma}$ when analyzing $\mathcal{L}h$ and the $L^{2}(\Sigma)$-adjoint $\mathcal{L}^{*}h$ (the latter being of interest as we are aiming for a Fredholm Alternative argument). More specifically, it is useful to split any given function $h \in L^2(\Sigma)$ into its \emph{mean value}
\begin{align}
h^{0}\definedas \fint_{\Sigma} h\,d\mu=f_{0}\int_\Sigma h f_0 \, d\mu,
\end{align}
its \emph{translational part}
\begin{align}
h^t \definedas \sum_{i=1}^3 f_i \int_\Sigma h f_i \, d\mu,
\end{align}
and the \emph{difference part}\footnote{We do not call $h^{d}$ the ``deformational part'' as some other authors do, because $h^{d}$ also contains the mean value information. In other words, we primarily use this splitting to distinguish between the eigenfunctions corresponding to  eigenvalues of different magnitude.}
\begin{align}
h^d \definedas h-h^t.
\end{align}
 
\begin{proposition}\label{propInvertibility} Let $\mathcal{I} = \IDS$ be a $C^{2}_{\sfrac{1}{2}+\varepsilon}$-asymptotically Euclidean initial data set with energy $E$. Suppose that $a \in [0,1)$, $b \geq 0$, $\eta \in (0,1]$, and that $\Sigma \in\mathcal{A}(a,b,\eta)$ with respect to $\mathcal{I}$ is a surface with non-vanishing Hawking mass $m_{\mathcal{H}}(\Sigma) \neq 0$. Then there exist constants $C>0$ and $\overline{\sigma}>0$, with $C$ depending only on $\varepsilon$, $a$, $b$, $\eta$, and $C_{\mathcal{I}}$ and $\overline{\sigma}$ in addition depending on $\vert E\vert $ in \eqref{eqInvertible}, such that if $\Sigma$ has constant spacetime mean curvature $\mathcal{H} \equiv\sfrac{2}{\sigma}$ for $\sigma >\overline{\sigma}$,  the following estimates
\begin{align}\label{eqTranslational0}
\Vert \mathcal{L}h^t\Vert _{L^2(\Sigma)} &\leq \frac{C}{\sigma^{\frac{5}{2}+\varepsilon}}\Vert h^t\Vert _{L^2(\Sigma)}\\\label{eqTranslational} 
\left\vert  \int_\Sigma (\mathcal{L} h_1^t) h_2^t \,d\mu - \frac{6m_\mathcal{H}(\Sigma)}{\sigma^3} \int_\Sigma h_1^t h_2^t \,d\mu \right\vert  & \leq  \frac{C}{\sigma^{3+\varepsilon}} \Vert h_1^t\Vert _{L^2(\Sigma)}\Vert h_2^t\Vert _{L^2(\Sigma)} \\\label{eqforLemmaLapse}
\left\vert \int_{\Sigma} h^d \mathcal{L} f_i \, d\mu \right\vert  &\leq \frac{C}{\sigma^{\frac{5}{2}+\varepsilon}}\Vert h^{d}\Vert_{L^{2}(\Sigma)}\\ \label{eqDeformational} 
\frac{3}{2\sigma^2}\Vert h^d\Vert _{L^2(\Sigma)} & \leq  \Vert \mathcal{L}h^d\Vert _{L^2(\Sigma)}\\\label{eqInvertible}
\frac{3\vert m_\mathcal{H}(\Sigma)\vert }{\sigma^3} \Vert h\Vert _{L^2(\Sigma)} & \leq  \Vert \mathcal{L} h\Vert _{L^2(\Sigma)}
\end{align}
hold for any $h, h_1, h_2 \in W^{2,2}(\Sigma)$. The same estimates apply to the $L^{2}(\Sigma)$-adjoint $\mathcal{L}^{*}$. Moreover, the Hawking mass $m_{\mathcal{H}}(\Sigma)$ and the energy $E$ are related by
\begin{align}\label{eqEmH}
\left\vert E - m_{\mathcal{H}}(\Sigma)\right\vert  &\leq  C\sigma^{-\varepsilon}.
\end{align}
In particular, the operator $\mathcal{L}\colon W^{2,2}(\Sigma) \to L^2(\Sigma)$ is invertible as long as the energy $E$ of the initial data set does not vanish and $\overline{\sigma}$ is sufficiently large, depending only on $\varepsilon$, $a$, $b$, $\eta$, and $C_{\mathcal{I}}$. 
\end{proposition}

\begin{proof} 
In this proof, $C>0$ and $\overline{\sigma}>0$ denote generic constants that may vary from line to line, but depend only on $\varepsilon$, $a$, $b$, $\eta$, and $C_{\mathcal{I}}$, and, in the case of \eqref{eqInvertible}, also on $\vert E\vert $.\\

\paragraph{\emph{Proving \eqref{eqTranslational0}.}} By definition of $\mathcal{L}$ in \eqref{eqLinearization}, we have
\begin{align*}
  \mathcal{L} h^t =& - \Delta ^{\Sigma} h^t - \left( \vert \mathring A \vert  ^2 + \tfrac{H^2}{2} + \ric (\nu,\nu) \right) h^t \\
  &- \tfrac{P}{H} \left(\left(\nabla_\nu \tr K - \nabla_\nu K(\nu, \nu)\right) h^t - 2 K \left(\nabla^\Sigma h^t,\nu\right) \right).
\end{align*}
It follows from Proposition~\ref{propRegularity} that $\vert \mathring A \vert  ^2 =O( \sigma^{-3-\epsilon})$, and we know from  \eqref{eqFallOff} that $H^2 = 4 \sigma^{-2} + O(\sigma^{-3-2\epsilon})$, and that $\frac{P}{H} = O(\sigma^{-\frac{1}{2}-\varepsilon})$. Furthermore, the definition of $C^{2}_{\sfrac{1}{2}+\varepsilon}$-asymptotically Euclidean initial data sets implies that $\ric (\nu,\nu) =O(\sigma^{-\frac{5}{2} - \epsilon})$ and that $\nabla _\nu \tr K - \nabla _\nu K (\nu,\nu) = O(\sigma^{-\frac{5}{2}-\epsilon})$. Hence, 
\begin{align*}
 \mathcal{L} h^t = -\Delta^\Sigma h^t - \tfrac{2}{\sigma^2} h^t  + \tfrac{2P}{H}K(\nabla^{\Sigma} h^t , \nu)  + O(\sigma^{-\frac{5}{2}- \epsilon})h^t.
\end{align*}
By \eqref{eqEigenEstimates1a}, we have
\begin{align*}
\left\Vert - \Delta^\Sigma h^t - \frac{2}{\sigma^2}h^t \right\Vert _{L^2(\Sigma)} \leq \frac{C}{\sigma^{\frac{5}{2} + \epsilon}}\, \Vert h^t \Vert _{L^2(\Sigma)},
\end{align*}
whereas \eqref{eqEigenEstimates2} implies by a Cauchy--Schwarz Inequality that 
\begin{align*}
\left\Vert\frac{2P}{H}K(\nabla^{\Sigma} h^t,\nu) \right\Vert _{L^2(\Sigma)} \leq \frac{C}{\sigma^{3+\epsilon}}\, \Vert h^t \Vert  _{L^2(\Sigma)}
\end{align*}
recalling that $h^{t}\in\Span(f_{1},f_{2},f_{3})$. This proves \eqref{eqTranslational0}.\\

\paragraph{\emph{Proving \eqref{eqTranslational}.}} Arguing as above, by Proposition \ref{propRegularity}, Lemma \ref{lemEigenLaplace2}, and our decay assumptions on the initial data set, we have that
\begin{align*}
\int_\Sigma (\mathcal Lf_i)f_j\, d\mu = \left(\lambda_i - \tfrac{2}{\sigma^2}\right)\delta_{ij}- \int_\Sigma \ric (\nu,\nu)f_i f_j \, d\mu + O(\sigma^{-3-2\varepsilon})
 \end{align*}
for any $i,j=1,2,3$. It then follows by Lemma \ref{lemLaplace} that
\begin{align*}
\begin{split}
\left\vert \int_\Sigma (\mathcal{L}f_i) f_j \, d\mu\right\vert  & \leq \frac{C}{\sigma^{3+\varepsilon}} \quad\;\; \text{ for } i \neq j,\; i,j= 1,2,3,
\end{split}
\end{align*}
and 
\begin{align*}
\left\vert  \int_\Sigma (\mathcal{L} f_i) f_i\, d\mu - \frac{6m_\mathcal{H}(\Sigma)}{\sigma^3} \right\vert  \leq \frac{C}{\sigma^{3+\varepsilon}} \quad \text{ for } i =1,2,3.
\end{align*}
In particular, we see that \eqref{eqTranslational} holds for $h_1^{t},h_2^{t}\in\{f_1,f_2,f_3\}$. The general case follows by bilinearity and by the Cauchy--Schwarz Inequality on $\mathbb{R}^{3}$.\\

\paragraph{\emph{Proving \eqref{eqforLemmaLapse}.}} By definition of $\mathcal{L}$, we have that
\begin{align*}
 \mathcal{L} f_{i} = -\Delta^\Sigma f_{i} -\! (\vert \mathring{A}\vert^{2}+ \tfrac{H^{2}}{2}+\ric(\nu,\nu)+\tfrac{P}{H}[\nabla_{\nu}\tr K-\nabla_{\nu}K(\nu,\nu)]) f_{i}  + \tfrac{2P}{H}K(\nabla^{\Sigma} f_{i}, \nu).
\end{align*}
Next, by \eqref{eqFallOff}, we have
\begin{align*}
-\Delta^{\Sigma}f_{i}-\tfrac{{H}^{2}}{2}f_{i}&=(\lambda_{i}-\tfrac{{H}^{2}}{2})f_{i}=(\lambda_{i}-\tfrac{2}{\sigma^{2}})f_{i} + O(\sigma^{-3-2\varepsilon}) f_{i},
\end{align*}
while the $C^{2}_{\sfrac{1}{2}+\varepsilon}$-asymptotic decay assumptions as well as \eqref{eqFallOff},  \eqref{eqEigenEstimates2}, and the Cauchy--Schwarz Inequality lead to
\begin{align*}
\tfrac{P}{H}\left[\nabla_{\nu}\tr K-\nabla_{\nu}K(\nu,\nu)\right]&=O(\sigma^{-3-2\varepsilon}),\\[0.5ex]
\left\Vert  \tfrac{2P}{H}K(\nabla^{\Sigma} f_{i},\nu) \right\Vert _{L^2(\Sigma)} &\leq C \sigma^{-3-2\epsilon}.
\end{align*}
Proposition \ref{propRegularity} gives us that $\vert\mathring{A}\vert^{2}=O(\sigma^{-3-2\varepsilon})$. Summarizing, we found
\begin{align*}
\left\vert \int_{\Sigma} u^d \mathcal{L} f_i \, d\mu \right\vert  &\leq \frac{C}{\sigma^{3 + 2\varepsilon}}\, \Vert u^d\Vert _{L^2(\Sigma)}+\left\vert \int_{\Sigma}\ric(\nu,\nu)u^{d}f_{i}\,d\mu \right\vert\\
&\leq \frac{C}{\sigma^{\frac{5}{2}+\varepsilon}}\,\Vert u^{d}\Vert_{L^{2}(\Sigma)},
\end{align*}
recalling that $u^{d}$ is $L^{2}(\Sigma)$-orthogonal to $f_{i}$ for $i=1,2,3$.

\paragraph{\emph{Proving \eqref{eqDeformational}.}} We will use the following manifest relation for the linear operator $\mathcal{L}$
\begin{align}\label{eqL2Norm}
\Vert \mathcal{L} h^d \Vert ^2_{L^2(\Sigma)} = \Vert \mathcal{L}(h^d-h^0)\Vert ^2_{L^2(\Sigma)} + \Vert \mathcal{L}h^0\Vert ^2_{L^2(\Sigma)} + 2\int_\Sigma \mathcal{L}(h^d-h^0) \mathcal{L} h^0 \, d\mu.
\end{align}
Arguing similarly to how we argued above, we now integrate by parts and use Proposition~\ref{propRegularity}, \eqref{eqFallOff}, and \eqref{eqBigEigenvalues} giving $\lambda_{i}>\sfrac{5}{\sigma^{2}}$ for $i=4,5,\dots$, and the asymptotic decay conditions on $\mathcal{I}$ to estimate from below the expression
\begin{align*}
\begin{split}
\int_\Sigma (h^d - h^0) &\mathcal{L} (h^d - h^0) \, d\mu\\
& = \int_\Sigma \left[-(h^d-h^0)\Delta^\Sigma (h^d-h^0) - \left(\tfrac{H^2}{2} +\vert \mathring{A}\vert ^2 + \ric(\nu,\nu)\right) (h^d-h^0)^2\right]\,d\mu\\
& \qquad  - \int_\Sigma \left[\tfrac{P}{H} \left(\nabla_\nu \tr K -\nabla_\nu K(\nu,\nu)\right) + \divg^\Sigma \left(\tfrac{P}{H}K(\nu,\cdot)\right)\right] (h^d-h^0)^2\, d\mu \\ 
& = \int_\Sigma \left[ - (h^d-h^0) \Delta^\Sigma (h^d-h^0) - \left(2\sigma^{-2} + O(\sigma^{-\frac{5}{2}-\varepsilon})\right) (h^d-h^0)^2 \right]  \, d\mu \\ & \geq \left(3\sigma^{-2}+ O(\sigma^{-\frac{5}{2}-\varepsilon})\right) \int_\Sigma (h^d -h^0)^2 \, d\mu \\ & \geq {\frac{7}{4\sigma^2}} \int_\Sigma (h^d -h^0)^2 \, d\mu
\end{split}
\end{align*}
as $h^{d}-h^{0}\in\Span(f_{4},f_{5},\dots)$. Here, the factor $\tfrac{7}{4}<3$ is chosen for later convenience. Hence by a Cauchy--Schwarz Inequality on $\int_\Sigma (h^d - h^0) \mathcal{L} (h^d - h^0)\,d\mu$, we obtain
\begin{align}\label{eqL2Norm1}
\Vert \mathcal{L}(h^d-h^0)\Vert _{L^2(\Sigma)} \geq {\frac{7}{4\sigma^2}} \Vert h^d-h^0\Vert _{L^2(\Sigma)}. 
\end{align}
Note also that $h_0$ is a constant, so that
\begin{align}\label{eqLh0}
\begin{split}
 \mathcal{L} h^0 & = - \left(\tfrac{H^2}{2} + \vert \mathring{A}\vert ^2 + \ric(\nu,\nu)+\tfrac{P}{H} \left(\nabla_\nu \tr K -\nabla_\nu K(\nu,\nu)\right)\right) h^0 \\
 & = - \left(2\sigma^{-2} + O(\sigma^{-\frac{5}{2}-\varepsilon})\right) h^0,
 \end{split}
\end{align}
and thus
\begin{align}\label{eqL2Norm2}
\Vert \mathcal{L}h^0\Vert _{L^2(\Sigma)} \geq {\frac{7}{4\sigma^2}} \Vert h^0\Vert _{L^2(\Sigma)},
\end{align}
where again, the factor $\tfrac{7}{4}<2$ is chosen for later convenience. Using \eqref{eqLh0}, integration by parts and finally Young's Inequality, one can also check with the same decay arguments as above that
\begin{align*}
\begin{split}
\int_\Sigma (\mathcal{L} h^0)& \mathcal{L}(h^d-h^0) \, d\mu\\
 & = -2\sigma^{-2} \int_\Sigma h^0  \mathcal{L}(h^d-h^0)\, d\mu + \int_\Sigma O(\sigma^{-\frac{5}{2}-\varepsilon})h^0 \mathcal{L}(h^d-h^0)\, d\mu \\
& = -2 \sigma^{-2}  \int_\Sigma h^0 \left(2 \tfrac{P}{H}K(\nabla^\Sigma(h^d-h^0),\nu)+ O(\sigma^{-\frac{5}{2}-\varepsilon})(h^d-h^0)\right)\, d\mu \\ & \qquad + \int_\Sigma O(\sigma^{-\frac{5}{2}-\varepsilon})h^0 \mathcal{L}(h^d-h^0)\, d\mu
\end{split}
\end{align*}
\begin{align*}
\begin{split}
& =  \int_\Sigma O(\sigma^{-\frac{9}{2}-\varepsilon})h^0 (h^d-h^0)\, d\mu + \int_\Sigma O(\sigma^{-\frac{5}{2}-\varepsilon})h^0 \mathcal{L}(h^d-h^0)\, d\mu \\
& \geq - C \sigma^{-\frac{9}{2}-\varepsilon} \int_\Sigma \left\vert h^d-h^0\right\vert \left\vert  h^0 \right\vert \,d\mu - C \sigma^{-\frac{5}{2}-\varepsilon} \int_\Sigma \left\vert h^0\right\vert  \left\vert \mathcal{L}( h^d-h^0) \right\vert \, d\mu \\
& \geq - C \sigma^{-\frac{9}{2}-\varepsilon} (\Vert h^d-h^0\Vert ^2_{L^2(\Sigma)} + \Vert h^0\Vert ^2_{L^2(\Sigma)}) - C\sigma^{-\frac{1}{2}-\varepsilon}\Vert \mathcal{L}(h^d-h^0)\Vert ^2_{L^2(\Sigma)}.
\end{split}
\end{align*}
Combing this estimate with \eqref{eqL2Norm1} and \eqref{eqL2Norm2}, \eqref{eqDeformational} follows from \eqref{eqL2Norm} once we recall that $h^{d}-h^{0}$ is $L^{2}(\Sigma)$-orthogonal to $h^{0}$.\\

\paragraph{\emph{Proving \eqref{eqEmH}.}} To see that $E$ and $m_{\mathcal{H}}(\Sigma)$ are as close as claimed, we recall the well-known fact that the Geroch mass $m_{H}(\Sigma)$ of sufficiently round large surfaces in a $C^{2}_{\sfrac{1}{2}+\varepsilon}$-asymptotically flat initial data set $\mathcal{I}$ is close to the energy $E$ of $\mathcal{I}$. More specifically, Lemma A.1 in \cite{NerzCMC} (relying on \cite{HerzlichRicci} and \cite{MiaoTam}) and \eqref{eqIntermediateEV2} imply that 
\begin{align*}
\left\vert E - m_{\mathcal{H}}(\Sigma)\right\vert  \leq &\left\vert E - \frac{r}{16\pi}\int_\Sigma (\scal-2\ric(\nu,\nu))\,d\mu \right\vert \\& + \left\vert \frac{r}{16\pi}\int_\Sigma (\scal-2\ric(\nu,\nu))\,d\mu - m_{\mathcal{H}}(\Sigma)\right\vert \\
\leq & \,C\sigma^{-\varepsilon}.
\end{align*}
Thus $m_{\mathcal{H}}(\Sigma) \neq 0$ if $E\neq 0$ as long as $\sigma>\overline{\sigma}$ with $C>0$ and $\overline{\sigma}>0$ sufficiently large, depending only on $\varepsilon$, $a$, $b$, $\eta$, and $C_{\mathcal{I}}$.\\

\paragraph{\emph{Proving \eqref{eqInvertible}.}} We pick a number $\kappa>0$ such that $1-4\varepsilon<2\kappa<1$ and consider two cases. In this part, we will abbreviate $m_{\mathcal{H}}(\Sigma)\asdefined m_{\mathcal{H}}$. \\[-0.75ex]

\paragraph{\emph{Case 1.}} Suppose that $\Vert h^d\Vert ^2_{L^2(\Sigma)} \geq \sigma^{-\frac{1}{2} - \kappa} \Vert h\Vert ^2_{L^2(\Sigma)}$.
As a consequence, using \eqref{eqTranslational0}, \eqref{eqDeformational}, and Young's Inequality, we obtain for any $0<\alpha<1$, e.g.~$\alpha=\tfrac{1}{2}$, that
\begin{align*}
\begin{split}
\Vert \mathcal{L} h\Vert ^2_{L^2(\Sigma)} & = \int_\Sigma \left((\mathcal{L} h^d)^2 + 2(\mathcal{L} h^d)( \mathcal{L} h^t) + (\mathcal{L} h^t)^2 \right)\, d\mu \\
& \geq  (1-\alpha)\left(\Vert \mathcal{L} h^d\Vert _{L^2(\Sigma)}^2 - \alpha^{-1} \Vert \mathcal{L} h^t\Vert _{L^2(\Sigma)}^2\right) \\
& \geq (1-\alpha) \left(\frac{9}{4\sigma^4}\Vert h^d\Vert ^2_{L^2(\Sigma)} - \frac{C}{\sigma^{5+2\varepsilon}}\Vert h^t\Vert ^2_{L^2(\Sigma)}\right) \\ 
& \geq (1-\alpha) \left(\frac{9}{4\sigma^{\frac{9}{2}+\kappa}}\Vert h\Vert ^2_{L^2(\Sigma)} - \frac{C}{\sigma^{5+2\varepsilon}}\Vert h\Vert ^2_{L^2(\Sigma)}\right) \\
& \geq \frac{9\vert m_\mathcal{H}\vert ^2}{\sigma^6} \Vert h\Vert ^2_{L^2(\Sigma)}
\end{split}
\end{align*}
provided that $\sigma>\overline{\sigma}$, where now $\overline{\sigma}$ may actually depend on $E$ as we used \eqref{eqEmH} in the last step. Thus \eqref{eqInvertible} holds in case $\Vert h^d\Vert ^2_{L^2(\Sigma)} \geq \sigma^{-\frac{1}{2} - \kappa} \Vert h\Vert ^2_{L^2(\Sigma)}$.\\[-0.75ex]

\paragraph{\emph{Case 2.}} Now assume that $\Vert h^d\Vert ^2_{L^2(\Sigma)} \leq \sigma^{-\frac{1}{2} - \kappa} \Vert h\Vert ^2_{L^2(\Sigma)}$. In this case,
\begin{align}\label{eqTriangle}
\left\vert  \int_{\Sigma} h \mathcal{L} h \, d\mu \right\vert  \geq \left\vert \int_\Sigma h^t \mathcal{L} h^t \, d \mu\right\vert  - \left\vert \int_\Sigma h^d \mathcal{L} h \, d \mu\right\vert  - \left\vert \int_\Sigma h^t \mathcal{L} h^d  \, d\mu \right\vert ,
\end{align}
where by \eqref{eqTranslational}, \eqref{eqEmH}, and because $h^{t}$ and $h^{d}$ are $L^{2}(\Sigma)$-orthogonal, we see that for $\sigma>\overline{\sigma}$, $\overline{\sigma}$ suitably large depending now in addition on $E$, we have
\begin{align*}
\begin{split}
\left\vert \int_\Sigma h^t \mathcal{L} h^t \, d \mu\right\vert  
& \geq \frac{6\vert m_\mathcal{H}\vert -C\sigma^{-\varepsilon}}{\sigma^3} \Vert h^t\Vert ^2_{L^2(\Sigma)} \\
& = \frac{6\vert m_\mathcal{H}\vert -C\sigma^{-\varepsilon}}{\sigma^3} \left(\Vert h\Vert ^2_{L^2(\Sigma)} - \Vert h^d\Vert ^2_{L^2(\Sigma)}\right) \\
& \geq \frac{6\vert m_\mathcal{H}\vert -C\sigma^{-\varepsilon}}{\sigma^3} \left(1 - \frac{1}{\sigma^{\frac{1}{2} + \kappa}}\right) \Vert h\Vert ^2_{L^2(\Sigma)}  \\
& \geq \frac{6\vert m_\mathcal{H}\vert -C\sigma^{-\varepsilon}}{\sigma^3}\Vert h\Vert ^2_{L^2(\Sigma)},
\end{split}
\end{align*}
where we used that $\varepsilon\leq \tfrac{1}{2}$ by definition. Moreover, using the Cauchy--Schwarz Inequality and the assumption $\Vert h^d\Vert ^2_{L^2(\Sigma)} \leq \sigma^{-\frac{1}{2} - \kappa} \Vert h\Vert ^2_{L^2(\Sigma)}$, we estimate
\begin{align*}
\left\vert \int_\Sigma h^d \mathcal{L} h \, d\mu \right\vert  
\leq C\sigma^{-\frac{1}{4}-\frac{\kappa}{2}} \Vert h\Vert _{L^2(\Sigma)} \Vert \mathcal{L} h\Vert _{L^2(\Sigma)}.
\end{align*}
Further, arguing once more as above with the explicit form of $\mathcal{L}$ in \eqref{eqLinearization}, using the asymptotic decay conditions of $\mathcal{I}$, \eqref{eq2FFtraceless}, \eqref{eqFallOff}, and integration by parts, one confirms that
\begin{align*}
\left\vert \int_\Sigma h^t \mathcal{L} h^d \, d \mu\right\vert  \leq C \sigma^{-\frac{5}{2}-\varepsilon} \Vert h^t\Vert _{L^2(\Sigma)} \Vert h^d\Vert _{L^2(\Sigma)} \leq C \sigma^{-\frac{11}{4}-\varepsilon-\frac{\kappa}{2}} \Vert h\Vert ^2_{L^2(\Sigma)},
\end{align*}
using again the assumption $\Vert h^d\Vert ^2_{L^2(\Sigma)} \leq \sigma^{-\frac{1}{2} - \kappa} \Vert h\Vert ^2_{L^2(\Sigma)}$ as well as $\Vert h^{t}\Vert_{L^{2}(\Sigma)}\leq\Vert h\Vert_{L^{2}(\Sigma)}$ in the last step. It then follows from \eqref{eqTriangle}, the Cauchy--Schwarz Inequality, and the bounds on $\kappa$ that \eqref{eqInvertible} also holds when $\Vert h^d\Vert ^2_{L^2(\Sigma)} \leq \sigma^{-\frac{1}{2} - \kappa} \Vert h\Vert ^2_{L^2(\Sigma)}$. \\[-0.75ex]

Combining Case 1 and Case 2, we have thus shown \eqref{eqInvertible}. To conclude the proof, it only remains to show that $\mathcal{L}^{*}$ also satisfies the estimates \eqref{eqTranslational0}--\eqref{eqInvertible} and that $\mathcal{L}$ is invertible provided that the initial data set $\mathcal{I}$ has non-vanishing energy $E\neq 0$ and the spacetime mean curvature radius $\sigma$ of $\Sigma$ is sufficiently large, $\sigma>\overline{\sigma}$, $\overline{\sigma}$ depending on $E$.\\

\paragraph{\emph{Invertibility of $\mathcal{L}$ and estimates on $\mathcal{L}^{*}$.}} The operator $\mathcal{L}$ is not self-adjoint, but its $L^{2}(\Sigma)$-adjoint $\mathcal{L}^*$ has very similar structure, differing only in the last term. In $\mathcal{L}h$, this term reads $\frac{P}{H}K(\nabla^{\Sigma}h,\nu)$, while in $\mathcal{L}^{*}h$, this term becomes $-\divg^{\Sigma}\left(\frac{P}{H}K(\cdot,\nu)\right)h$. Recall that
\begin{align}\label{eqL*}
\tfrac{P}{H}K(\cdot,\nu)=O_{1}(\sigma^{-2-2\varepsilon}).
\end{align}
Going back to the proofs of \eqref{eqTranslational0}--\eqref{eqInvertible}, we see that all of them work if we replace $\mathcal{L}$ by $\mathcal{L}^*$ modulo exchanging the performance of  partial integration with the decay estimate \eqref{eqL*} and vice versa. This, in particular \eqref{eqInvertible}, implies that the $L^2(\Sigma)$-kernel of $\mathcal{L}^{*}$ is trivial, and hence $\mathcal{L}\colon W^{2,2}(\Sigma) \to L^2(\Sigma)$ is invertible by the Fredholm Alternative \cite[Appendix I, Theorem 31]{Besse}, as long as $m_{\mathcal{H}}(\Sigma) \neq 0$ which is guaranteed from $E\neq0$ and \eqref{eqEmH}. The Fredholm Alternative applies as $\mathcal{L}$ is clearly a linear elliptic operator as its symbol is that of the Laplacian $-\Delta^{\Sigma}$ and because $\Sigma$ is compact.
\end{proof} 

\begin{corollary}\label{corEllipticRegularity}
For every $h \in W^{2,2}(\Sigma)$, we have
\begin{align*}
 \Vert h\Vert _{W^{2,2}(\Sigma)}& \leq C \left(\sigma^2\Vert \mathcal{L} h\Vert _{L^2(\Sigma)} + \Vert h\Vert _{L^2(\Sigma)} \right)\leq  C\sigma^3 \Vert \mathcal{L} h\Vert _{L^2(\Sigma)}, \\
\Vert h^d\Vert _{W^{2,2}(\Sigma)} &\leq   C \sigma^2 \Vert \mathcal{L}h^d\Vert _{L^2(\Sigma)}
\end{align*}
for $\sigma>\overline{\sigma}$ and with $C>0$, $\overline{\sigma}>0$ only depending on $\varepsilon$, $a$, $b$, $\eta$, $C_{\mathcal{I}}$, and $E$.
\end{corollary}
\begin{proof}

Note by \eqref{eqGaussConsequences} we have $\scal^\Sigma = \tfrac{2}{\sigma^2}+O(\sigma^{-\tfrac{5}{2}-\varepsilon})$, hence, in the view of \eqref{eqMCArea}, Lemma \ref{lemCalderon} applies to $\Sigma$. Combined with the Cauchy-Schwartz Inequality and \eqref{eqMCArea}, this result gives us
\begin{align}\label{eqElRegLaplace}
\Vert h\Vert_{W^{2,2}(\Sigma)} \leq C\sigma^2 \Vert \Delta^\Sigma h \Vert_{L^2 (\Sigma)} + \Vert  h \Vert_{L^2 (\Sigma)}.
\end{align}
Recalling the definition of the operator $\mathcal{L}$ (see \eqref{eqLinearization}) and the fall-off properties of the initial data set, we further find that
\begin{eqnarray*}
\Vert \Delta^{\Sigma} h\Vert_{L^2(\Sigma)} & \leq&  \Vert \mathcal{L} h \Vert_{L^2 (\Sigma)} + \left(\tfrac{2}{\sigma^2} + O(\sigma^{-\frac{5}{2}-\varepsilon}) \right) \Vert  h \Vert_{L^2 (\Sigma)} + O (\sigma^{-2}) \Vert \nabla^\Sigma h \Vert_{L^2 (\Sigma)} \\
& \leq & \Vert \mathcal{L} h \Vert_{L^2 (\Sigma)} + \left(\tfrac{2}{\sigma^2} + O(\sigma^{-\frac{5}{2}-\varepsilon}) \right) \Vert  h \Vert_{L^2 (\Sigma)} + O (\sigma^{-2}) \Vert h \Delta^\Sigma h \Vert_{L^1 (\Sigma)} \\ & \leq & \Vert \mathcal{L} h \Vert_{L^2 (\Sigma)} + O(\sigma^{-2}) \Vert  h \Vert_{L^2 (\Sigma)} + O (\sigma^{-2}) \Vert \Delta^\Sigma h \Vert_{L^2 (\Sigma)}
\end{eqnarray*}
hence 
\begin{align*}
\Vert \Delta^{\Sigma} h\Vert_{L^2(\Sigma)}  \leq  \Vert \mathcal{L} h \Vert_{L^2 (\Sigma)} + C \sigma^{-2} \Vert  h \Vert_{L^2 (\Sigma)}
\end{align*}
Combining \eqref{eqElRegLaplace} with this estimate and \eqref{eqInvertible} 
we thereby obtain
\begin{align*}
\Vert h\Vert _{W^{2,2}(\Sigma)} & \leq C\left(\sigma^2 \Vert \mathcal{L} h\Vert _{L^2(\Sigma)}+ \Vert  h\Vert _{L^2(\Sigma)} \right) \\
& \leq C\left(\sigma^2 \Vert \mathcal{L} h\Vert _{L^2(\Sigma)}+ \frac{\sigma^3} {3\vert m_\mathcal{H}(\Sigma) \vert}\Vert \mathcal{L} h\Vert _{L^2(\Sigma)} \right) \\
& \leq C \sigma^3 \Vert \mathcal{L} h\Vert _{L^2(\Sigma)}.
\end{align*}
This proves the estimate for $h$.
The estimate for $h^d$ is proven similarly, using \eqref{eqDeformational}  instead of \eqref{eqInvertible}.
\end{proof}

%% file: existence.tex
% !TEX root = main.tex
In this section, we will prove that any $C^2_{\sfrac{1}{2} + \varepsilon}$-asymptotically Euclidean initial data set $\mathcal{I}=\IDS$ is foliated (i.e., roughly speaking, covered without gaps or overlaps), outside a compact set, by $2$-surfaces of constant spacetime mean curvature (STCMC). We also prove a uniqueness result for STCMC-surfaces.  

\subsection{Existence of the STCMC-foliation}\label{subsecExistence}

In \cite{NerzCMC}, Nerz proved the following result, rephrased here in our notation. Note that because of time symmetry $K\equiv0$, the CMC-foliation constructed by Nerz can be viewed as a special case of the STCMC-foliation under consideration here.

\begin{theorem}[Nerz 2015]\label{thExistenceCMC}
 Let $(M^{3},g)$ be a $C^2_{\sfrac{1}{2} + \varepsilon}$-asymptotically Euclidean Riemannian manifold viewed as a $C^2_{\sfrac{1}{2} + \varepsilon}$-asymptotically Euclidean  initial data set $\mathcal{I}_0=(M^{3},g,K\equiv 0,\mu=\tfrac{1}{2}\scal,J\equiv 0)$, with non-vanishing energy $E \neq 0$. Then there is a constant $\sigma_{\mathcal{I}_{0}}>0$ depending only on $\varepsilon$, $C_{\mathcal{I}_0}$, and $E$, a compact set $\mathcal{K}_{0} \subset M^{3}$, and a bijective $C^{1}$-map $\Psi_0\colon (\sigma_{\mathcal{I}_{0}}, \infty) \times \mathbb{S}^2 \to M^{3} \setminus \mathcal{K}_0$ such that each of the surfaces $\Sigma_0^\sigma \definedas \Psi_0(\sigma, \mathbb{S}^2)$ has constant mean curvature $H(\Sigma_0^\sigma)\equiv\sfrac{2}{\sigma}$ provided that $\sigma>\sigma_{\mathcal{I}_{0}}$. 
\end{theorem}

This result is a starting point for proving the following theorem, which is essentially the main result of this paper. For the sake of clarity of exposition, we provide the proof of the following theorem right away, saving the verification of some preliminary lemmas for later. We state Theorem \ref{thExistence} here in a notation convenient for its proof.

\begin{theorem}[Existence of STCMC-foliation]\label{thExistence}
Let $\mathcal{I}_{1}=\IDS$ be a $C^2_{\sfrac{1}{2} + \varepsilon}$-asymptotically Euclidean initial data set with non-vanishing energy $E \neq 0$. Then there is a constant $\sigma_{\mathcal{I}_{1}}>0$ depending only on $\varepsilon$, $C_{\mathcal{I}_{1}}$, and $E$, a compact set $\mathcal{K}_1 \subset M^{3}$, and a bijective $C^{1}$-map $\Psi_1\colon (\sigma_{\mathcal{I}_{1}}, \infty) \times \mathbb{S}^2 \to M^{3} \setminus \mathcal{K}_1$ such that each of the surfaces $\Sigma_1^\sigma \definedas \Psi_1(\sigma, \mathbb{S}^2)$ has constant spacetime mean curvature $\mathcal{H}( \Sigma_1^\sigma) \equiv\sfrac{2}{\sigma_{1}}$ provided that $\sigma>\sigma_{\mathcal{I}_{1}}$. 
\end{theorem}

\begin{remark}
As the proof of Theorem \ref{thExistence} will show, the surfaces $\Sigma_1^\sigma$ are in fact asymptotically centered in the sense of Definition \ref{defConcentric}, more specifically, $\Sigma_1^\sigma \in \mathcal{A}(0, b_{\mathcal{I}_{1}},\eta_{\mathcal{I}_{1}})$ for all $\sigma> \sigma_{\mathcal{I}_{1}}$, with constants $b_{\mathcal{I}_{1}}>0$, $\eta_{\mathcal{I}_{1}}\in(0,1]$, and $\sigma_{\mathcal{I}_{1}}>0$ defined in the proof of Theorem \ref{thExistence}, and depending only on $\varepsilon$, $C_{\mathcal{I}_{1}}$, and $E$. 
\end{remark}
                                                                    
\begin{proof}
The family of closed, oriented $2$-surfaces $\{\Sigma_1^\sigma\}_{\sigma>\sigma_{\mathcal{I}_{1}}}$ will be constructed via a method of continuity, see also Section \ref{secMain}. Roughly speaking, we will deform  the constant (automatically spacetime) mean curvature foliation $\{\Sigma_0^\sigma\}_{\sigma>\sigma_{\mathcal{I}_{0}}}$ of the initial data set $\mathcal{I}_0$ from Theorem \ref{thExistenceCMC} along the curve of initial data sets $\{\mathcal{I}_\tau\}_{\tau \in [0,1]}$, where $\mathcal{I}_\tau \definedas (M^{3},g,\tau K, \mu_\tau, \tau J)$ is as described in Section \ref{ssecStrategy}, arriving at the foliation of the initial data set $\mathcal{I}_1$ by constant spacetime mean curvature surfaces $\{\Sigma^\sigma_1\}_{\sigma>\sigma_{\mathcal{I}_{1}}}$. In order to make this idea more precise, we introduce the following construction.

By Theorem \ref{thExistenceCMC}, we know that for every $\sigma> \sigma_{\mathcal{I}_{0}}$ there is a closed, oriented $2$-surface $\Sigma^\sigma_0\hookrightarrow M^{3}$ with constant spacetime mean curvature $\mathcal{H} (\Sigma^\sigma_0)  \equiv \sfrac{2}{\sigma}$ with respect to the initial data set~$\mathcal{I}_0$. Furthermore, the proof of this result  in \cite{NerzCMC} shows that there are constants $b_{\mathcal{I}_{0}}\geq 0$ and $1\geq\eta_{\mathcal{I}_{0}}>0$ such that $\Sigma^\sigma_0 \in \mathcal{A}(0,b_{\mathcal{I}_{0}},\eta_{\mathcal{I}_{0}})$ for all $\sigma>\sigma_{\mathcal{I}_{0}}$. We recall from \cite{NerzCMC} that $b_{\mathcal{I}_{0}}$ and $\eta_{\mathcal{I}_{0}}$ only depend on $\varepsilon$, $C_{\mathcal{I}_{0}}$, and $E$ which can be restated as saying that they only depend on $\varepsilon$, $C_{\mathcal{I}_{1}}$, and $E$ by our construction. Set $b_{\mathcal{I}_{1}}\definedas 4b_{\mathcal{I}_{0}}>b_{\mathcal{I}_{0}}$ and $\eta_{\mathcal{I}_{0}}>\eta_{\mathcal{I}_{1}}\definedas\tfrac{\eta_{\mathcal{I}_{0}}}{4}>0$.  From Section \ref{secLinearization} and by the definition of $b_{\mathcal{I}_{1}}$ and $\eta_{\mathcal{I}_{1}}$, we know that there are constants $C$ and $\overline{\sigma}$ depending only on $\varepsilon$, $C_{\mathcal{I}_{1}}$, and $E$ such that the operator $\mathcal{L}\colon W^{2,2}(\Sigma) \to L^2(\Sigma)$ is invertible whenever $\Sigma\in \mathcal{A}(0,b_{\mathcal{I}_{1}},\eta_{\mathcal{I}_{1}})$ is a surface of constant spacetime mean curvature $\mathcal{H}(\Sigma) \equiv \sfrac{2}{\sigma}$ with respect to the initial data set $\mathcal{I}_1$ for $\sigma\geq \overline{\sigma}$, and whenever in addition the estimates of Proposition \ref{propInvertibility} and Corollary \ref{corEllipticRegularity} are available on $\Sigma$. Without loss of generality, we may also assume that $C$ and $\overline{\sigma}$ are such that the regularity result in Proposition \ref{propRegularity} as well as a supplementary result stated in Lemma \ref{lemConstEvolution} (see Section \ref{secSupplementary} below) apply with $a=\overline{a}=0$, $b=\tfrac{b_{\mathcal{I}_{1}}}{2}$, $\overline{b}=b_{\mathcal{I}_{1}}$, $\eta=2\eta_{\mathcal{I}_{1}}$, $\overline{\eta}=\eta_{\mathcal{I}_{1}}$.  We set $\sigma_{\mathcal{I}_1}\definedas\max\{\overline{\sigma},4\sigma_{\mathcal{I}_0}\}$, and note that by their definition $\sigma_{\mathcal{I}_{1}}$, $b_{\mathcal{I}_{1}}$, and $\eta_{\mathcal{I}_{1}}$ only depend on $\varepsilon$, $C_{\mathcal{I}_{1}}$, and $E$.

Now fix $\sigma_{*}>\sigma_{\mathcal{I}_{1}}$ for the rest of the argument until we start discussing the foliation property when applying Lemma~\ref{lemFoliation}. Let $Y^{\sigma_{*}} \subseteq [0,1]$ be the \emph{maximal} subset such that there is a $C^1$-map 
\begin{align*}
\mathcal{F}^{\sigma_{*}}\colon  Y^{\sigma_{*}} \times \mathbb{S}^2 \to M^{3}
\end{align*} 
with the following properties for every $\tau \in Y^{\sigma_{*}}$: 
\begin{enumerate}[(i)]
\item \label{cond STCMC} The surface $\Sigma^{\sigma_{*}}_\tau \definedas \mathcal{F}^{\sigma_{*}}(\tau, \mathbb{S}^2)$  has constant spacetime mean curvature $\mathcal{H} (\Sigma^{\sigma_{*}}_\tau)  \equiv \sfrac{2}{\sigma_{*}}$ with respect to the initial data set $\mathcal{I}_\tau$.\\[-1ex]
\item\label{cond ortho} $\partial_\tau \mathcal{F}^{\sigma_{*}}(\tau, q)$ is orthogonal to $\Sigma^{\sigma_{*}}_{\tau}$ for every $q\in \mathbb{S}^2$.\\[-1ex]
\item\label{cond centered} The surface $\Sigma^{\sigma_{*}}_{\tau}$ is asymptotically centered in the sense that $\Sigma^{\sigma_{*}}_\tau \in \mathcal{A}(0,b_{\mathcal{I}_{1}},\eta_{\mathcal{I}_{1}})$.\\[-1ex]
\end{enumerate}
\emph{Maximality} of $Y^{\sigma_{*}}$ is understood here as follows: if the above conditions are satisfied for some $\widetilde{Y}^{\sigma_{*}}\subseteq [0,1]$ and a map $\widetilde{\mathcal{F}}^{\sigma_{*}}\colon \widetilde{Y}^{\sigma_{*}} \times \mathbb{S}^2 \to M^3$, then $\widetilde{Y}^{\sigma_{*}} \subseteq Y^{\sigma_{*}}$ as well as $\mathcal{F}^{\sigma_{*}}\vert_{\widetilde{Y}^{\sigma^{*}}}\equiv \widetilde{\mathcal{F}}^{\sigma_{*}}$. 

Note that for $\tau=0$, Condition \eqref{cond STCMC} is ensured by the assumptions in Theorem~\ref{thExistenceCMC}. The same is true for Condition \eqref{cond centered} once one takes into account that $\mathcal{A}(0,b_{\mathcal{I}_{0}},\eta_{\mathcal{I}_{0}})\subseteq \mathcal{A}(0,b_{\mathcal{I}_{1}},\eta_{\mathcal{I}_{1}})$. However, Condition \eqref{cond ortho} is not automatically satisfied for $\tau=0$ as we do not even know whether the map $\mathcal{F}^{\sigma_{*}}$ exists. The following lemma ensures that $Y^{\sigma_{*}}$ contains an interval $[0,\tau_0)$ for some $\tau_0>0$. In particular, Condition \eqref{cond ortho} is satisfied a posteriori for $\tau=0$. More generally, this result shows that $Y^{\sigma_{*}}$ is open around any $\tau_{*}\in Y^{\sigma_{*}}$ such that $\Sigma^{\sigma_{*}}_{\tau_{*}}\in \mathcal{A}(0,b,\eta)$ for $0\leq b<b_{\mathcal{I}_{1}}$ and $\eta_{\mathcal{I}_{1}}<\eta \leq 1$. 

\begin{lemma}\label{lemStatement}
Let $0\leq b<\overline{b}\leq b_{\mathcal{I}_1}$ and $\eta_{\mathcal{I}_1}\leq\overline{\eta}<\eta\leq 1$. For any $\tau_{*}\in[0,1]$ for which there exists a smooth surface $\Sigma^{\sigma_{*}}_{\tau_{*}}\in\mathcal{A}(0,b,\eta)$ satisfying $\mathcal{H}(\Sigma^{\sigma_{*}}_{\tau_{*}})\equiv \sfrac{2}{\sigma_{*}}$, there exists an open, connected neighborhood $U_{\tau_{*}}$ of $\tau_{*}$ inside $[0,1]$ and a unique $C^{1}$-map $\mathcal{F}^{\sigma_{*}}\colon U_{\tau_{*}}\times\mathbb{S}^{2}\to M^3$ with $\Sigma^{\sigma_{*}}_{\tau_{*}}=\mathcal{F}^{\sigma_{*}}(\tau_{*},\cdot)$ such that \eqref{cond STCMC} and \eqref{cond ortho} are satisfied for $\Sigma^{\sigma_{*}}_{\tau}\definedas \mathcal{F}^{\sigma_{*}}(\tau,\cdot)$, and such that $\Sigma^{\sigma_{*}}_{\tau}\in\mathcal{A}(0,\overline{b},\overline{\eta})$ for all $\tau\in U_{\tau_{*}}$. 
\end{lemma}

\begin{proof} In order to prove this lemma, suppose that $\tau_* \in [0,1]$, and that $b$ and $\eta$ are such as in the statement. As discussed in Section \ref{subsecStab}, in a neighborhood of each $\Sigma^{\sigma_{*}}_{\tau_*}$, we may introduce normal geodesic coordinates $y\colon\Sigma^{\sigma_{*}}_{\tau_*} \times (-\xi,\xi) \to M^{3}$ near $\Sigma^{\sigma_{*}}_{\tau_*}$. Now let $U^{2,2}_\xi (\Sigma^{\sigma_{*}}_{\tau_*})\subseteq W^{2,2}(\Sigma^{\sigma_{*}}_{\tau_*})$ be an open neighborhood of $0\in W^{2,2}(\Sigma^{\sigma_{*}}_{\tau_*})$ such that $f \in U^{2,2}_\xi(\Sigma^{\sigma_{*}}_{\tau_{*}})$ implies $\vert f\vert <\xi$; such a neighborhood exists by Lemma~\ref{lemSobolevEmbed}.

Next, we consider the graphical spacetime mean curvature map 
\begin{align*}
h^{\sigma_{*}}\colon [0,1] \times U_\xi^{2,2} (\Sigma^{\sigma_{*}}_{\tau_*}) \to L^2(\Sigma^{\sigma_{*}}_{\tau_*})
\end{align*}
which assigns, to every $\tau\in [0,1]$ and every $f\in U_\xi^{2,2} (\Sigma^{\sigma_{*}}_{\tau_*})$, the spacetime mean curvature $\mathcal{H}(\graph f)$ of the geodesic graph, $\graph f= \{y(q,f(q)) : q \in \Sigma^{\sigma_{*}}_{\tau_*}\}$, with respect to the initial data set $\mathcal{I}_\tau$. The Fr\'echet derivative of the map $h^{\sigma_{*}}$ with respect to the second argument $f$ at the point $(\tau_*,0)$ is the operator $L^{\mathcal{H}}\colon W^{2,2} (\Sigma^{\sigma_{*}}_{\tau_*}) \to L^2(\Sigma^{\sigma_{*}}_{\tau_*})$ given by Lemma \ref{lemLinearizationCMC}, where all geometric quantities are computed with respect to the initial data set $\mathcal{I}_{\tau_*}$.
As shown in Section \ref{secLinearization}, the linearized operator $L^{\mathcal{H}}\colon W^{2,2} (\Sigma^{\sigma_{*}}_{\tau_*}) \to L^2(\Sigma^{\sigma_{*}}_{\tau_*})$ is continuously invertible, because $\sigma_{*}>\sigma_{\mathcal{I}_{1}}$.

By the Implicit Function Theorem, there thus exists a relatively open neighborhood $\widetilde{U}_{\tau_*}\subseteq[0,1]$ of $\tau_*$ and a unique $C^1$-map $\gamma^{\sigma_{*}}\colon \widetilde{U}_{\tau_*} \to U_\xi^{2,2} (\Sigma^{\sigma_{*}}_{\tau_*})$ such that $\gamma^{\sigma_{*}}(\tau_*)= 0$ and 
\begin{align*}
h^{\sigma_{*}}(\tau,\gamma^{\sigma_{*}}(\tau)) = h^{\sigma_{*}}(\tau_*, \gamma^{\sigma_{*}}(\tau_*))\equiv\sfrac{2}{{\sigma_{*}}} \quad \text{ for all } \tau \in \widetilde{U}_{\tau_*}.
\end{align*}
Thus, by construction, for every $\tau\in \widetilde{U}_{\tau_*}$, the surface $\Sigma^{\sigma_{*}}_\tau = \graph \gamma^{\sigma_{*}}(\tau)$ has constant spacetime mean curvature $\sfrac{2}{{\sigma_{*}}}$ with respect to the initial data set $\mathcal{I}_\tau$.

Recall that the surface $\Sigma^{\sigma_{*}}_{\tau_{*}}$ is a graph over some round sphere by our assumptions and by Proposition~\ref{propRegularity}, recalling again the a priori bounds on $\sigma_{*}$, $b$, and $\eta$. As $\Sigma^{\sigma_{*}}_{\tau}$ was defined as a graph over $\Sigma^{\sigma_{*}}_{\tau_{*}}$ for every $\tau\in \widetilde{U}_{\tau_{*}}$, composition of these two graphical representations gives us that $\Sigma^{\sigma_{*}}_{\tau}$ is parametrized over a round sphere.

Thus, we may now define the map $\mathcal{F}^{\sigma_{*}}\colon \widetilde{U}_{\tau_*} \times \mathbb{S}^2 \to M^3$ by setting $\mathcal{F}^{\sigma_{*}} (\tau,\mathbb{S}^2)\definedas\Sigma^{\sigma_{*}}_\tau$, and defining the parametrization of $\Sigma^{\sigma_{*}}_\tau$ by requesting that $\partial_\tau \mathcal{F}^{\sigma_{*}}$ be orthogonal to $\Sigma^\tau_{\sigma_{*}}$ for all $\tau\in \widetilde{U}_{\tau_*}$. 

By continuity of $\mathcal{F}^{\sigma_{*}}$ and because $0\leq b<\overline{b}$, $0<\overline{\eta}<\eta\leq 1$, and $\Sigma^{\sigma_{*}}_{\tau_{*}}\in\mathcal{A}(0,b,\eta)$, there exists an open neighborhood $U_{\tau_{*}}\subseteq\widetilde{U}_{\tau_{*}}$ of $\tau_{*}$ such that
\begin{align}\label{eqContMethod}
\Sigma^{\sigma_{*}}_\tau \in \mathcal{A} \left(0, \overline{b}, \overline{\eta} \right)
\end{align}
holds for all $\tau\in U_{\tau_{*}}$ as desired. This proves Lemma \ref{lemStatement}.
\end{proof}

Choosing $b=b_{\mathcal{I}_{0}}$, $\overline{b}=2b_{\mathcal{I}_{0}}=\tfrac{b_{\mathcal{I}_{1}}}{2}$, and $\eta=\eta_{\mathcal{I}_{0}}$, $\overline{\eta}=\tfrac{\eta_{\mathcal{I}_{0}}}{2}=2 \eta_{\mathcal{I}_{1}}$, Lemma \ref{lemStatement} shows directly via Theorem \ref{thExistenceCMC} that $0\in Y^{\sigma_*}$ and that $Y^{\sigma_*}$ is relatively open near $0$. Now we let $X^{\sigma_*}$ be the \emph{maximal} connected subinterval of $Y^{\sigma_*}$ containing $\tau=0$. As we have just seen by Lemma \ref{lemStatement}, $X^{\sigma_*}$ is relatively open near $\tau=0$. Set $\tau^* \definedas \sup X^{\sigma_*}$. In Lemma~\ref{lemClosedness} below we will show that $\tau^* \in X^{\sigma_*}$, so that $X^{\sigma_*} = [0,\tau^*]$ is closed, where $0<\tau^*\leq 1$.

\begin{lemma}\label{lemClosedness}
The interval $Y^{\sigma_*} \subseteq [0,1]$ is closed.
\end{lemma}
\begin{proof}
Closedness of $Y^{\sigma_*}$ can be addressed by following the arguments given in \cite[Lemma~5.6]{NerzCMC} and \cite[Lemma 3.14]{NerzCE}, as the necessary preliminaries are available in the form of Lemma \ref{lemLapse} and Lemma \ref{lemEstimates} below. Alternatively, one may rely on a more standard method used in \cite[Proof of Proposition 6.1]{MetzgerCMC}, which we describe below. The Sobolev spaces we use throughout the paper are weighted, however, for a given closed, oriented $2$-surface, the weighted Sobolev norms are equivalent to the traditional unweighted ones; we will thus switch to the usual unweighted ones for this proof in order to allow us to use standard results on Sobolev spaces on $2$-surfaces.

Let $\lbrace\tau_n\rbrace_{n=1}^{\infty} \subset Y^{\sigma_*}$ be a sequence such that $\lim_{n\to \infty} \tau_{n}=\tau\in[0,1]$ and let $\Sigma^{\sigma_*}_{\tau_n} \in \mathcal{A}(0,b_{\mathcal{I}_1},\eta_{\mathcal{I}_1})$ be a surface with constant spacetime mean curvature $\mathcal{H} (\Sigma^{\sigma_{*}}_{\tau_n})  \equiv \sfrac{2}{\sigma_{*}}$ with respect to the initial data set $\mathcal{I}_{\tau_n}$. By Proposition \ref{propRegularity} we know that there are functions $f_n\colon\mathbb{S}
_{r_n}(\vec{z}_n) \to \mathbb{R}$ such that $\Sigma^{\sigma_*}_{\tau_n}=\graph f_n$ where $r_n$ and $\vec{z}_n$ are the area radius and the coordinate center of $\Sigma^{\sigma_{*}}_{\tau_n}$. By the first inequality of \eqref{eqDefConcentric} and by \eqref{eqMCArea}, we know that the sequences $\{r_n\}_{n=1}^\infty$ and $\{\vec{z}_n\}_{n=1}^\infty$ are uniformly bounded, so we may assume (up to passing to a subsequence) that $\lim_{n\to\infty}r_n = r$ and $\lim_{n\to\infty} \vec{z}_n = \vec z$. Consequently, in view of \eqref{eqGraphFunction}, we may assume that there is a  sequence $\{\tilde{f}_n\}_{n=1}^\infty$, such that $\tilde{f}_n\colon\mathbb{S}
_{r}(\vec{z}\,) \to \mathbb{R}$ and $\Sigma^{\sigma_*}_{\tau_n}=\graph \tilde{f}_n$. Again in the view of \eqref{eqGraphFunction}, we may assume that this sequence is uniformly bounded in $W^{2,\infty}(\mathbb{S}^2_r(\vec{z}\,))$ and hence in $C^{1,\beta}(\mathbb{S}^2_r(\vec{z}\,))$ for any $0<\beta<1$. Recalling that $\Sigma^{\sigma_*}_{\tau_n}$ are surfaces of constant spacetime mean curvature, we see that the functions $\tilde{f}_n$ satisfy a linear elliptic PDE of the form
\begin{align}\label{eqPrescribedSTCMC}
\sum_{\beta,\gamma=1}^2 a_n^{\beta\gamma}\partial_\beta \partial_\gamma \tilde{f}_n + \sum_{\beta=1}^2 b_n^\beta \partial_\beta \tilde{f}_n= F_n,
\end{align}
with uniformly bounded coefficients $a_n^{\beta\gamma}, b_n^\beta, F_n \in C^{0,\beta}(\mathbb{S}^2_r(\vec{z}\,))$, see Appendix \ref{apFermi} for details. A standard argument using Schauder estimates (see e.g.~\cite[Theorem 9.19]{GilbargTrudinger} and \cite[Theorem 10.2.1]{JostPDE}) allows us to conclude that the functions $\tilde{f}_n \in C^{2,\beta}(\mathbb{S}^2_r(\vec{z}\,))$ are uniformly bounded in $C^{2,\beta}(\mathbb{S}^2_r(\vec{z}\,))$, and consequently, up to passing to a subsequence, we may assume that $\{\tilde{f}_n\}_{n=1}^\infty$ converges in $C^{2,\alpha}(\mathbb{S}^2_r(\vec{z}\,))$  to a limit $f\in C^{2,\alpha}(\mathbb{S}^2_r(\vec{z}\,))$ for some fixed $0<\alpha<1$. As a consequence of \eqref{eqPrescribedSTCMC} and $C^{2,\alpha}$-convergence, we see that $\Sigma^{\sigma_{*}}_{\tau} \definedas \graph f$ has constant spacetime mean curvature $\mathcal{H} (\Sigma^{\sigma_{*}}_{\tau})\equiv \sfrac{2}{\sigma^*}$.

Finally, we confirm that $\Sigma^{\sigma_{*}}_\tau = \graph f \in \mathcal{A}(0,b_{\mathcal{I}_{1}},\eta_{\mathcal{I}_{1}})$ by passing to the limit in the respective inequalities of \eqref{eqDefConcentric} for $\Sigma^{\sigma_*}_{\tau_n}=\graph \tilde{f}_n\in \mathcal{A}(0,b_{\mathcal{I}_{1}},\eta_{\mathcal{I}_{1}})$. Again, this is possible in the view of the $C^{2,\alpha}$-convergence of the graph functions.
\end{proof}

Continuing the proof of Theorem \ref{thExistence}, we will now use Lemma \ref{lemConstEvolution} below to show that $\tau^*=1$, arguing by contradiction. Suppose instead that $\tau^* < 1$. Then $\Sigma^{\sigma_{*}}_{\tau} \in \mathcal{A}(0,b_{\mathcal{I}_1},\eta_{\mathcal{I}_1})$ for all $\tau\in[0,\tau^*]$, whereas $\Sigma^{\sigma_{*}}_{0} \in \mathcal{A}(0,b_{\mathcal{I}_0},\eta_{\mathcal{I}_0})$. Applying Lemma \ref {lemConstEvolution}, we see that in fact $\Sigma^{\sigma_{*}}_{\tau} \in \mathcal{A}(0,\tfrac{b_{\mathcal{I}_1}}{2}, 2\eta_{\mathcal{I}_1})$. As a consequence, we may apply Lemma~\ref{lemStatement}
 with $b=\tfrac{b_{\mathcal{I}_1}}{2}$, $\overline{b}=b_{\mathcal{I}_1}$, $\eta=2\eta_{\mathcal{I}_1}$, and $\overline{\eta}=\eta_{\mathcal{I}_1}$ to show that $[0,\tau^*+\rho)\subseteq Y^{\sigma_*}$ for some $\rho>0$. This contradicts the maximality of the intervals $X^{\sigma_*}$, hence $\tau^* =1$ and $Y^{\sigma_*}=[0,1]$.

Summing up, we have shown that for each $\sigma>\sigma_{\mathcal{I}_{1}}$ there is a surface $\Sigma^{\sigma}_1 = \mathcal{F}^{\sigma} (1,\mathbb{S}^2)$ such that its spacetime mean curvature in the initial data set $\mathcal{I}_1$ is $\mathcal{H}(\Sigma^{\sigma}_{1})\equiv\sfrac{2}{\sigma}$. We may now define $\Psi_1\colon (\sigma_{\mathcal{I}_{1}}, \infty) \times \mathbb{S}^2 \to M^{3}$ by setting 
\begin{align}\label{eqDefPsi}
\Psi_1(\sigma,\cdot)\definedas\mathcal{F}^\sigma (1,\cdot).
\end{align}
The only remaining thing to check is that the family $\{\Sigma^\sigma_1\}_{\sigma>\sigma_{\mathcal{I}_{1}}}$ is a foliation, which will be the case if $\Psi_1$ is a bijective $C^{1}$-map onto the exterior region $M^3\setminus\mathcal{K}_{1}$ of a suitably large compact set $\mathcal{B}\subseteq\mathcal{K}_{1}\subset M^{3}$. This is proven in Lemma~\ref{lemFoliation}. Note that in this step, we may need to increase $\sigma_{\mathcal{I}_1}$, albeit without introducing new dependencies.
\end{proof}

\subsection{Supplementary lemmas}\label{secSupplementary}
We will now prove the supplementary lemmas that were used in the proof of Theorem \ref{thExistence} above.

\begin{lemma}\label{lemLapse}
 Let $\mathcal{I}=\IDS$ be a $C^{2}_{\sfrac{1}{2}+\varepsilon}$-asymptotically Euclidean initial data set with non-vanishing energy $E\neq0$, with $\vec{x}\colon M^{3}\setminus\mathcal{B}\to\R^{3}\setminus\overline{B_{R}(0)}$ denoting the asymptotic coordinate chart. Assume in addition that $K$ satisfies the potentially stronger decay assumptions $\vert K\vert\leq C_{\mathcal{I}}\,\vert \vec{x}\,\vert^{-\delta-\varepsilon}$ for  some $\delta\geq\tfrac{3}{2}$ and all $\vec{x}\in\R^{3}\setminus\overline{B_{R}(0)}$.
 
Let $\emptyset\neq U\subseteq[0,1]$ be an open subset of $[0,1]$ and define $\mathcal{I}_{\tau}$ as in the proof of Theorem \ref{thExistence} for each $\tau\in U$. Let $a\in[0,1)$, $b\geq0$, $\eta\in(0,1]$ be fixed. Then there exist constants $\overline{\sigma}>0$ and $C>0$, depending only on $\varepsilon$, $\delta$, $a$, $b$, $\eta$, $C_{\mathcal{I}}$, and $E $ such that the following holds for any $\sigma>\overline{\sigma}$: Assume there exists a $C^{1}$-map $\mathcal{F}^{\sigma}\colon U \times \mathbb{S}^2\to M^3$ such that for every $\tau \in U$ the surface $\Sigma^{\sigma}_\tau \definedas \mathcal{F}^{\sigma} (\tau, \mathbb{S}^2)$ is in the a priori class $\mathcal{A}(a,b,\eta)$ and has constant spacetime mean curvature $\mathcal{H}(\Sigma^{\sigma}_\tau)\equiv\sfrac{2}{\sigma}$ with respect to the initial data set $\mathcal{I}_\tau$. Assume further that $\mathcal{F}^{\sigma}$ is a \emph{normal variation map} in the sense that there exists a continuous lapse function $u=u_\tau^{\sigma}\colon\Sigma^{\sigma}_{\tau}\to\R$ such that $\partial_\tau \mathcal{F}^{\sigma} = u\, \nu$, where $\nu = \nu^{\sigma}_\tau$ is the unit normal to $\Sigma^{\sigma}_\tau$ in $(M^3,g)$. Then we have
\begin{align}\label{eqLapse}
 \Vert u\Vert _{W^{2,2}(\Sigma_\tau^{\sigma})} \leq C\sigma^{5-2\delta- 2\varepsilon} , \qquad \Vert u^d\Vert _{W^{2,2}(\Sigma_\tau^{\sigma})} \leq C\sigma^{\frac{9}{2}-2\delta - 3\varepsilon},
\end{align}
and $\mathcal{L} u = O(\sigma^{1-2\delta-2\varepsilon})$.
\end{lemma}

\begin{proof}     
In this proof, $C>0$ and $\overline{\sigma}>0$ denote generic constants that may vary from line to line, but depend only on $\varepsilon$, $\delta$, $a$, $b$, $\eta$, $C_{\mathcal{I}}$, and $E $. The surfaces $\Sigma_\tau^\sigma$ have constant spacetime mean curvature $\mathcal{H}(\Sigma^\sigma_{\tau}) \equiv \sfrac{2}{\sigma}$ in the initial data set $\mathcal{I}_\tau$. For clarity, we will write this constant spacetime mean curvature with an explicit reference to the initial data set $\mathcal{I}_{\tau}$ as $\mathcal{H}(\Sigma^{\sigma}_{\tau},\mathcal{I}_{\tau})\equiv\sfrac{2}{\sigma}$ for all $\tau\in U$. Hence
\begin{align*}
\partial_\tau \mathcal{H}(\Sigma_\tau^\sigma,\mathcal{I}_\tau) = 0,
\end{align*}
which gives us the following linear elliptic PDE on the closed surface $\Sigma^{\sigma}_{\tau}$ for the a priori only continuous lapse function $u=u^{\sigma}_{\tau}\colon\Sigma^{\sigma}_{\tau}\to\R$: 
\begin{align}\label{eqPDEu}
\mathcal{L}u&= \frac{\tau\,(\tr_{\Sigma_\tau^\sigma} K)^2}{H(\Sigma^\tau_\sigma)},
\end{align}
where the elliptic operator $\mathcal{L}$ is (up to a certain factor) the linearization of the spacetime mean curvature operator for the surface $\Sigma^\sigma_\tau$ in the initial data set $\mathcal{I}_\tau$, defined in~\eqref{eqLinearization}. Then Proposition \ref{propInvertibility} implies that $u\in W^{2,2}(\Sigma^{\sigma}_{\tau})$, and that such a $u$ is unique. Together with \eqref{eqFallOff} and $P=O(\sigma^{-\delta-\varepsilon})$, \eqref{eqPDEu} implies that
\begin{align}\label{eqLinearizationSTCMC}
\begin{split}
\mathcal{L} u  = O(\sigma^{1-2\delta-2\varepsilon}).
\end{split}
\end{align}
As a consequence, by Corollary \ref{corEllipticRegularity} and \eqref{eqTranslational0}, we get 
\begin{align}\label{eqUD}
\begin{split}
\Vert u^d\Vert _{W^{2,2}(\Sigma_\tau^\sigma)} 
& \leq C\sigma^2 \Vert \mathcal{L} u^d\Vert _{L^2(\Sigma_\tau^\sigma)} \\
& \leq C\sigma^2 \left(\Vert \mathcal{L} u\Vert _{L^2(\Sigma_\tau^\sigma)} + \Vert \mathcal{L} u^t\Vert _{L^2(\Sigma_\tau^\sigma)}\right) \\
& \leq C\sigma^2 \left(\sigma^{2-2\delta-2\varepsilon} + \Vert \mathcal{L} u^t\Vert _{L^2(\Sigma_\tau^\sigma)}\right) \\
& \leq C\sigma^{4-2\delta-2\varepsilon} + \sigma^{-\frac{1}{2}-\varepsilon}\Vert u^t\Vert _{L^2(\Sigma_\tau^\sigma)}.                   
\end{split}
\end{align}
In order to estimate $\Vert u^t\Vert _{L^2(\Sigma_\tau^\sigma)}$, note that by \eqref{eqTranslational} we have for $i=1,2,3$
\begin{align*}
\begin{split}
&\left\vert \int_{\Sigma_\tau^\sigma} u f_i \, d\mu - \frac{\sigma^3}{6m_\mathcal{H}}\int_{\Sigma_\tau^\sigma} u \mathcal{L}f_i \, d\mu \right\vert  \\
& \leq \left\vert \int_{\Sigma_\tau^\sigma} u^t f_i \, d\mu - \frac{\sigma^3}{6m_\mathcal{H}}\int_{\Sigma_\tau^\sigma} u^t \mathcal{L}f_i \, d\mu \right\vert 
+ \frac{\sigma^3}{6\vert m_\mathcal{H}\vert } \left\vert \int_{\Sigma_\tau^\sigma}  u^d  \mathcal{L} f_i  \, d\mu \right\vert \\
& \leq C \sigma^{-\varepsilon} \Vert u^t\Vert _{L^2(\Sigma_\tau^\sigma)} + \frac{\sigma^3}{6\vert m_\mathcal{H}\vert } \left\vert \int_{\Sigma_\tau^\sigma}  u^d  \mathcal{L} f_i \, d\mu \right\vert ,
\end{split}
\end{align*}
where we again use $m_{\mathcal{H}}=m_{\mathcal{H}}(\Sigma^{\sigma}_{\tau})$ as an abbreviation for the Hawking mass and the fact that $m_{\mathcal{H}}\neq0$ as $E\neq0$ and $\vert m_{\mathcal{H}}-E\vert\leq C\sigma^{-\varepsilon}$ by Proposition \ref{propInvertibility}. By \eqref{eqforLemmaLapse}, we have 
\begin{align*}
\left\vert \int_{\Sigma} u^d \mathcal{L} f_i \, d\mu \right\vert  &\leq \frac{C}{\sigma^{\frac{5}{2}+\varepsilon}}\Vert u^{d}\Vert_{L^{2}(\Sigma)}.
\end{align*}
Using the Cauchy--Schwarz Inequality, integration by parts, and \eqref{eqEigenEstimates2}, together with $f_{i}\approx \sqrt{\tfrac{3}{4\pi r^{4}}}\,x^{i}$, \eqref{eqMCArea}, and \eqref{eqLinearizationSTCMC} we obtain
\begin{align*}
\begin{split}
\left\vert  \int_{\Sigma_\tau^\sigma} u \mathcal{L} f_i \, d\mu \right\vert  & \leq \left\vert  \int_{\Sigma_\tau^\sigma} f_i \mathcal{L} u \, d\mu \right\vert + 2\left\vert  \int_{\Sigma_\tau^\sigma} \tfrac{P}{H} K \left(u \nabla^{\Sigma_\tau^\sigma} f_i - f_i \nabla^{\Sigma_\tau^\sigma} u, \nu \right)\, d\mu \right\vert   \\ &\leq  C\sigma^{2-2\delta-2\varepsilon} + C\sigma^{-2\delta-2\varepsilon}\Vert u\Vert _{L^2(\Sigma_\tau^\sigma)}.
\end{split}
\end{align*}
Combining the last three estimates with a Triangle Inequality, it follows that
\begin{align*}
\Vert u^t\Vert _{L^2(\Sigma_\tau^\sigma)} \leq&\, \frac{\sigma^3}{6\vert m_\mathcal{H}\vert}\left\vert \int_{\Sigma_\tau^\sigma} u \mathcal{L}f_i \, d\mu \right\vert + C \sigma^{-\varepsilon} \Vert u^t\Vert _{L^2(\Sigma_\tau^\sigma)} + \frac{\sigma^3}{6\vert m_\mathcal{H}\vert } \left\vert \int_{\Sigma_\tau^\sigma}  u^d  \mathcal{L} f_i \, d\mu \right\vert \\
\leq& \,C \sigma^{5-2\delta-2\varepsilon} + C\sigma^{-\varepsilon}\Vert u^{t}\Vert _{L^2(\Sigma_\tau^\sigma)} +C \sigma^{\frac{1}{2}-\varepsilon} \Vert u^d\Vert _{L^2(\Sigma_\tau^\sigma)},
\end{align*}
so that
\begin{align*}
\Vert u^t\Vert _{L^2(\Sigma_\tau^\sigma)} &\leq C \sigma^{5-2\delta-2\varepsilon} +C \sigma^{\frac{1}{2}-\varepsilon} \Vert u^d\Vert _{L^2(\Sigma_\tau^\sigma)}.
\end{align*}
Recalling \eqref{eqUD}, we now get a $W^{2,2}$-estimate for $u^{d}$, namely
\begin{align*}
\Vert u^d\Vert _{W^{2,2}(\Sigma_\tau^\sigma)} \leq C \sigma^{\frac{9}{2}-2\delta-3\varepsilon}.
\end{align*}
From this, as a consequence of \eqref{eqLinearizationSTCMC} and  Corollary \ref{corEllipticRegularity}, we also have 
\begin{align*}
\Vert u^t\Vert _{W^{2,2}(\Sigma_\tau^\sigma)} &\leq \Vert u^d\Vert _{W^{2,2}(\Sigma_\tau^\sigma)} + \Vert u\Vert _{W^{2,2}(\Sigma_\tau^\sigma)}\\
& \leq \Vert u^d\Vert _{W^{2,2}(\Sigma_\tau^\sigma)} + C\sigma^3 \Vert \mathcal{L} u\Vert _{L^{2}(\Sigma_\tau^\sigma)}\leq C \sigma^{5-2\delta-2\varepsilon}.
\end{align*}
\end{proof}

\bigskip

Lemma \ref{lemLapse} enables us to prove the following result.

\begin{lemma}\label{lemEstimates}
Let $\mathcal{I}=\IDS$ be a $C^{2}_{\sfrac{1}{2}+\varepsilon}$-asymptotically Euclidean initial data set with non-vanishing energy $E\neq0$, with $\vec{x}$ denoting the asymptotic coordinate chart. Let $\emptyset\neq U\subseteq[0,1]$ be an open subset of $[0,1]$ and define $\mathcal{I}_{\tau}$ as in the proof of Theorem~\ref{thExistence} for each $\tau\in U$. Let $a\in[0,1)$, $b\geq0$, $\eta\in(0,1]$ be fixed. Then there exist constants $\overline{\sigma}>0$ and $C>0$, depending only on $\varepsilon$, $a$, $b$, $\eta$, $C_{\mathcal{I}}$, and $E $ such that the following holds for any $\sigma>\overline{\sigma}$: Assume there exists a $C^{1}$-map $\mathcal{F}^{\sigma}\colon U \times \mathbb{S}^2\to M^3$ such that for every $\tau \in U$ the surface $\Sigma^{\sigma}_\tau \definedas \mathcal{F}^{\sigma} (\tau, \mathbb{S}^2)$ is in the a priori class $\mathcal{A}(a,b,\eta)$ and has constant spacetime mean curvature $\mathcal{H}(\Sigma^{\sigma}_\tau)\equiv\sfrac{2}{\sigma}$ with respect to the initial data set $\mathcal{I}_\tau$. Assume further that $\mathcal{F}^{\sigma}$ is a normal variation map in the sense explained in Lemma \ref{lemLapse}. Then
\begin{align}
\left\vert \partial_\tau \left(\vert\vec{x}\,\vert  \circ \mathcal{F}^\sigma \right)\right\vert  \leq C \sigma^{1-\varepsilon},\label{eqDerCoord}  \\ \left\vert \partial_\tau \vert \Sigma^ \sigma_\tau\vert \right\vert  \leq C\sigma^{\frac{3}{2} - \varepsilon}, \label{eqDerArea} \\  |\partial_\tau (\vec{z} \circ \mathcal{F}^\sigma) | = O(\sigma^{1-2\varepsilon}).\label{eqDerCenter}
\end{align}
\end{lemma}

\begin{proof}
In this proof, $C>0$ and $\overline{\sigma}>0$ denote generic constants that may vary from line to line, but depend only on $\varepsilon$, $a$, $b$, $\eta$, and $C_{\mathcal{I}}$, and $E $. Let $u\colon\Sigma^{\sigma}_{\tau}\to\R$ denote the lapse function as in Lemma~\ref{lemLapse}. Then Lemma \ref{lemLapse} applied with $\delta=\tfrac{3}{2}$ and the Sobolev Embedding Theorem in the form of Lemma \ref{lemSobolevEmbed} imply that $\vert \partial_\tau \mathcal{F}^\sigma\vert =\vert u\vert  \leq C \sigma^{1-\varepsilon} $. Then
 the elementary estimate
\begin{align*}
 \left\vert \partial_\tau \left(\vert\vec{x}\,\vert  \circ \mathcal{F}^\sigma\right)\right\vert  =  \frac{\left\vert \sum_{i=1}^3 (x^i\circ\mathcal{F}^\sigma) (\partial_\tau \mathcal{F}^\sigma)\right\vert }{\vert\vec{x}\,\vert  \circ \mathcal{F}^\sigma} \leq C \sigma^{1-\varepsilon}
\end{align*}
proves \eqref{eqDerCoord}.

In order to prove \eqref{eqDerArea}, we first recall that the mean curvature of $\Sigma^\sigma_\tau$ satisfies $H = \frac{2}{\sigma} + O(\sigma^{-2-\varepsilon})$, see \eqref{eqFallOff}. The first variation of area formula, the fact that the eigenfunctions used to span $L^{2}(\Sigma^{\sigma}_{\tau})$ are $L^{2}(\Sigma^{\sigma}_{\tau})$-orthogonal so that in particular $\int_{\Sigma^\sigma_\tau} u^t \, d\mu=0$, combined with Lemma~\ref{lemLapse} for $\delta=\tfrac{3}{2}$ lead to
\begin{align*}
\begin{split}
\left\vert \partial_\tau \vert \Sigma^ \sigma_\tau\vert \right\vert 
& = \left\vert  \int_{\Sigma^\sigma_\tau} H u \, d\mu \right\vert  \\
& \leq \left\vert  \int_{\Sigma^\sigma_\tau} H u^d \, d\mu \right\vert  + \left\vert  \int_{\Sigma^\sigma_\tau} (H - \tfrac{2}{\sigma})  u^t \, d\mu \right\vert  \\
& \leq C\Vert u^d\Vert _{L^2(\Sigma^\sigma_\tau)} + C \sigma^{-1-\varepsilon} \Vert u^t\Vert _{L^2(\Sigma^\sigma_\tau)} \\
& \leq C \sigma^{\frac{3}{2}-\varepsilon},
\end{split}
\end{align*}
where we also used the Cauchy--Schwarz Inequality.

A very similar analysis demonstrates $\left\vert \partial_\tau \vert \Sigma^ \sigma_\tau\vert_{\delta} \right\vert \leq C\sigma^{\frac{3}{2} - \varepsilon}$. Finally, we prove \eqref{eqDerCenter}: By definition,
\begin{align*}
z^i \circ \mathcal{F}^\sigma = \frac{1}{\vert \Sigma^\sigma_\tau \vert_{\delta}} \int_{\Sigma^\sigma_\tau} x^i \, d\mu^\delta.
\end{align*}
Using the variation of area formula, the Cauchy--Schwarz Inequality, \eqref{eqDistArea}, \eqref{eqMCArea}, and Lemma \ref{lemLapse} with $\delta=\tfrac{3}{2}$ we compute
 \begin{align*}
 \begin{split}
\partial_\tau (z^i \circ \mathcal{F}^\sigma)
& = \frac{1}{\vert \Sigma^\sigma_\tau \vert_\delta} \left(\int_{\Sigma^\sigma_\tau} u\nu^i \,d\mu^\delta + \int_{\Sigma^\sigma_\tau} x^i u H\,d\mu^\delta\right) - \frac{1}{\vert \Sigma^\sigma_\tau \vert_{\delta}^{2}}  \partial_\tau \vert \Sigma^ \sigma_\tau\vert_{\delta} \\ & = O(\sigma^{1-2\varepsilon}).
\end{split}
\end{align*}
This proves \eqref{eqDerCenter}.
\end{proof}

\begin{lemma}\label{lemConstEvolution}
Let $\mathcal{I}=\IDS$ be a $C^{2}_{\sfrac{1}{2}+\varepsilon}$-asymptotically Euclidean initial data set with non-vanishing energy $E\neq0$. Let $\emptyset\neq U\subseteq[0,1]$ be an open, connected subset of $[0,1]$ and define $\mathcal{I}_{\tau}$ as in the proof of Theorem~\ref{thExistence} for each $\tau\in U$. 

Let $a,\overline{a} \in [0,1)$, $b,\overline{b}\in [0,\infty)$, $\eta \in (0,2\varepsilon)$, and $\overline{\eta} \in (0,1]$. Then there exists a constant $\overline{\sigma}>0$, depending only on $\varepsilon$, $a$, $\overline{a}$, $b$, $\overline{b}$, $\eta$, $\overline{\eta}$, $C_{\mathcal{I}}$, and $E$ such that the following holds for any $\sigma>\overline{\sigma}$: Assume there exists a $C^{1}$-map $\mathcal{F}^{\sigma}\colon U \times \mathbb{S}^2\to M^3$ such that for every $\tau \in U$ the surface $\Sigma^{\sigma}_\tau \definedas \mathcal{F}^{\sigma} (\tau, \mathbb{S}^2)$ is in the a priori class $\mathcal{A}(\overline{a},\overline{b},\overline{\eta})$ and has constant spacetime mean curvature $\mathcal{H}(\Sigma^{\sigma}_\tau)\equiv\sfrac{2}{\sigma}$ with respect to the initial data set $\mathcal{I}_\tau$. Assume further that $\mathcal{F}^{\sigma}$ is a normal variation map in the sense explained in Lemma \ref{lemLapse}. Now suppose in addition that $\Sigma^\sigma_{\tau_0} \in \mathcal{A} (a,b,\eta)$ for some $\tau_0 \in U$.  Then in fact $\Sigma^\sigma_\tau \in \mathcal{A}(a, b_\tau, \eta_\tau)$ with $b_\tau = b + O(\sigma^{-\min\{2\varepsilon-\eta,\varepsilon\}})$ and $\eta_\tau = \eta + O(\sigma^{-\varepsilon})$ for any $\tau\in U$. Here, the constants in the $O$-notation depend only on $\varepsilon$, $a$, $\overline{a}$, $b$, $\overline{b}$, $\eta$, $\overline{\eta}$, $C_{\mathcal{I}}$, and $E $.
\end{lemma}

\begin{remark}
Note that the assumption $\eta \in (0,2\varepsilon)$ of the lemma is not restrictive as the inclusion $\mathcal{A}(a,b,\eta_1) \subseteq \mathcal{A}(a,b,\eta_2)$ for $0 < \eta_2 < \eta_1 \leq 1$ implies that we may without loss of generality decrease $\eta \in (0,1]$ to achieve $\eta \in (0,2\varepsilon)$. 
\end{remark}

\begin{proof}
We drop the explicit reference to $\sigma$ for notational convenience, as $\sigma$ will not be modified in this proof. Let $r_\tau$ and $\vec{z}_\tau$ denote the area radius and the coordinate center of $\Sigma_\tau$, respectively, and let (slightly abusing notation) $\vec{x}_\tau$ denote the restriction of the coordinate vector $\vec{x}$ to $\Sigma_\tau$, where $\vec{x}$ denotes the asymptotic coordinate chart. The mean curvature of $\Sigma_\tau$ is denoted by $H_\tau$. In this proof $\overline{\sigma}>0$, $C>0$ and constants involved in the $O$-notation may vary from line to line but depend only on $\varepsilon$, $a$, $\overline{a}$, $b$, $\overline{b}$, $\eta$, $\overline{\eta}$, $C_{\mathcal{I}}$, and $E$.

We first show that there exists $\eta_\tau=\eta + O(\sigma^{-\varepsilon})$ such that the second inequality describing the fact that $\Sigma_\tau \in \mathcal{A}(a, b_\tau, \eta_\tau)$ in \eqref{eqDefConcentric} holds, namely
\begin{equation}\label{eqEvolutionEta}
(r_{\tau})^{2+\eta_\tau}\leq |\vec{x}_{\tau}|^{\frac{5}{2}+\varepsilon}.
\end{equation}
Since $\Sigma_\tau \in \mathcal{A} (\overline{a},\overline{b},\overline{\eta})$ for all $\tau \in U$, by the Mean Value Theorem combined with \eqref{eqDerArea}  and  \eqref{eqMCArea} we have 
\[
4 \pi (r_\tau)^2 = |\Sigma_\tau| = |\Sigma_{\tau_0}| + O(\sigma^{\frac{3}{2}- \varepsilon})
 = 4 \pi (r_{\tau_0})^2 (1+ O(\sigma^{-\frac{1}{2}-\varepsilon})),
\]
hence 
\begin{equation}\label{eqMVTr}
r_\tau = r_{\tau_0} (1+O(\sigma^{-\frac{1}{2}-\varepsilon})). 
\end{equation}
Similarly, combining \eqref{eqDerCoord} with \eqref{eqDistArea} and \eqref{eqMCArea} we conclude that 
\begin{equation}\label{eqMVTx}
|\vec{x}_\tau| = |\vec{x}_{\tau_0}|(1+O(\sigma^{-\varepsilon})).
\end{equation}
Since $\Sigma_{\tau_0} \in \mathcal{A} (a,b,\eta)$ we have 
\[
(r_{\tau_0})^{2+\eta} \leq |\vec{x}_{\tau_0}|^{\frac{5}{2}+\varepsilon}, 
\]
which in the view of \eqref{eqMVTr} and \eqref{eqMVTx} can be written as  
\[
(r_{\tau}(1+O(\sigma^{-\sfrac{1}{2}-\varepsilon})))^{2+\eta} \leq \left(|\vec{x}_{\tau}|(1+O(\sigma^{-\varepsilon}))\right)^{\frac{5}{2}+\varepsilon}.
\]
Consequently, we have
 \[
(r_{\tau})^{2+\eta}(1-C\sigma^{-\varepsilon}) \leq |\vec{x}_{\tau}|^{\frac{5}{2}+\varepsilon}.
\] 
Choosing 
\begin{equation}\label{eqEtaTau}
\eta_\tau \definedas \eta + \log_{r_\tau} (1-C\sigma^{-\varepsilon}),
\end{equation}
\eqref{eqEvolutionEta} follows. Note that by \eqref{eqDistArea} we have
\begin{equation*}
\eta_\tau = \eta + \frac{\ln(1-C\sigma^{-\varepsilon})}{\ln r_\tau} =  \eta + O(\sigma^{-\varepsilon} (\ln \sigma)^{-1}) = \eta + O(\sigma^{-\varepsilon}).
\end{equation*}
%provided that $\overline{\sigma}$ depending only on $\varepsilon$, %$a$, $\overline{a}$, $b$, $\overline{b}$, $\eta$, $\overline{\eta}$, and $C_{\mathcal{I}}$ is chosen to be sufficiently large.

We will now apply a similar method and adjust the value of the constant $b_\tau$ so that the first and the third  inequality in \eqref{eqDefConcentric} describing the fact that $\Sigma_\tau \in \mathcal{A}(a, b_\tau, \eta_\tau)$, that is
\begin{equation}\label{eqEvolutionBeta}
\vert \vec{z}_\tau\vert  \leq a r_\tau + b_\tau (r_\tau)^{1-\eta_\tau}
\end{equation}
and 
\begin{equation}\label{eqEvolLast}
\int_{\Sigma_\tau} H_\tau^2 \, d\mu - 16 \pi \leq \frac{b_\tau}{(r_\tau)^{\eta_\tau}},
\end{equation}
hold with $\eta_\tau$ as defined in \eqref{eqEtaTau}. 

First, we deal with \eqref{eqEvolutionBeta}. Since $\Sigma_{\tau_0} \in \mathcal{A} (a,b,\eta)$ we have 
\[
\vert \vec{z}_{\tau_0}\vert  \leq a r_{\tau_0} + b (r_{\tau_0})^{1-\eta}.
\]
Combining this with \eqref{eqMVTr} and \eqref{eqDerCenter} we obtain
\[
\vert \vec{z}_{\tau}\vert  + O(\sigma^{1-2\varepsilon})  \leq a r_{\tau}+ O(\sigma^{\frac{1}{2}-\varepsilon})+ b (r_{\tau}+ O(\sigma^{\frac{1}{2}-\varepsilon}))^{1-\eta},
\]
which in the view of \eqref{eqMCArea} may be further rewritten as 
\[
\vert \vec{z}_{\tau}\vert \leq a r_{\tau} + (b + O(\sigma^{\eta-2\varepsilon}) + O(\sigma^{\eta - \frac{1}{2} - \varepsilon}))(r_{\tau})^{1-\eta}.
\]
Since $\varepsilon \leq \tfrac{1}{2}$, we conclude that
\[
\vert \vec{z}_{\tau}\vert \leq a r_{\tau} +  (r_{\tau})^{1-\eta}(b+O(\sigma^{\eta-2\varepsilon})).
\]
Recall that by our definition \eqref{eqEtaTau}  of $\eta_\tau$ we have  $\eta = \eta_\tau - \log_{r_\tau} (1-C\sigma^{-\varepsilon})$. Hence
\[
\vert \vec{z}_{\tau}\vert \leq a r_{\tau} + r_{\tau}^{1-\eta_\tau}(1-C\sigma^{-\varepsilon})(b+O(\sigma^{\eta-2\varepsilon})).
\]
Consequently, we have 
\begin{equation}\label{eqAlmostEvolutionBeta}
\vert \vec{z}_{\tau}\vert \leq a r_{\tau} + (r_{\tau})^{1-\eta_\tau}(b+ C\sigma^{-\min\{2\varepsilon-\eta,\varepsilon\}}).
\end{equation}

Next, we address \eqref{eqEvolLast}. We recall that $H_\tau^2 = \mathcal{H}^2 - \tau^2 P^2 = \frac{4}{\sigma^2}  - \tau^2 P^2$  which implies that $\partial_\tau H_\tau^2 = O(\sigma^{-3-2\varepsilon})$. Consequently, we may compute using the variation of area formula and \eqref{eqLapse} with $\delta=\tfrac{3}{2}$ that
\begin{equation}\label{eqEvolInt}
\partial_\tau \int_{\Sigma_\tau} H_\tau^2 \, d\mu = \int_{\Sigma_\tau} \partial_\tau H_\tau^2 \, d\mu + \int_{\Sigma_\tau} u H_\tau^3 \, d\mu = O(\sigma^{-2\varepsilon}).
\end{equation}
Again, since $\Sigma_{\tau_0} \in \mathcal{A} (a,b,\eta)$ we have
\[
\int_{\Sigma_{\tau_0}} H_{\tau_0}^2 \, d\mu - 16 \pi \leq \frac{b}{(r_{\tau_0})^{\eta}}.
\]
As before, we use the Mean Value Theorem, \eqref{eqEvolInt}, and \eqref{eqMVTr} to rewrite this as
\[
\int_{\Sigma_{\tau}} H_{\tau}^2 \, d\mu - 16 \pi + O(\sigma^{-2\varepsilon}) \leq b (r_{\tau}+O(\sigma^{\frac{1}{2}-\varepsilon}))^{-\eta}
\]
which in the view of \eqref{eqMCArea} may be further rewritten as 
\[
\int_{\Sigma_{\tau}} H_{\tau}^2 \, d\mu - 16 \pi \leq (r_{\tau})^{-\eta}(b+O(\sigma^{-\frac{1}{2}-\varepsilon})+O(\sigma^{\eta-2\varepsilon})).
\]
Substituting $\eta = \eta_\tau - \log_{r_\tau} (1-C\sigma^{-\varepsilon})$ in the view of $\varepsilon\leq \tfrac{1}{2}$ gives
\[
\int_{\Sigma_{\tau}} H_{\tau}^2 \, d\mu - 16 \pi \leq (r_{\tau})^{-\eta_\tau}(1-C\sigma^{-\varepsilon})(b+O(\sigma^{\eta-2\varepsilon})).
\]
Consequently, we have 
\begin{equation}\label{eqAlmostEvolLast}
\int_{\Sigma_{\tau}} H_{\tau}^2 \, d\mu - 16 \pi \leq (r_{\tau})^{-\eta_\tau}(b+ C\sigma^{-\min\{2\varepsilon-\eta,\varepsilon\}}).
\end{equation}

Together, the inequalities \eqref{eqAlmostEvolutionBeta} and \eqref{eqAlmostEvolLast} imply the existence of the constant $b_\tau = b + O(\sigma^{-\min\{2\varepsilon-\eta,\varepsilon\}})$  such that \eqref{eqEvolutionBeta} and \eqref{eqEvolLast} hold. This concludes the proof as we can choose $a_{\tau}=a$ as the above computations show.
\end{proof}

\begin{lemma}\label{lemFoliation}
Under the assumptions of Theorem \ref{thExistence}, there exists a constant $\overline{\sigma}>0$, depending only on $\varepsilon$, $a$, $b$, $\eta$, $C_{\mathcal{I}_{1}}$, and $E$, and a compact set $\mathcal{K}\subset M^{3}$ such that the map $\Psi_{1}\colon (\overline{\sigma},\infty)\times\mathbb{S}^{2}\to M^{3}\setminus \mathcal{K}$ defined by \eqref{eqDefPsi} is a bijective $C^{1}$-map.
\end{lemma}
\begin{proof} 
To prove the claim, we need to show that $\Psi_{1}$ is $C^{1}$, injective, and surjective onto a suitably chosen exterior region of $M^3$. We already proved in Theorem \ref{thExistence} that $\mathcal{F}^{\sigma}$ and thus $\Psi_{1}$ is $C^{1}$ with respect to the $\mathbb{S}^{2}$-component. The differentiability with respect to $\sigma$ can be proven following the Implicit Function Theorem argument of Lemma \ref{lemStatement}, where the graphical spacetime mean curvature map is to be interpreted as a function of $\sigma\in(\overline{\sigma},\infty)$ instead of as a function of $\tau\in[0,1]$. This is to be viewed in light of the uniqueness results in Section \ref{subsecUnique}.\\

\paragraph{\emph{Injectivity.}} In order to show injectivity of $\Psi_{1}$, we need to assert that  $\Sigma^{\sigma_{1}}_{1}\cap\Sigma^{\sigma_{2}}_{1}=\emptyset$ for any choice of $\sigma_{1}\neq\sigma_{2}$, $\sigma_{1},\sigma_{2}>\overline{\sigma}$. This can be done by analyzing the lapse function of the variation $\Psi_{1}\colon (\overline{\sigma},\infty) \times \mathbb{S}^{2}\to M^3$ with respect to $\sigma$, namely $u=u^{\sigma}_{1} \definedas g(\partial_\sigma \Psi_1,\nu)$, where $\nu=\nu^{\sigma}_{1}$ denotes the outward unit normal of $\Sigma^{\sigma}_{1}\hookrightarrow(M^3,g)$. If $u>0$ on $(\overline{\sigma},\infty)\times\mathbb{S}^{2}$, we can conclude that $\Psi_{1}$ is injective.

We will in fact show that $u=1+O(\sigma^{-\frac{1}{2}-\varepsilon})$. Again, $C>0$ and $\overline{\sigma}>0$ denote generic constants that may vary from line to line, but depend only on $\varepsilon$, $a$, $b$, $\eta$, $C_{\mathcal{I}_{1}}$, and $E$. First, we note that since the spacetime mean curvature of $ \Sigma^\sigma_1$ in the initial data set $\mathcal{I}_{1}$ is constant, $\mathcal{H}(\Sigma^{\sigma}_{1})\equiv\sfrac{2}{\sigma}$, we have
\begin{align}\label{eqLHu}
\begin{split}\mathcal{L}u&=\sqrt{1-\left(\tfrac{P}{H}\right)^{2}}\,\partial_{\sigma}\mathcal{H}(\Sigma^{\sigma}_{1})\\
&=(1+O(\sigma^{-1-\varepsilon}))\,\partial_{\sigma}\left(\tfrac{2}{\sigma}\right)\\
&=-\tfrac{2}{\sigma^{2}}+O(\sigma^{-3-\varepsilon}),
\end{split}
\end{align}
which uniquely determines $u\in W^{2,2}(\Sigma^{\sigma}_{1})$ by Proposition \ref{propInvertibility}. Furthermore, by \eqref{eqLHu}, \eqref{eqFallOff}, \eqref{eq2FFtraceless}, and the asymptotic decay assumptions on $\mathcal{I}_{1}$, we have
\begin{align}\label{equ1}
\begin{split}
\mathcal{L}(u-1) & = \mathcal{L}u + \vert A\vert ^2 + \ric (\nu, \nu)
+ \tfrac{P}{H}(\nabla_\nu \tr_g K - \nabla_\nu K(\nu, \nu)) \\
&= -\tfrac{2}{\sigma^{2}}+O(\sigma^{-3-\varepsilon})+ \tfrac{H^{2}}{2}+\vert \mathring{A}\vert ^2 + \ric (\nu, \nu)
+ \tfrac{P}{H}(\nabla_\nu \tr_g K - \nabla_\nu K(\nu, \nu)) \\
& = \ric (\nu, \nu) + O(\sigma^{-3-\varepsilon}),\\[-3ex]
\end{split}
\end{align}
where $H=H(\Sigma^{\sigma}_{1})$, $P=P(\Sigma^{\sigma}_{1})$. 
This shows that $\mathcal{L}(u-1) = O(\sigma^{-\frac{5}{2}-\varepsilon})$ which is not sufficient for concluding that $u = 1 + O(\sigma^{-\varepsilon})$ and thus $u>0$ via Corollary \ref{corEllipticRegularity} and Lemma \ref{lemSobolevEmbed}: such an argument would require $\mathcal{L}(u-1) = O(\sigma^{-3-\varepsilon})$ which we obviously do not have. 

Instead, we argue as in the proof of Lemma \ref{lemLapse}. For $v\definedas u-1$, the above computation shows that $\mathcal{L}v=O(\sigma^{-\frac{5}{2}-\varepsilon})$, which in combination with Corollary \ref{corEllipticRegularity} and \eqref{eqTranslational0} gives
\begin{align}\label{eqVD}
\begin{split}
\Vert v^d\Vert _{W^{2,2}(\Sigma_1^\sigma)} 
& \leq C\sigma^2 \Vert \mathcal{L} v^d\Vert _{L^2(\Sigma^\sigma_1)} \\
& \leq C\sigma^2 \left(\Vert \mathcal{L} v \Vert _{L^2(\Sigma^\sigma_1)} + \Vert \mathcal{L} v^t\Vert _{L^2(\Sigma^\sigma_1)}\right) \\
& \leq C\sigma^2 \left(\sigma^{-\frac{3}{2} - \varepsilon} + \Vert \mathcal{L} v^t\Vert _{L^2(\Sigma^\sigma_1)}\right) \\
& \leq C\sigma^{\frac{1}{2}-\varepsilon} + \sigma^{-\frac{1}{2}-\varepsilon}\Vert v^t\Vert _{L^2(\Sigma^\sigma_1)}.                   
\end{split}
\end{align}
In addition, for $i=1,2,3$, by adding a rich zero and using the orthogonality of $v^d$ and $f_i$, we have
\begin{align*}
\begin{split}
\left\vert \int_{\Sigma^\sigma_1} v f_i \, d\mu\right\vert  \leq&\, \frac{\sigma^3}{6 \vert m_\mathcal{H}\vert }\left\vert  \int_{\Sigma_1^\sigma} v \mathcal{L}f_i \, d\mu \right\vert  + \left\vert \int_{\Sigma^\sigma_1} v^t f_i \, d\mu - \frac{\sigma^3}{6m_\mathcal{H}}\int_{\Sigma^\sigma_1} v^t \mathcal{L}f_i \, d\mu \right\vert \\
&+ \frac{\sigma^3}{6\vert m_\mathcal{H}\vert } \left\vert \int_{\Sigma^\sigma_1}  v^d  \mathcal{L} f_i  \, d\mu \right\vert ,
\end{split}
\end{align*}
where 
\begin{align*}
\left\vert \int_{\Sigma^\sigma_1} v^t f_i \, d\mu - \frac{\sigma^3}{6m_\mathcal{H}}\int_{\Sigma^\sigma_1} v^t \mathcal{L}f_i \, d\mu \right\vert \leq \frac{C}{\sigma^{\varepsilon}}\, \Vert v^t\Vert _{L^2(\Sigma^\sigma_1)}
\end{align*}
by \eqref{eqTranslational}, and 
\begin{align*}
\left\vert \int_{\Sigma^\sigma_1} v^d \mathcal{L} f_i \, d\mu \right\vert  \leq \frac{C}{\sigma^{\frac{5}{2} + \varepsilon}}\, \Vert v^d\Vert _{L^2(\Sigma^\sigma_1)}
\end{align*}
by \eqref{eqforLemmaLapse}. Using the fact that $\mathcal{L} v = \ric(\nu, \nu) + O(\sigma^{-3-\varepsilon})$ by \eqref{equ1}, Lemma A.3 from \cite{NerzCMC} (a result showing that $\int_{\Sigma^\sigma_1} \ric(\nu,\nu) x^i d\mu = O(\sigma^{-\varepsilon})$ which is fully Riemannian and thus directly carries over to our spacetime context), and integration by parts, we obtain as in Lemma~\ref{lemLapse} that
\begin{align*}
\begin{split}
\left\vert  \int_{\Sigma^\sigma_1} v \mathcal{L} f_i \, d\mu \right\vert  & \leq \left\vert  \int_{\Sigma^\sigma_1} f_i \mathcal{L} v \, d\mu \right\vert + 2\left\vert  \int_{\Sigma^\sigma_1} \tfrac{P}{H} K \left(v \nabla^{\Sigma^\sigma_1} f_i - f_i \nabla^{\Sigma^\sigma_1} v, \nu  \right)\, d\mu \right\vert   \\ 
&\leq C \sigma^{-2-\varepsilon} + C\sigma^{-3-\varepsilon}\Vert v\Vert _{L^2(\Sigma_1^\sigma)}.
\end{split}
\end{align*}
Combining these estimates, we get, again grouping terms as in Lemma~\ref{lemLapse}, that
\begin{align*}
\Vert v^t\Vert _{L^2(\Sigma^\sigma_1)} \leq C \sigma^{\frac{1}{2}-\varepsilon} \Vert v^d\Vert _{L^2(\Sigma_1^\sigma)} + C \sigma^{1-\varepsilon},
\end{align*}
which, together with \eqref{eqVD}, gives
\begin{align*}
\Vert v^d\Vert _{W^{2,2}(\Sigma_1^\sigma)} \leq C \sigma^{\frac{1}{2}-\varepsilon},\quad\quad \Vert v^t\Vert _{L^2(\Sigma^\sigma_1)} \leq C \sigma^{1-\varepsilon}.
\end{align*}
Finally, Corollary \ref{corEllipticRegularity} gives us
 \begin{align*}
 \begin{split}
 \Vert v\Vert _{W^{2,2}(\Sigma^\sigma_1)} \leq C \left(\sigma^2\Vert \mathcal{L} v\Vert _{L^2(\Sigma^\sigma_1)} + \Vert v\Vert _{L^2(\Sigma^\sigma_1)} \right) \leq C \sigma^{1-\varepsilon}.
 \end{split}
 \end{align*}
Then $\Vert v\Vert _{W^{2,2}(\Sigma^\sigma)} \leq C \sigma^{1-\varepsilon}$, so by Lemma \ref{lemSobolevEmbed} we get that $u=1+v=1+O(\sigma^{-\varepsilon})$ is strictly positive for all $\sigma > \overline{\sigma}$. This shows that $\Psi_{1}$ is indeed injective.\\[-1ex]

\paragraph{\emph{Surjectivity.}} By construction, the STCMC-surfaces $\Sigma^{\sigma}=\Psi_1 (\sigma,\mathbb{S}^2)$ for $\sigma>\overline{\sigma}$ are in the class of asymptotically centered surfaces $\mathcal{A}(0,b,\eta)$ for some $b > 0$ and $\eta\in(0,1]$. In particular, recalling Proposition \ref{propRegularity}, each $\Sigma^\sigma$ can be written as a graph over a sphere enclosing the interior region of $M^3$. Suppose $p\in M^3$ is in the exterior region of some $\Sigma^\sigma$ with $\sigma>\overline{\sigma}$. By comparability of the coordinate and mean curvature radii for surfaces in the class $\mathcal{A}(0,b,\eta)$ (see \eqref{eqDistArea} and \eqref{eqMCArea}), we can find $\widetilde{\sigma} > \sigma$ such that $p$ lies in the region enclosed by $\Sigma^{\widetilde{\sigma}}$, and hence in the annulus $A_{\sigma,\widetilde{\sigma}}$ between $\Sigma^{\sigma}=\Psi_1(\sigma, \mathbb{S}^2)$ and $\Sigma^{\widetilde{\sigma}}=\Psi_1(\widetilde{\sigma}, \mathbb{S}^2)$. Since $\Psi_1\colon[\sigma,\infty) \times \mathbb{S}^2 \to M^3$ is continuous it follows that $A_{\sigma,\widetilde{\sigma}}=\Psi_1([\sigma,\widetilde{\sigma}]\times \mathbb{S}^2)$ hence $p=\Psi_1(\hat{\sigma}, q)$ for some $\hat{\sigma} \in [\sigma,\widetilde{\sigma}]$ and $q\in\mathbb{S}^2$. As $\sigma>\overline{\sigma}$ was arbitrary, this proves surjectivity.
\end{proof}

\subsection{Uniqueness of the STCMC surfaces.}\label{subsecUnique} We close this section by proving that the constant spacetime mean curvature surfaces are unique in the a priori class of asymptotically centered surfaces $\mathcal{A}(a,b,\eta)$. As Brendle and Eichmair \cite{BrendleEichmair} constructed examples of asymptotically Euclidean Riemannian manifolds with ``off-center'' (i.e.~not included in the a priori class) CMC-surfaces provided $\scal\geq0$ is violated, we will restrict our uniqueness statements to the  a priori class --- at least when not assuming the dominant energy condition $\mu\geq\vert J\vert_{g}$. To the best knowledge of the authors, it is not known whether such examples can also be constructed if the dominant energy condition or its Riemannian analog $\scal\geq0$ are satisfied.

The uniqueness result is proven in a similar way as the existence result of Section~\ref{subsecExistence}, namely by the method of continuity. The ``starting point" of the method of continuity is the following result from \cite{NerzCMC}, again adapted to our notation.

\begin{theorem}[Nerz 2015]\label{thUniquenessCMC}
Let $a\in [0,1)$, $b \geq 0$, and $\eta\in (0,1]$ be constants and let $(M^{3},g)$ be a $C^2_{\sfrac{1}{2} + \varepsilon}$-asymptotically Euclidean manifold viewed as a $C^2_{\sfrac{1}{2} + \varepsilon}$-asymptotically Euclidean initial data set $\mathcal{I}_0=(M^{3},g,K\equiv 0,\mu =\tfrac{1}{2}\scal,J\equiv 0)$ with
non-vanishing energy $E \neq 0$. Then there is a constant $\sigma_{\mathcal{I}_{0}}$ depending only on $\varepsilon$, $a$, $b$, $\eta$, $C_{\mathcal{I}_0}$, and $E$, such that for all $\sigma > \sigma_{\mathcal{I}_{0}}$, there is a unique surface $\Sigma^{\sigma}_{0} \in \mathcal{A}(a,b,\eta)$  with constant mean curvature $H(\Sigma^{\sigma}_{0}) \equiv \sfrac{2}{\sigma}$ with respect to $\mathcal{I}_{0}$.
\end{theorem}

Our uniqueness result generalizes this to STCMC-surfaces in the spacetime context.

\begin{theorem}[Uniqueness of STCMC-foliation]\label{thUniqueness}
Let $a \in [0,1)$, $b \geq 0$, and $\eta\in (0,1]$ be constants and let $\mathcal{I}_{1}=\IDS$ be a
$C^2_{\sfrac{1}{2} + \varepsilon}$-asymptotically Euclidean initial data set with
non-vani\-shing energy $E \neq 0$. Then there is a constant $\sigma_{\mathcal{I}_{1}}$  depending only on $\varepsilon$, $a$, $b$, $\eta$, $C_{\mathcal{I}_{1}}$, and $E$, such that for all $\sigma > \sigma_{\mathcal{I}_{1}}$, there is a unique surface $\Sigma^{\sigma}_{1} \in \mathcal{A}(a,b,\eta)$
with constant spacetime mean curvature $\mathcal{H}(\Sigma^{\sigma}_{1}) \equiv\sfrac{2}{\sigma_{1}}$ with respect to~$\mathcal{I}_{1}$.
\end{theorem}

\begin{proof}
We rely on the same type of argument as in the proof of Theorem \ref{thExistence}. Fix a surface $\Sigma^{\sigma}_{1}$ as in the assumptions, with $\sigma>\sigma_{\mathcal{I}_{1}}$. We now drop the explicit reference to $\sigma$ for notational convenience, as $\sigma$ will not be modified in this proof. Let $Z \subseteq [0,1]$ be a \emph{maximal} subset such that there is a $C^1$-map $\Phi\colon Z\times \mathbb{S}^2 \to M^3$ with the following properties for all $\tau\in Z$:
\begin{enumerate}[(i)]
\item $\Phi(1,\mathbb{S}^{2})=\Sigma_{1}$,\\[-2ex]
\item $\Sigma_\tau\definedas\Phi(\tau,\mathbb{S}^2)$ has constant spacetime mean curvature $\mathcal{H}(\Sigma_\tau)\equiv\sfrac{2}{\sigma}$ with respect to the initial data set $\mathcal{I}_\tau$, where $\mathcal{I}_{\tau}$ is defined as in the proof of Theorem~\ref{thExistence},\\[-2ex]
\item $\partial_\tau \Phi$ is orthogonal to $\Sigma_\tau$.
\end{enumerate} 
\emph{Maximality} is understood as in the proof of Theorem \ref{thExistence}. Arguing as in the proof of Theorem \ref{thExistence}, we conclude that $Z=[0,1]$ and that there are  constants $a_{\mathcal{I}_{1}} \in [0,1)$, $b_{\mathcal{I}_{1}} \geq 0$ and $\eta_{\mathcal{I}_{1}}\in (0,1]$ such that $\Sigma_\tau \in \mathcal{A} (a_{\mathcal{I}_{1}}, b_{\mathcal{I}_{1}},\eta_{\mathcal{I}_{1}})$ for every $\tau\in[0,1]$, if $\sigma_{\mathcal{I}_{1}}$ suitably large, depending only on $\varepsilon$, $a$, $b$, $\eta$, $C_{\mathcal{I}_1}$, and $E$. In particular, we see that $\Sigma_0 \in \mathcal{A} (a_{\mathcal{I}_{1}}, b_{\mathcal{I}_{1}},\eta_{\mathcal{I}_{1}})$ is a surface with constant  mean curvature $H(\Sigma_{0})\equiv\sfrac{2}{\sigma}$ with respect to $\mathcal{I}_{0}$\footnote{Note that although it was assumed throughout the proof of Theorem \ref{thExistence} that $a=0$, this proof extends straightforwardly to deal with the general case $a\in [0,1)$. In fact, in the view of Lemma \ref{lemConstEvolution}, we may set $a_{\mathcal{I}_{1}}=a$.}. By Theorem \ref{thUniquenessCMC}, such a surface is unique in this class. By the method of continuity approach and the local uniqueness in the Implicit Function Theorem, the map $\Phi$ is uniquely determined also by its start value $\Phi(0,\mathbb{S}^2) = \Sigma_0$. It follows directly that $\Sigma_{1} = \Phi(1,\mathbb{S}^2)$ is uniquely determined by its spacetime mean curvature in $\mathcal{I}_{1}$.
\end{proof}

%% file: CoM.tex
% !TEX root = main.tex
Let $\{\Sigma^\sigma\}_{\sigma>\overline{\sigma}}$ be a foliation of a $C^2_{\sfrac{1}{2}+\varepsilon}$-asymptotically Euclidean initial data set $\mathcal{I}=\IDS$ for which $\Sigma^{\sigma}$ grows to the round sphere at infinity as $\sigma\to\infty$. Then we may define the \emph{coordinate center} of this foliation as the limit $\lim_{\sigma\to \infty} \vec{z}^{\,\sigma}$, where $\vec{z}^{\,\sigma}=\vec{z}\,(\Sigma^{\sigma})$ is the coordinate center of $\Sigma^\sigma$ as defined in Definition \ref{defConcentric}, provided that this limit exists (in $\R^{3}$). We would like to draw the attention of the reader to the fact that, while the foliations considered here do not depend on the choice of asymptotic coordinates $\vec{x}$, the coordinate centers $\vec{z}^{\,\sigma}$ and as a consequence also their limit, do depend on $\vec{x}$. We will discuss the subtle consequences of this within this section, too.

Let us first consider the case of the CMC-foliation: In this case, $\{\Sigma^\sigma\}_{\sigma>\sigma_{\mathcal{I}_0}}$ is the unique foliation of a given $C^2_{\sfrac{1}{2}+\varepsilon}$-asymptotically Euclidean manifold $(M^3,g)$ or initial data set $\mathcal{I}_0=(M^{3},g,K\equiv0,\scal,J\equiv0)$ by surfaces of constant mean curvature constructed in \cite{HY,MetzgerCMC,NerzCMC} and discussed in Section \ref{secExistence} above. Under the additional assumption that $(M^{3},g)$ satisfies the Riemannian $C^2_{1+\varepsilon}$-Regge--Teitelboim condition (see Definition \ref{defRT}), the coordinate center of this foliation is well-defined if and only if the Beig--\'O Murchadha center of mass $\vec{C}_\text{B\'OM}$ given by \eqref{BOMRT} is well-defined as was shown by \cite{HuangstableCMC,NerzCMC}. They also show that in this case, one has
\begin{align}\label{eqLimitCMC}
\vec{C}_{\text{CMC}}\definedas\lim_{\sigma\to \infty} \vec{z}^{\,\sigma} = \vec{C}_\text{B\'OM}.
\end{align}

Now suppose that $\{\Sigma^\sigma\}_{\sigma>\sigma_{\mathcal{I}_1}}$ is the unique foliation of a $C^2_{\sfrac{1}{2}+\varepsilon}$-asymptotically Euclidean initial data set $\mathcal{I}=\IDS$ by surfaces of constant spacetime mean curvature as constructed in Theorem \ref{thExistence}. One cannot in general expect that \eqref{eqLimitCMC} also holds for the STCMC-foliation because the foliation is defined in terms of $K$, whereas the Beig--\'O Murchadha center of mass is a purely Riemannian quantity, i.e.~independent of $K$. One can only expect that \eqref{eqLimitCMC} will hold if $K$ falls off very fast, in particular faster than the optimal decay assumed in this paper. In this section we will confirm that this is indeed the case. 

\subsection{A variational formula for STCMC-surfaces}
The following proposition generalizes \cite[Proposition 6.5]{NerzCMC} to the spacetime case. 
                                                                               
\begin{proposition}\label{propCenterMotion}
For $b \geq 0$ and $\eta \in (0,1]$, let $\Sigma\in\mathcal{A}(0,b,\eta)$ be a closed, oriented constant spacetime mean curvature $2$-surface with spacetime mean curvature $\mathcal{H}\equiv\sfrac{2}{\sigma}$, with outer unit normal denoted by $\nu$, in a $C^2_{\sfrac{1}{2}+\varepsilon}$-asymptotically Euclidean initial data set $\mathcal{I}=\IDS$.  Consider a $C^1$-map $\mathcal{F}\colon (-s_0, s_0) \times \Sigma \rightarrow M^3$ such that $\mathcal{F}(0,\cdot) = \id_\Sigma$. For each $s\in (-s_0,s_0)$, we let $\vec{z}_s = (z_s^1,z_s^2,z_s^3)$ denote the coordinate center of the surface $\Sigma_s=\mathcal{F}(s,\Sigma)$ and consider the \emph{lapse} $u \definedas g((\partial_s \mathcal{F})\left\vert\right._{s=0}, \nu)$ of the foliation.

Then there are constants $C>0$ and $\overline{\sigma}>0$ depending only on $\varepsilon$, $b$, $\eta$, $C_{\mathcal{I}}$ such that 
\begin{align}\label{eqVarCenter}
\left\vert \left.(\partial_s z_s^i)\right\vert_{s=0} - \frac{3}{\vert \Sigma\vert } \int_\Sigma  u \nu^i \, d\mu \right\vert  \leq \frac{C}{\sigma^{\frac{3}{2} +\varepsilon}} \Vert u\Vert _{L^2(\Sigma)}, \qquad i=1,2,3,
\end{align}
provided that $\sigma>\overline{\sigma}$.
\end{proposition}

\begin{proof}
Since the coordinate center of a surface is invariant under tangential diffeomorphisms (along $\Sigma$), we may without loss of generality assume that $\mathcal{F}$ is a normal variation of $\Sigma$, such that in particular $\left.(\partial_s \mathcal{F})\right\vert_{s=0} = u\nu$ holds on $\Sigma$. By definition,
\begin{align*}
z_s^i= \frac{1}{\vert \Sigma_s\vert_{\delta} } \int_{\Sigma_s} x^i_s \, d\mu^\delta.
\end{align*}
Using the variation of area formula and adding rich zeroes in the third and the forth lines, we compute, dropping the explicit reference to $\delta$ in the denominator,
 \begin{align*}
\begin{split}
&\left.(\partial_s z^i_s)\right\vert_{s=0} \\
&= \frac{1}{\vert \Sigma\vert } \left(\int_\Sigma u\nu^i \,d\mu^\delta + \int_\Sigma x^i u H^\delta \,d\mu^\delta - \frac{1}{\vert \Sigma\vert } \int_\Sigma x^i \,d\mu^\delta \int_\Sigma u H^\delta \,d\mu^\delta \right)\\
& = \frac{1}{\vert \Sigma\vert } \left(\int_\Sigma u\nu^i \,d\mu^\delta + \int_\Sigma (x^i-z^i) u H^\delta \,d\mu^\delta \right)\\
& = \frac{1}{\vert \Sigma\vert } \left(\int_\Sigma u\nu^i \,d\mu^\delta + 2\int_\Sigma \frac{x^i-z^i}{r} u  \,d\mu^\delta + \int_\Sigma (x^i-z^i) u \left(H^\delta-\frac{2}{r}\right) \,d\mu^\delta \right)\\
& = \frac{1}{\vert \Sigma\vert }\left(3\int_\Sigma u\nu^i \,d\mu^\delta + 2\int_\Sigma \left(\frac{x^i-z^i}{r} -\nu^i\right)u  \,d\mu^\delta + \int_\Sigma (x^i-z^i) u \left(H^\delta-\frac{2}{r}\right) \,d\mu^\delta \right).
\end{split}
\end{align*}
Now subtract and add the component $(\nu^{\delta})^{i}$ of the $\delta$-outward unit normal $\nu^{\delta}$ to $\Sigma$ in the bracket of the second term and recall the fact that $\Sigma$ can be written as a graph over $\mathbb{S}^{2}_{r}(\vec{z}\,)$ with graph function $f$ satisfying $\| f\|_{W^{2,\infty}}=O(r^{\frac{1}{2}-\varepsilon})$ by Proposition~\ref{propRegularity}. Then, by comparability of $\vert \vec{x}\,\vert$ and $r$ as established in Proposition \ref{propRegularity}, we find, recalling $a=0$ in our case,
\begin{align}\label{eqnudelta}
\vec{\nu}^{\,\delta}=\frac{\vec{x}-\vec{z}}{\vert \vec{x}-\vec{z}\vert}+O(\sigma^{-\frac{1}{2}-\varepsilon})=\frac{\vec{x}-\vec{z}}{r}+O(\sigma^{-\frac{1}{2}-\varepsilon}).
\end{align}
Applying the Cauchy--Schwarz Inequality to the above identity for $(\partial_s z^i_s)\left\vert\right._{s=0} $ and using \eqref{eqnudelta}, Lemma \ref{lemComparison}, and \eqref{eqFallOff} to estimate the individual terms, respectively, we obtain \eqref{eqVarCenter}.
\end{proof}

\subsection{STCMC-center of mass}
In Section \ref{secExistence}, we constructed the unique STCMC-foliation of a $C^2_{\sfrac{1}{2}+\varepsilon}$-asymptotically Euclidean initial data set $\mathcal{I}=\IDS$, i.e.~the unique foliation by surfaces $\{\Sigma_1^\sigma\}_{\sigma>\sigma_1}$ of constant spacetime mean curvature $\mathcal{H}(\Sigma_1^\sigma) \equiv\sfrac{2}{\sigma}$, provided it has non-vanishing energy $E\neq0$. This was achieved by deforming the constant mean curvature foliation $\{\Sigma_0^\sigma\}_{\sigma>\sigma_0}$ of the $C^2_{\sfrac{1}{2}+\varepsilon}$-asymptotically Euclidean manifold $(M^{3},g)$ from Theorem \ref{thExistenceCMC} along the curve of initial data sets $\{\mathcal{I}_\tau\}_{\tau \in [0,1]}$, where $\mathcal{I}_\tau = (M^{3},g,\tau K, \mu_\tau, \tau J)$ is as described in Section \ref{ssecStrategy}. We will now apply Proposition \ref{propCenterMotion} to find how the coordinate center of a leaf changes under this particular deformation. As a result, we prove Lemma \ref{lemSTCMC-CoM} relating the respective coordinate centers $\vec{z}^{\,\sigma}_0$ and $\vec{z}^{\,\sigma}_1$ of the surfaces $\Sigma^\sigma_0$ and $\Sigma^\sigma_1$. Note that for the proof of this result, it is necessary to assume that the fall-off rate of $K$ is $K=O(\vert \vec{x}\,\vert ^{-2})$, which is faster than we originally assumed in Definition \ref{defAEdata} and in particular faster than one needs for existence and uniqueness of the foliation. See also Conjecture \ref{conj:convergence} below.
                                                                   
\begin{lemma}\label{lemSTCMC-CoM}
Let $\mathcal{I}=\IDS$ be an STCMC-foliated $C^2_{\sfrac{1}{2}+\varepsilon}$-asymptotically Euclidean initial data set with non-vanishing energy $E \neq 0$. Assume in addition that 
\begin{align}\label{eqKExtra}
\vert K\vert \leq C_{\mathcal{I}} \vert \vec{x}\,\vert ^{-2}
\end{align}
for $\vec{x}\in\R^{3}\setminus\overline{B_{R}(0)}$, with $\vec{x}$ the asymptotic chart. Then there exist constants $C>0$, $\overline{\sigma}>0$ depending only on $\varepsilon$, $C_{\mathcal{I}}$, and $E$ such that for $i=1,2,3$ we have for all $\sigma>\overline{\sigma}$
\begin{align}\label{eqDeformationCoM}
\left\vert (z_1^\sigma)^i - (z_0^\sigma)^i - \frac{1}{32\pi E}\int_{\mathbb{S}^{2}_{\sigma}} \frac{x^i \left(\sum_{k,l}\pi_{kl}x^k x^l\right)^2}{\sigma^3} \, d\mu^\delta \right\vert  \leq \frac{C}{\sigma^{\varepsilon}},
\end{align}
where $(z_0^\sigma)^i$ and $(z_1^\sigma)^i$ denote the components of the coordinate centers $\vec{z}^{\,\sigma}_0$ and $\vec{z}^{\,\sigma}_1$ of the STCMC-surfaces $\Sigma^{\sigma}_{0}$ and $\Sigma^{\sigma}_{1}$ with respect to $\mathcal{I}_{0}$ and $\mathcal{I}_{1}$ as defined above, respectively.
\end{lemma}
\begin{remark}
Instead of assuming $K=O(\vert\vec{x}\,\vert^{-2})$, one could also assume Regge--Teitelboim conditions on $K$ by carefully tracking all even and odd parts or identify other sufficient decay conditions, as necessary for what one wants to do. For our purposes, it is enough to assume $K=O(\vert \vec{x}\,\vert^{-2})$.
\end{remark}
\begin{remark}
Assumption \eqref{eqKExtra}, Equation \eqref{eqDifferent} in the proof below, and the Mean Value Theorem imply that there exists a constant $C>0$ depending only on $\varepsilon$, $C_{\mathcal{I}}$, and $E$ such that 
\begin{align*}
\vert \vec{z}^{\,\sigma}_1 - \vec{z}^{\,\sigma}_0\vert  \leq C \quad \text{ for any } \sigma>\overline{\sigma}.
\end{align*}
However, either or both of the limits \begin{align*}\vec{C}_\text{CMC} = \lim_{\sigma\to\infty} \vec{z}^{\,\sigma}_0 \quad \text{ and } \quad \vec{C}_\text{STCMC} \definedas \lim_{\sigma\to\infty} \vec{z}^{\,\sigma}_1,\end{align*} may fail to exist. On the other hand, $\vec{C}_\text{STCMC} = \lim_{\sigma\to\infty} \vec{z}^{\,\sigma}_1$ converges if and only if $\lim_{\sigma\to\infty}\left\lbrace\vec{z}^{\,\sigma}_0 - \frac{1}{32\pi E}\int_{\mathbb{S}^{2}_{\sigma}} \frac{ \left(\sum_{k,l}\pi_{kl}x^k x^l\right)^2\vec{x}}{\sigma^3} \, d\mu^\delta\right\rbrace$ converges. This in particular shows that $K$ can in a sense ``compensate'' for the diverging coordinate center of the CMC-foliation. See Section \ref{secHYCN} for more details on this.
\end{remark}

\begin{proof}   
First, note that the constants $C_{\mathcal{I}_\tau}$ are uniformly bounded by the constant $C_{\mathcal{I}}$. Second, pick constants $b\geq 0$ and $\eta\in(0,1]$ that will remain fixed in this argument and always use the class $\mathcal{A}(0,b,\eta)$ in what follows. Also, $C>0$ and $\overline{\sigma}>0$ denote generic constants that may vary from line to line, but depend only on $\varepsilon$, $C_{\mathcal{I}}$, and $\vert E\vert $ (as well as on our global choice of $b$ and $\eta$). From now on we assume that $\sigma>\overline{\sigma}$ is fixed. For our choice of $\sigma$ and for $\tau\in [0,1]$, we let $z_\tau^i = (z_\tau^{\sigma})^i$, $i=1,2,3$, denote the components of the coordinate center $\vec{z}\,(\Sigma^{\sigma}_{\tau})$ of the unique surface $\Sigma^{\sigma}_\tau$ of constant spacetime mean curvature $\mathcal{H}(\Sigma^{\sigma}_\tau)\equiv\sfrac{2}{\sigma}$ in the initial data set $\mathcal{I}_\tau$ (see Section \ref{secExistence} for details). Since the index $\sigma$ is assumed to be fixed, it will be suppressed in the remainder of this proof. 

According to Proposition \ref{propCenterMotion}, the variation of the coordinate center with respect to $\tau$ is given by the formula 
\begin{align}\label{eqCenterMotion}
\left\vert  \partial_\tau z_\tau^i  - \frac{3}{\vert \Sigma_\tau\vert } \int_{\Sigma_\tau}   u_\tau \nu_\tau^i \, d\mu \right\vert  \leq \frac{C}{\sigma^{\frac{3}{2} +\varepsilon}} \Vert u_\tau\Vert _{L^2(\Sigma_\tau)}, 
\end{align}
where $u_\tau$ is the respective lapse function for an arbitrary $\tau\in[0,1]$ and $\nu_\tau$ is the outward pointing unit normal to $\Sigma_\tau \hookrightarrow (M^3,g)$. In order to pass from \eqref{eqCenterMotion} to \eqref{eqDeformationCoM}, we will apply Lemma \ref{lemLapse} with $\delta=2-\varepsilon\geq\tfrac{3}{2}$. By this, we have that $\mathcal{L} u_\tau = O(\sigma^{-3})$, $\Vert (u_\tau)^t\Vert _{W^{2,2}(\Sigma_\tau)} \leq C\sigma$ and $\Vert (u_\tau)^d\Vert _{W^{2,2}(\Sigma_\tau)} \leq C\sigma^{\frac{1}{2}-\varepsilon}$.

Next, by \eqref{eqEigenEstimates1b} and Lemma \ref{lemComparison} we have  
\begin{align}\label{eqNormal2}
\left\Vert \sqrt{\frac{3}{\vert\Sigma_\tau\vert}} \,\nu_\tau^i - f_\tau^i\right\Vert_{L^2(\Sigma_\tau)}\leq \frac{C}{ \sigma^{\frac{1}{2}+\varepsilon}},
\end{align}
where $f_\tau^i$ denotes the $i$-th eigenfunction of the operator $-\Delta^{\Sigma_\tau}$, see Section \ref{secEigenLaplace}. Then we may rewrite \eqref{eqCenterMotion} by a Cauchy--Schwarz Inequality and Lemma \ref{lemLapse} with $\delta=2-\varepsilon$ as
\begin{align*}
\left\vert  \partial_\tau z_\tau^i  - \sqrt{\frac{3}{\vert \Sigma_\tau\vert }} \int_{\Sigma_\tau} u_\tau f_\tau^i  \, d\mu \right\vert  \leq \frac{C}{\sigma^{\frac{3}{2} +\varepsilon}}\Vert u_\tau\Vert _{L^2(\Sigma_\tau)}\leq \frac{C}{\sigma^{\frac{1}{2} +\varepsilon}}.
\end{align*}
At the same time, Proposition \ref{propInvertibility} implies that
\begin{align*}
\left\vert \sqrt{\frac{3}{\vert \Sigma_\tau\vert }}\int_{\Sigma_\tau} u_\tau f_\tau^i \, d\mu - \frac{\sigma^3}{2\sqrt{3} m_{\mathcal{H}}\sqrt{\vert \Sigma_\tau\vert }}\int_{\Sigma_\tau} f_\tau^i \, \mathcal{L} (u_\tau)^t \, d\mu \right\vert  \leq \frac{C}{\sigma^{\varepsilon}}\frac{\Vert (u_{\tau})^{t}\Vert_{L^{2}(\Sigma_{\tau})}}{\sqrt{\vert\Sigma_\tau\vert}}\leq\frac{C}{\sigma^{\varepsilon}},
\end{align*}
where $m_{\mathcal{H}}$ is the Hawking mass of $\Sigma_\tau$ with respect to $\mathcal{I}_{\tau}$. Recall that $\vert m_{\mathcal{H}}-E\vert\leq C\sigma^{-\varepsilon}$ by Proposition \ref{propInvertibility} so that $m_{\mathcal{H}}\neq0$ follows from $E\neq0$. Thus,
\begin{align}\label{eqCoMIntermediate}
\left\vert \partial_\tau z_\tau^i - \frac{\sigma^3}{2\sqrt{3} m_{\mathcal{H}}\sqrt{\vert \Sigma_\tau\vert }}\int_{\Sigma_\tau} f_\tau^i \, \mathcal{L}(u_\tau)^t \, d\mu \right\vert  \leq \frac{C}{\sigma^{\varepsilon}}.
\end{align}
Note that a computation in the proof of Lemma \ref{lemLapse} shows that
\begin{align*}
\int_{\Sigma_\tau} f_\tau^i \, \mathcal{L} (u_\tau)^d \, d\mu = \int_{\Sigma_\tau}   (u_\tau)^d \mathcal{L} f_\tau^i \,  d\mu + C \sigma^{-4} \|u^d\|_{L^2(\Sigma_\tau)}
\leq \frac{C}{\sigma^{\tfrac{5}{2}+\varepsilon}}\|u^d\|_{L^2(\Sigma_\tau)}
\leq C \sigma^{-2-2\varepsilon},
\end{align*}
hence $\int_{\Sigma_\tau} f_\tau^i \, \mathcal{L} (u_\tau)^t \, d\mu = \int_{\Sigma_\tau} f_\tau^i \, \mathcal{L} u_\tau \, d\mu+ O(\sigma^{-2-2\varepsilon})$. Consequently, in view of \eqref{eqPDEu} and \eqref{eqNormal2}, and the Cauchy--Schwarz Inequality, \eqref{eqCoMIntermediate}  is equivalent to
\begin{align}\label{eqSame}
\left\vert \partial_\tau z_\tau^i - \frac{\sigma}{8\pi m_{\mathcal{H}}}\int_{\Sigma_\tau}  \frac{\tau(\tr_{\Sigma_\tau} K)^2}{H_{\Sigma_\tau}} \nu_\tau^i \, d\mu \right\vert  \leq \frac{C}{\sigma^{\varepsilon}}.
\end{align}
Since the expression in the integral is of order $O(\sigma^{-3})$ as mentioned above, and since
\begin{align*}
\tr_{\Sigma_\tau} K = \tr K - K(\nu_\tau, \nu_\tau) = \pi(\nu_\tau,\nu_\tau),
\end{align*}
using as before the fact that $m_{\mathcal{H}} = E + O(\sigma^{-\varepsilon})$ along the STCMC-foliation, we conclude by Lemma~\ref{lemComparison} that
\begin{align}\label{eqDifferent}
\left\vert \partial_\tau z_\tau^i - \frac{\tau}{16\pi E}\int_{\mathbb{S}^{2}_{\sigma}} \frac{x^i \left(\sum_{k,l}\pi_{kl}x^k x^l\right)^2}{\sigma^3} \, d\mu^\delta \right\vert  \leq \frac{C}{\sigma^{\varepsilon}}.
\end{align}
Integrating this with respect to $\tau$ over $[0,1]$, we obtain \eqref{eqDeformationCoM}. 
\end{proof}

\begin{theorem}[STCMC-coordinate expression]\label{thCoordinateExpression}
Let $\mathcal{I}=\IDS$ be a $C^2_{\sfrac{1}{2}+\varepsilon}$-asymptotically Euclidean initial data set with respect to an asymptotic coordinate chart $\vec{x}\colon M^{3}\setminus \mathcal{B}\to\R^{3}\setminus\overline{B_{R}(0)}$  and decay constant $C_{\mathcal{I}}$, with non-vanishing energy $E \neq 0$. Assume in addition that
\begin{align}
\vert K\vert \leq C_{\mathcal{I}} \vert \vec{x}\,\vert ^{-2}
\end{align}
for all $\vec{x}\in\R^{3}\setminus\overline{B_{R}(0)}$ and that $g$ satisfies the Riemannian $C^2_{\sfrac{3}{2}+\varepsilon}$-Regge--Teitelboim condition. Then the coordinate center $\vec{C}_\text{STCMC}$ of the unique foliation by surfaces of constant spacetime mean curvature is well-defined if and only if the \emph{correction term}
\begin{align*}
Z^i \definedas \frac{1}{32\pi E} \lim_{r\to\infty}\int_{\mathbb{S}^{2}_{r}} \frac{x^i \left(\sum_{k,l}\pi_{kl}x^k x^l\right)^2}{r^3} \, d\mu^\delta
\end{align*}
limits exist for $i=1,2,3$. In this case, we have 
\begin{align}\label{eqCenterSTCMC}
\vec{C}_\text{STCMC} = \vec{C}_\text{B\'OM} +  \vec{Z},
\end{align}
where $\vec{C}_\text{B\'OM}$ is the Beig--\'O~Murchadha center of mass and $\vec{Z}=(Z^1,Z^2,Z^3)$, or equivalently
\begin{align*}
\begin{split}
C^i_\text{STCMC} = \frac{1}{16\pi E}\lim_{r\to \infty} \left[\int_{\mathbb{S}^{2}_{r}}   \left(x^i \sum_{k,l}(\partial_k g_{kl} - \partial_l g_{kk})\frac{x^l}{r} - \sum_k \left(g_{ki}\frac{x^k}{r}-g_{kk}\frac{x^i}{r}\right) \right) d\mu^\delta \right.\\ + \left.\int_{\mathbb{S}^{2}_{r}} \frac{x^i \left(\sum_{k,l}\pi_{kl}x^k x^l\right)^2}{2r^3} \, d\mu^\delta\right], \quad  i=1,2,3.
\end{split}
\end{align*} 
\end{theorem}

\begin{proof}
Since $g$ satisfies the Riemannian $C^2_{\sfrac{3}{2}+\varepsilon}$-Regge--Teitelboim condition, $\vec{C}_\text{CMC}$  is well-defined and equal to $\vec{C}_\text{B\'OM}$, see \cite[Theorem 6.3]{NerzCMC}. The result is then a direct consequence of  Lemma \ref{lemSTCMC-CoM}. 
\end{proof}

\begin{remark}
We expect that the coordinate center $\vec{C}_\text{STCMC}$ of the spacetime mean curvature foliation translates in a certain sense to the center of mass $\vec{C}_\text{Sz}$ as defined by Szabados \cite{Szabados}. Similar to \eqref{eqCenterSTCMC}, the definition of Szabados takes form $\vec{C}_\text{Sz} = \vec{C}_\text{B\'OM} +  \vec{S}$, where $\vec{S}$ is characterized by the extrinsic curvature $K$ of the initial data set and the time function $t$ which realizes this initial data set as a slice in an asymptotically Minkowskian spacetime. The relation between $\vec{C}_\text{STCMC}$, $\vec{C}_\text{Sz}$, and the Chen--Wang--Yau center of mass defined via optimal isometric embeddings in Minkowski spacetime (see \cite{CWY}) will be studied in detail in our forthcoming work.
\end{remark}

\begin{definition}
We suggest to call the expression $\vec{C}_\text{mB\'OM}\definedas\vec{C}_\text{B\'OM}+\vec{Z}$ the \emph{modified Beig--\'O Murchadha center of mass}.
\end{definition}

In Section \ref{secHYCN}, we will give an example that shows that the contribution of the correction term $\vec{Z}$ is indeed relevant and fixes a problem of the CMC-center of mass uncovered in~\cite{CN}.

\begin{remark}\label{remConditions}
It is not obvious which decay conditions on $\mathcal{I}$ (e.g.~versions of Regge--Teitelboim, faster decay assumptions on $K$, etc.) are sufficient to ensure convergence of the correction term $\vec{Z}$ without forcing it to vanish entirely. This will be studied in detail in our forthcoming work. More importantly, sufficient conditions for convergence of $\vec{C}_{\text{STCMC}}$ that do not force vanishing of $\vec{Z}$ in accordance with the example studied in Section \ref{secHYCN} will also be studied in our forthcoming work.
\end{remark}

We conjecture the following sufficient conditions, in line with Bartnik's \cite{Bartnik} and Chru\'sciel's \cite{Chrusciel} corresponding results for convergence of ADM-energy and ADM-linear momentum. 
\begin{conjecture}\label{conj:convergence}
We conjecture that the coordinate expression we derived will converge for asymptotic coordinates $\vec{x}$ if $\mu x^{i}\in L^{1}(M^{3})$.
\end{conjecture}

%% file: evolution.tex
% !TEX root = main.tex
\subsection{Evolution}
In this section, we will study the evolution of the coordinate center of the unique foliation by surfaces of constant spacetime mean curvature under the Einstein evolution equations. We will show that the STCMC-center of mass has the same evolution properties as a point particle in special relativity, evolving according to the formula 
\begin{align*} 
\frac{d}{dt} \vec{C}_{STCMC} = \frac{\vec{P}}{E}.\\[-0.5ex]
\end{align*}
Note that the analogous formula is valid for the CMC-center of mass and also for Chen--Wang--Yau's center of mass, although under stronger decay assumptions, see \cite{Nerzevo} and \cite{CWY}, respectively.

\begin{theorem}[Time evolution of STCMC-foliation]\label{thEvolution}
Let $(\R\times M^{3},\mathfrak{g})$ be a smooth, globally hyperbolic Lorentzian spacetime satisfying the Einstein equations with energy momentum tensor $\mathfrak{T}$. Suppose that, outside a set of the form $\R\times \mathcal{K}$, $\mathcal{K}\subset M^{3}$ compact, there is a diffeomorphism $\id_{\R}\times\,\vec{x}\colon \R\times(M^{3}\setminus\mathcal{K})\to \R\times(\R^{3}\setminus\overline{B_{R}(0)})$ which gives rise to asymptotic coordinates $(t,\vec{x})$ on $\R\times(M^{3}\setminus\mathcal{K})$. 

Assume that $\mathcal{I}_{0}=(\{0\}\times M^{3},g,K,\mu,J)\hookrightarrow(\R\times M^{3},\mathfrak{g})$ is a $C^2_{\sfrac{1}{2}+\varepsilon}$-asymptotically Euclidean initial data set with respect to the coordinate chart $\vec{x}$ and with $E\neq0$, and suppose additionally that $K=O_{1}(\vert\vec{x}\,\vert^{-2})$ with constant $C_{\mathcal{I}}$ as $\vert\vec{x}\,\vert\to\infty$. Now consider the $C^{1}$-parametrized family of $C^2_{\sfrac{1}{2}+\varepsilon}$-asymptotically Euclidean initial data sets 
\begin{align*}
\mathcal{I}(t)=(\{t\}\times M^{3},g(t),K(t),\mu(t),J(t))\hookrightarrow(\R\times M^{3},\mathfrak{g})
\end{align*}
with respect to $\vec{x}$ which starts from $\mathcal{I}(0) = \mathcal{I}_{0}$, and which exists for all $t\in(-t_{*},t_{*})$ for some $t_{*}>0$. Assume furthermore that the constants $C_{\mathcal{I}(t)}$ are uniformly bounded on $(-t_{*},t_{*})$, without loss of generality such that $C_{\mathcal{I}(t)}\leq C_{\mathcal{I}_{0}}$.

Assume the foliation $\mathcal{I}(t)$ has initial lapse $N=1+O_2(\vert\vec{x}\,\vert^{-\frac{1}{2}-\varepsilon})$ as $\vert\vec{x}\,\vert\to\infty$ with decay measuring constant denoted by $C_{N}$ and initial shift $X=0$, and suppose furthermore that the initial stress tensor $S$ of $\mathcal{I}_{0}$ satisfies $S = O(\vert\vec{x}\,\vert^{-\frac{5}{2}-\varepsilon})$ as $\vert\vec{x}\,\vert\to\infty$. There is a constant $\overline{t}>0$, depending only on $\varepsilon$, $C_{\mathcal{I}_{0}}$, $C_{N}$, and $E(0)$ such that the following holds: If the initial data set $\mathcal{I}_{0}$ has well-defined STCMC-center of mass $\vec{C}_{\text{STCMC}}\,(0)$ then the STCMC-center of mass $\vec{C}_{\text{STCMC}}\,(t)$ of $\mathcal{I}(t)$ is also well-defined for $\vert t\vert<\overline{t}$. Furthermore, the initial velocity at $t=0$ is given by
\begin{align}\label{eqEvolution}
\left.\frac{d}{dt}\right\vert_{t=0} \vec{C}_{STCMC} = \frac{\vec{P}}{E}.
\end{align}
Moreover, we have that $\left.\frac{d}{dt}\right\vert_{t=0} E=0$ and $\left.\frac{d}{dt}\right\vert_{t=0} \vec{P}=\vec{0}$.
\end{theorem}

\begin{remark}
In fact, one gets more information about the evolution of the STCMC-center of mass from the proof of Theorem \ref{thEvolution}: Not only the coordinate expression $\vec{C}_{\text{STCMC}}$ evolves according to \eqref{eqEvolution}, but also the individual leaves of the ‹STCMC-foliation evolve in a way more and more close to a translation in direction $\sfrac{\vec{P}}{E}$ according to formula \eqref{eqMovement}.
\end{remark}

\begin{remark}\label{remExpect}
We expect that the evolution of the leaves of the STCMC-foliation as well as the evolution of $\vec{C}_{\text{STCMC}}$ can actually be understood when replacing the condition $K=O_{1}(\vert\vec{x}\,\vert^{-2})$ by more natural conditions related to integrability criteria on the constraints when integrated against $\vec{x}$. We will investigate this in our forthcoming work, see also Remark \ref{remConditions} and Conjecture \ref{conj:convergence}.
\end{remark}

\begin{remark}
It is straightforward to prove a version of this theorem allowing for non-vanishing shift. As this is not of primary interest here and can also be fixed by a suitable gauge, we will not go in this direction.
\end{remark}

\begin{proof}
Throughout this proof, dotted quantities like for example $\dot{E}$ will denote time derivatives at $t=0$, e.g.~$\dot{E}=\left.\frac{d}{dt}E\right\vert_{t=0}$. Moreover, $\overline{\sigma}>0$, $\overline{t}>0$, and $C>0$ denote generic constants that may vary from line to line, but depend only on $\varepsilon$, $C_{\mathcal{I}_{0}}$, and $E(0)$ as well as on $C_{N}(0)$, the constant in the $O$-term of $N$.

The Einstein evolution equations with zero shift are given at $t=0$ by
\begin{align}\label{eqEvog}
\dot{g}_{ij}&\definedas\left.\frac{d}{dt}\right\vert_{t=0} g_{ij} =\left.2NK_{ij}\right\vert_{t=0}=O(\vert\vec{x}\,\vert^{-\frac{3}{2}-\varepsilon}), \\
\begin{split}\label{eqEvoK}
\dot{K}_{ij}&\definedas\left.\frac{d}{dt}\right\vert_{t=0} K_{ij} = \left.\{\hess_{ij} N + N(\mathfrak{Ric}_{ij} - \ric_{ij} + 2K^k_i K_{jk} - \tr K K_{ij})\}\right\vert_{t=0}\\
&=O(\vert\vec{x}\,\vert^{-\frac{5}{2}-\varepsilon}),
\end{split}
\end{align}
where $\mathfrak{Ric}$ is the Ricci tensor of the spacetime $(\R\times M^3,\mathfrak{g})$, which is completely determined by the stress-energy tensor $\mathfrak{T}$ through the Einsteins equations
\begin{align*}
\mathfrak{Ric} - \tfrac{1}{2}\mathfrak{Scal} \, \mathfrak{g} = \mathfrak{T}.
\end{align*}
Here, we used that $\mathfrak{Scal}=-\tr_{\mathfrak{g}}\mathfrak{T}=O(\vert\vec{x}\,\vert^{-\frac{5}{2}-\varepsilon})$ and thus $\mathfrak{Ric}_{ij}=O(\vert\vec{x}\,\vert^{-\frac{5}{2}-\varepsilon})$. In particular, we see from the ADM-formulas \eqref{eqADME} and \eqref{eqADMP} respectively that the energy and linear momentum satisfy $\dot{E} =0$ and $\dot{\vec{P}} = \vec{0}$, with respect to this variation. Also, as $E(0)\neq0$ and $E(t)$ is continuous, the initial data sets $\mathcal{I}(t)$ have a unique foliation by surfaces of constant spacetime mean curvature near infinity for $\vert t\vert<\overline{t}$ and mean curvature radii $\sigma>\overline{\sigma}$. This foliation depends in a $C^{1}$-fashion on $t$ for $\vert t\vert<\overline{t}$ which can be seen as follows: Perform a method of continuity procedure around $t=0$ as in the proof of Theorems \ref{thExistence}, \ref{thUniqueness} or in the proof of Lemma \ref{lemFoliation}. Note that, as the initial shift $X$ was chosen to vanish, $X=0$, we in fact know that the STCMC-surface variation is normal at $t=0$. This gives you an STCMC-foliation of $\mathcal{I}(t)$ for each $\vert t\vert<\overline{t}$ which depends on $t$ in a $C^{1}$-fashion. By Theorem \ref{thUniqueness}, this family of foliations must coincide with the one studied here and must thus depend on $t$ in a $C^{1}$-fashion. 

Now fix $\sigma>\overline{\sigma}$ and let $\Sigma_t$ denote the unique leaf of the STCMC-foliation with constant spacetime mean curvature $\mathcal{H}(\Sigma_t, \mathcal{I}(t)) \equiv \sfrac{2}{\sigma}$ in the initial data set $\mathcal{I}(t)$, where $\vert t\vert<\overline{t}$. For this, we use the product rule to see that 
\begin{align}\label{eqVarEinstein}
0 = \left.\frac{d}{dt}\right\vert_{t=0} \mathcal{H}(\Sigma_t, \mathcal{I}(t)) = \left.\frac{d}{dt}\right\vert_{t=0} \mathcal{H}(\Sigma_t, \mathcal{I}(0)) + \left.\frac{d}{dt}\right\vert_{t=0} \mathcal{H}( \Sigma_{0}, \mathcal{I}(t)).
\end{align}
The first term on the right shows how the spacetime mean curvature changes if $\Sigma_t$ is considered to be a varying surface in the initial data set $\mathcal{I}(0)$. As the initial shift vanishes and we thus have an initially normal variation, this term exactly gives our well-known linearization $L^{\mathcal{H}}u$, where the operator $L^{\mathcal{H}}$ is as defined in Lemma \ref{lemLinearizationCMC} and $u$ is the initial lapse function of the normal variation. The second term on the right shows how the spacetime mean curvature changes if $\Sigma$ is considered to be a fixed surface in the varying initial data set $\mathcal{I}(t)$. More precisely, this term is 
\begin{align*}
\left.\frac{d}{dt}\right\vert_{t=0} \mathcal{H}(\Sigma_{0}, \mathcal{I}(t)) = \frac{H \dot{H} - P \dot{P}}{\sqrt{H^2 - P^2}},
 \end{align*}
where $H=H(0)$ and $P=P(0)$.

In order to compute $\dot{H}= \left.\frac{d}{dt}\right\vert_{t=0} H (\Sigma_{0}, \mathcal{I}(t))$, we introduce geodesic normal coordinates in a neighborhood $U\subset M^{3}$ of $\Sigma_{0}$, with $y^n$ such that $\partial_{n}$ is the outer unit normal to the level set $\{y^{n}=\text{const.}\}$, in particular $\partial_{n} = \nu$ on $\Sigma_{0}$, and $y^\alpha$, $\alpha=1,2$, are some coordinates on $\Sigma_{0}$ transported to $U$ along the flow generated by $\partial_{n}$. Note that in this case $g_{nn}=1$, $g_{n\alpha}=0$, and 
\begin{align}\label{eqChristoffel1}
A_{\alpha\beta}&=g(\nabla_\alpha \partial_n, \partial_\beta)=\Gamma^\gamma_{\alpha n} g_{\gamma\beta} = \Gamma^\gamma_{n\beta} g_{\gamma\alpha}= -\Gamma^n_{\alpha\beta},\\\label{eqChristoffel2}
\Gamma^{n}_{n\alpha}&=0,\\\label{eqChristoffel3}
\Gamma^{\gamma}_{\alpha\beta}&=(\Gamma^{\Sigma_{0}})^{\gamma}_{\alpha\beta}
\end{align}
in $U$, for all $\alpha,\beta,\gamma=1,2$. We will now drop the index on $\Sigma_{0}$ and just write $\Sigma$ instead for notational convenience. We use the standard formula for the variation of the second fundamental form when the ambient metric is changing (see e.g.~Section~3 in~\cite{Lott}\footnote{note that our sign convention for the second fundamental form is the opposite of \cite{Lott}.}) and compute, using first \eqref{eqEvog} and \eqref{eqEvoK}, second the decay properties of $N$ and the decay estimate for the second fundamental form $A=\tfrac{H}{2}g^{\Sigma}+\mathring{A}$, with $g^{\Sigma}$ the metric induced on $\Sigma$ by $\mathcal{I}(t)$, namely
\begin{align*}
\vert A\vert\leq \frac{C}{\sigma}
\end{align*}
from Proposition \ref{propRegularity}, third adding some rich zeros, fourth because $J=O(\sigma^{-3-\varepsilon})$ by assumption, fifth by \eqref{eqChristoffel1}, \eqref{eqChristoffel2}, \eqref{eqChristoffel3}, and finally \eqref{eqFallOff}, and Proposition \ref{propRegularity} to obtain
\begin{align*}
\begin{split}
\dot{H} & = \dot{g}^{\alpha\beta} A_{\alpha\beta} + g^{\alpha\beta} \dot{A}_{\alpha\beta} \\
&= \dot{g}^{\alpha\beta} A_{\alpha\beta} -  \tfrac{1}{2} g^{\alpha\beta} (2\nabla_\alpha \dot{g}_{n\beta} - \nabla_n \dot{g}_{\alpha\beta} - A_{\alpha\beta} \dot{g}_{nn}) \\
&= -2K^{\alpha\beta} A_{\alpha\beta} - 2 g^{\alpha\beta} \nabla_\alpha K_{n\beta} + g^{\alpha\beta}\nabla_n K_{\alpha\beta} + H K_{nn} + O(\sigma^{-3-\varepsilon})\\
&= - H\tr_\Sigma K - 2 \nabla^i K_{ni} + \nabla_n \tr K + \nabla_n K_{nn} + H K_{nn} + O(\sigma^{-3-\varepsilon})\\
&= - H\tr_\Sigma K - J(\nu) - \nabla^\alpha K_{n\alpha} + H K_{nn} + O(\sigma^{-3-\varepsilon})\\
&= - H\tr_\Sigma K - \left(\divg^\Sigma K(\cdot,\nu) - K^{\alpha\beta} A_{\alpha\beta} + H K_{nn}\right) + H K_{nn} + O(\sigma^{-3-\varepsilon})\\
&=  - \divg^\Sigma K(\cdot,\nu) - \tfrac{1}{\sigma} \tr_\Sigma K +  O(\sigma^{-3-\varepsilon}),
\end{split}
\end{align*}
where $J$ is the momentum density defined on page \pageref{textmuJ}. Further, let $\eta$ denote the timelike future unit normal vector field to $M^3\hookrightarrow (\R\times M^{3},\mathfrak{g})$. Then it is straightforward to check that, by \eqref{eqEvog}, \eqref{eqEvoK}, the decay assumptions on the initial data set and on $N$, as well as the definition of $\mu$ and $S$ from page \pageref{textmuJ}
\begin{align*}
\begin{split}
\dot{P} &= \dot{g}^{\alpha\beta} K_{\alpha\beta} + g^{\alpha\beta} \dot{K}_{\alpha\beta} \\
& = -2N\vert K\vert^{2}+\tr_\Sigma \dot{K} \\ 
& = -2N \vert K\vert^{2}+\Delta^{\Sigma} N +H \nu(N) + N \tr_\Sigma \mathfrak{Ric} - \tr_\Sigma \ric + 2 \tr_\Sigma (K \circ K) - \tr K \tr_\Sigma K) \\
& =N \tr_\Sigma \mathfrak{Ric}+O(\sigma^{-\frac{5}{2}-\varepsilon})\\
&= N(\mathfrak{Scal}-\mathfrak{Ric}(\nu,\nu)+\mathfrak{Ric}(\eta,\eta))+O(\sigma^{-\frac{5}{2}-\varepsilon})\\
& = N(-\tr_{\mathfrak{g}}\mathfrak{T} - (\mathfrak{T} (\nu,\nu) - \tfrac{1}{2} \tr_{\mathfrak{g}}\mathfrak{T}) + (\mathfrak{T} (\eta,\eta) + \tfrac{1}{2}\tr_{\mathfrak{g}} \mathfrak{T})) + O(\sigma^{-\frac{5}{2}-\varepsilon})\\
& = N(- \mathfrak{T} (\nu,\nu) +  \mathfrak{T} (\eta,\eta)) + O(\sigma^{-\frac{5}{2}-\varepsilon})\\
& = N(- S(\nu,\nu) +  \mu) + O(\sigma^{-\frac{5}{2}-\varepsilon})\\
& = O(\sigma^{-\frac{5}{2}-\varepsilon}).
\end{split}
\end{align*}

Summing up and multiplying by $\sqrt{1-(\tfrac{P}{H})^{2}}$, it follows from \eqref{eqVarEinstein} that 
\begin{align}\label{eqLUsingLapse}
\mathcal{L} u = -\dot{H} + \tfrac{P}{H} \dot{P} = \divg^\Sigma K(\cdot,\nu) + \tfrac{1}{\sigma} (\tr K - K(\nu,\nu)) + O(\sigma^{-3-\varepsilon}),
\end{align}
where the operator $\mathcal{L}$ is given by \eqref{eqLinearization}. This uniquely defines $u\in W^{2,2}(\Sigma)$ by the invertibility of $\mathcal{L}$, see Proposition \ref{propInvertibility} as the right hand side is bounded and thus in~$L^{2}(\Sigma)$. 

In order to compute the initial velocity of $\mathcal{I}(0)$, we first need to compute the initial velocity $\dot{\vec{z}}$ of the Euclidean coordinate center 
\begin{align*}
\vec{z}\,(t)&=\frac{1}{\vert\Sigma_{t}\vert_{\delta}}\int_{\Sigma_{t}}\vec{x}\,d\mu^{\delta}
\end{align*}
of $\Sigma_t$. We remind the reader that we chose coordinates $\vec{x}$ which do not depend on $t$. Relying on Proposition \ref{propCenterMotion}, we will now compute the variation of the coordinate center, $\dot{\vec{z}}$, starting from the variation formula
\begin{align}\label{eqSTCMCEvolution}
\left\vert \dot{z}^i  - \frac{3}{\vert\Sigma\vert} \int_{\Sigma}   u \nu^i \, d\mu \right\vert\leq \frac{C}{\sigma^{\frac{3}{2} +\varepsilon}} \Vert u\Vert_{L^2(\Sigma)}.
\end{align}
The idea is to argue as in the proof of Lemma \ref{lemSTCMC-CoM} and pass from $u$ to $\mathcal{L}u$, and from $\nu^i$ to $f_i$ in \eqref{eqSTCMCEvolution}. For this we note that since $\mathcal{I}$ is $C^2_{\sfrac{1}{2}+\varepsilon}$-asymptotically Euclidean, and since $K$ has faster fall-off $K=O_{1}(\vert\vec{x}\,\vert^{-2})$, we have $\mathcal{L}u=O(\vert\vec{x}\,\vert^{-3})$ by \eqref{eqLUsingLapse}. A computation identical to the one in the proof of Lemma \ref{lemSTCMC-CoM} yields
\begin{align*}
\left\vert \dot{z}^i  - \frac{\sigma^3}{2\sqrt{3} E\sqrt{\vert\Sigma\vert}} \int_{\Sigma} f_i \mathcal{L}u \, d\mu \right\vert \leq \frac{C}{\sigma^{\varepsilon}}.
\end{align*}
Thus, \eqref{eqLUsingLapse} and integration by parts give us
\begin{align*}
\left\vert \dot{z}^i + \frac{\sigma^3}{2\sqrt{3} E\sqrt{\vert\Sigma\vert}}\int_\Sigma \left(K\left(\nu, \nabla^\Sigma f_i + \frac
{\nu}{\sigma}f_i\right) - \frac{1}{\sigma} f_i \tr K \right) \, d\mu \right\vert \leq \frac{C}{\sigma^{\varepsilon}}.
\end{align*}
Recall that $\Sigma$ is approximated by the coordinate sphere $\mathbb{S}_\sigma^{2}$, as described in Section~\ref{secHypersurfaces}. In particular, the functions $f_i$, $i=1,2,3$, are close to the respective eigenfunctions $f^\delta_i$ of the Laplacian $-\Delta^{\mathbb{S}_\sigma}$, see \eqref{eqEigenEstimates1b}. Thus 
\begin{align*}
\left\vert \dot{z}^i + \frac{\sigma^2}{4\sqrt{3\pi} E}\int_{\mathbb{S}^{2}_{\sigma}} \left(K\left(\frac{\vec{x}}{\sigma}, \nabla^{\mathbb{S}_\sigma^{2}} f^\delta_i + \frac
{\vec{x}}{\sigma^{2}}f^\delta_i\right) - \frac{1}{\sigma} f^\delta_i \tr K \right) \, d\mu \right\vert \leq \frac{C}{\sigma^{\varepsilon}},
\end{align*}
where $\frac{\vec{x}}{\sigma}$ is the unit normal vector to $\mathbb{S}_\sigma^{2} \hookrightarrow (\mathbb{R}^3,\delta)$. Furthermore, a computation shows that 
\begin{align*}
\nabla^{\mathbb{S}_\sigma^{2}} f^\delta_i = \frac{\sqrt{3}}{\sqrt{4\pi}\sigma^2} \partial_{x^i} - \frac{\vec{x}}{\sigma^{2}} f^\delta_i.  
\end{align*}
Since $f^\delta_i = \tfrac{\sqrt{3}x^i}{\sqrt{4\pi}\sigma^{2}}$ and since $g_{ij} = \delta_{ij} + O(\sigma^{-\frac{1}{2}-\varepsilon})$, we finally arrive at
\begin{align}\label{eqMovement}
\left\vert \dot{z}^i - \frac{1}{8\pi E}\int_{\vert\vec{x}\,\vert=\sigma} \pi_{ij} \frac{x^j}{\sigma} \, d\mu \right\vert \leq \frac{C}{\sigma^{\varepsilon}}.
\end{align}
Passing to the limit when $\sigma\to\infty$ we obtain the result.\end{proof}

\subsection{Poincar\'e-covariance and accordance with Special Relativity}\label{secPoincare}
As we have seen before, whether or not a given $2$-surface is STCMC is in fact independent of a choice of slice (as well as of a choice of coordinates). In this sense, STCMC-surfaces are covariant in the sense of General Relativity. The role of the initial data set then is to select a unique family of STCMC-surfaces near the asymptotic end of the spacetime, forming its abstract STCMC-center of mass. In this sense, STCMC-foliations and the associated (abstract) center of mass are Poincar\'e-covariant.

We will now discuss the transformation behavior of the STCMC-coordinate center under the asymptotic Poincar\'e group of the ambient spacetime --- assuming vanishing angular momentum in the boost case. Dealing with angular momentum and treating the boost case more adequately will be left for our future work. Let $\mathcal{I}=\IDS$ be an initial data set which is $C^2_{\sfrac{1}{2}+\varepsilon}$-asymptotically Euclidean with respect to asymptotic coordinates $\vec{x}$ and has $E\neq0$.\\

\paragraph*{\emph{Euclidean motions.}} Consider the coordinates $\vec{y}\definedas O\vec{x}+\vec{T}$, with $O$ an orthogonal rotation matrix and $\vec{T}\in\R^3$ a translation vector. In other words, $\vec{y}$ arises from $\vec{x}$ through a Euclidean motion. Then, for each leaf $\Sigma^\sigma$ of the STCMC-foliation constructed in Theorem \ref{thExistence}, we find that the Euclidean center of $\Sigma^\sigma$ with respect to the $\vec{y}$-coordinates is given by
\begin{align*}
\frac{1}{\vert\Sigma^\sigma\vert_{\delta}}\int_{\Sigma^\sigma}\vec{y}\,d\mu^\delta&=O\left(\frac{1}{\vert\Sigma^\sigma\vert_{\delta}}\int_{\Sigma^\sigma}\vec{x}\,d\mu^\delta\right)+\vec{T}.
\end{align*}
Thus, the STCMC-coordinate center
\begin{align*}
\vec{C}^{\,\vec{y}}_\text{STCMC}&=\lim_{\sigma\to\infty}\frac{1}{\vert\Sigma^\sigma\vert_{\delta}}\int_{\Sigma^\sigma}\vec{y}\,d\mu^\delta
\end{align*}
with respect to the coordinates $\vec{y}$ converges if and only the STCMC-coordinate center 
\begin{align*}
\vec{C}^{\,\vec{x}}_\text{STCMC}&=\lim_{\sigma\to\infty}\frac{1}{\vert\Sigma^\sigma\vert_{\delta}}\int_{\Sigma^\sigma}\vec{x}\,d\mu^\delta
\end{align*}
converges with respect to the coordinates $\vec{x}$ converges and if they converge, we find 
\begin{align*}
\vec{C}^{\,\vec{y}}_\text{STCMC}&=O\vec{C}^{\,\vec{x}}_\text{STCMC}+\vec{T}
\end{align*}
as one would expect from Euclidean Geometry, Newtonian Gravity, and from the description of the spacetime position of a point particle in Special Relativity.\\

\paragraph*{\emph{Time translations.}} The transformation behavior of $\vec{C}_\text{STCMC}$ under asymptotic time translation corresponds to its evolution behavior under the Einstein equations. In other words, Theorem \ref{thEvolution} tells us under the additional assumption $K=O_{1}(\vert\vec{x}\,\vert^{-2})$ that 
\begin{align*}
\left.\frac{d}{dt}\right\vert_{t=0} \vec{C}_{STCMC} = \frac{\vec{P}}{E}
\end{align*}
which corresponds precisely to the instantaneous law of motion of a point particle in Special Relativity.\\ 

\paragraph*{\emph{Boosts.}} The last constituent of the asymptotic Poincar\'e group of the spacetime are of course the asymptotic boosts. In a given (asymptotic region of a) spacetime 
\begin{align*}
(\R\times M^3,\mathfrak{g}&=-N^2(dx^0+X_i dx^i)(dx^0+X_j dx^j)+h_{ij}dx^idx^j)
\end{align*}
with asymptotic coordinates $x^\alpha=(x^0,x^i)$ and suitably decaying lapse $N$, shift $X$, and tensor $h$, a \emph{boosted initial data set $\mathcal{I}=\IDS\hookrightarrow(\R\times M^3,\mathfrak{g})$} is any spacelike hypersurface arising as the set $\lbrace y^0=0\rbrace$ with respect to a \emph{boosted coordinate system} $y^\alpha\definedas \Lambda^\alpha_\beta x^\beta$, $y^\alpha=(y^0,\vec{y}\,)$, meaning that the matrix $\Lambda$ is a boost. If the lapse $N$, the shift $X$, and the tensor $h$ decay suitably fast in space and time coordinate directions, the boosted initial data set $\lbrace y^0=0\rbrace=\mathcal{I}$ is in fact $C^1_{\sfrac{1}{2}+\varepsilon}$-asymptotically Euclidean with respect to $\vec{y}$. It is thus reasonable to ask how the STCMC-coordinate centers of the initial data sets $\lbrace x^0=0\rbrace$ and $\lbrace y^0=0\rbrace$ are related (if they converge). The corresponding question was addressed by Szabados \cite{Szabados} for the B\'OM-center of mass although from a slightly different perspective. Of course, we expect that the STCMC-coordinate center boosts like the position of a point particle in Special Relativity,
namely
\begin{align}\label{eqBoostTransfo}
\begin{pmatrix}{0}\\{\vec{C}^{\,\vec{y}}_\text{STCMC}}\end{pmatrix}&=\Lambda\begin{pmatrix}{0}\\{\vec{C}^{\,\vec{x}}_\text{STCMC}}\end{pmatrix},
\end{align}
at least in the absence of angular momentum. This can easily be verified for example for a boosted slice (over the canonical slice) in the Schwarzschild spacetime where in fact the centers both coincide with the center of symmetry $\vec{0}$. Similarly, if one first spatially translates the coordinates on the Schwarzschild spacetime and then considers a boosted slice, the transformation law will be as in \eqref{eqBoostTransfo}. In both of these examples, the deviation $\vec{Z}$ introduced in Theorem \ref{thCoordinateExpression} in fact vanishes, so that the transformation law \eqref{eqBoostTransfo} already holds for the CMC-B\'OM-center of mass. In view of Section \ref{secHYCN} below, it is possible to construct examples of boosted slices in the Schwarzschild spacetime by boosting the example discussed below. One can then see that \eqref{eqBoostTransfo} in fact also applies in this case, but only for the STCMC-center of mass, and \emph{not} for the CMC-B\'OM-center of mass. However, the computation is so tedious that we prefer not to show it here as it is not particularly enlightening. 

A proof of a generalized version of \eqref{eqBoostTransfo} incorporating the angular momentum will be given elsewhere, see also Remarks \ref{remConditions}, \ref{remExpect}.

%% file: HY-CN.tex
% !TEX root = main.tex
As briefly sketched in Sections \ref{secPrelim}, \ref{secMain} and analyzed in more detail in Section \ref{secADMstyleexpr}, determining the coordinate center of an asymptotic foliation is tricky and depends on choosing suitable coordinates (see also Conjecture \ref{conj:convergence}). In \cite[Section 6]{CN}, this was illustrated by explicitly computing the coordinate center of the CMC-foliation of an asymptotically Euclidean ``graphical'' time-slice in the Schwarzschild spacetime of mass $m\neq0$. This example, to be described in more detail below,  satisfies all assumptions in~\cite{HY}, in particular those of Theorem 4.2, but yet its CMC-coordinate center does \emph{not} converge. Equivalently, its B\'OM-center also does not converge. After a brief introduction to the graphical example discussed in \cite{CN}, we will compute that the STCMC-coordinate center \eqref{eqCenterSTCMC} does in fact converge in this example and moreover converges to the origin $\vec{0}$, i.e.~to the center of symmetry of the spherically symmetric spacetime as one would expect.

We consider the Schwarzschild spacetime  $(\R\times M^3,\mathfrak{g})$ of mass $m\neq0$ in Schwarz\-schild coordinates, meaning that
\begin{align*}
M^3&=(\max\lbrace{0,2m\rbrace},\infty)\times\mathbb{S}^2\ni(r,\vec{\eta}),\\
\mathfrak{g}&=-N^2dt^2+g,\\
g&=N^{-2}dr^2+r^2d\Omega^2,\\
N(r)&=\sqrt{1-\frac{2m}{r}},
\end{align*}
where $d\Omega^2$ denotes the canonical metric on $\mathbb{S}^2$. We will freely switch between polar coordinates $(r,\eta)$ and the naturally corresponding Cartesian coordinates $\vec{x}$ defined on~$M^3$. A \emph{graphical time-slice} in the (automatically vacuum) Schwarzschild spacetime is an initial data set $(M^3_T,g_T,K_T,\mu_T\equiv0,J_T\equiv0)$ arising as the graph of a smooth function $T\colon M^3\to\R$ ``over'' the \emph{canonical time-slice $\lbrace{t=0\rbrace}$} (in time-direction), meaning that 
\begin{align*}
M_T\definedas \lbrace t=T(\vec{x}) : \vec{x}\in M^3\rbrace,
\end{align*}
while $g_T$ is the Riemannian metric induced on $M_T\hookrightarrow(\R\times M^3,\mathfrak{g})$ and $K_T$ is the second fundamental form induced by this embedding with respect to the future pointing unit normal.\\

\paragraph*{\emph{Computing the CMC-coordinate center of mass (via the B\'OM-center of mass).}} Clear\-ly, the center of mass of the canonical time-slice $\lbrace{t=0\rbrace}$ of the Schwarz\-schild spacetime is the coordinate origin, $\vec{C}_{\text{CMC}}=\vec{C}_{\text{B\'OM}}=\vec{0}$. We will now compute this vector for graphical time-slices with the asymptotic decay conditions on $T$ chosen such that $(M^3_T,g_T,K_T,\mu_T\equiv0,J_T\equiv0)$ is $C^2_{1}$-asymptotically Euclidean with respect to the coordinates~$\vec{x}$. To most easily comply with the asymptotic decay conditions specified in Section \ref{secPrelim}, we will assume that $T=O_{k}(r^{0})$ as $r\to\infty$, with $k\geq3$.

Now let $\vec{y}\definedas \vec{x}\vert_{M_T}$ denote the induced coordinates on $M_T$. As computed in \cite[Section 6]{CN}, the metric $g_T$ and second fundamental form $K_T$\footnote{The corresponding formula for the second fundamental form in \cite{CN} has a typo which we corrected here. We thank Axel Fehrenbach for pointing this out to us.} are given by
\begin{align*}
(g_T)_{ij}&= g_T(\partial_{y^{i}},\partial_{y^{j}})=g(\partial_{x^{i}},\partial_{x^{j}})-N^{2}\,T_{,i}\,T_{,j}= g_{ij}-N^{2}\,T_{,i}\,T_{,j},\\
(K_T)_{ij}&= \frac{T_{,i}N_{,j}+T_{,j}N_{,i}+N\nabla^2_{ij}T-N^2T_{,i}T_{,j}\, dN(\operatorname{grad}_g T)}{\sqrt{1-N^2\vert dT\vert^2_g}}
\end{align*}
in the coordinates $\vec{y}$. A straightforward computation shows that the graphical initial data set $(M^3_T,g_T,K_T,\mu_T\equiv0,J_T\equiv0)$ is indeed $C^2_{1}$-asymptotically Euclidean and in fact has $E=m\neq0$.

When evaluating the B\'OM-center of mass surface integral on a finite coordinate sphere with respect to the $\vec{y}$-coordinates in $M_T$, using $s\definedas\vert \vec{y}\,\vert$ and $\vec{\eta}\definedas\tfrac{\vec{y}}{s}$, we find
\begin{align*}
\begin{split}
\left[\vec{C}\left(\mathbb{S}^{2}_{s}(\vec{0}\,)\right)\right]_{l}=\frac{1}{16\pi m} &\int_{\mathbb{S}^2_s(\vec{0}\,)} \left[\left((g_T)_{ij,i} - (g_T)_{ii,j}\right)\frac{y_l y^j}{s} -\left((g_T)_{il}\frac{y^i}{s}-(g_T)_{ii}\frac{y_l}{s}\right)\right]\, d\mu^\delta\\
=\frac{1}{16\pi m} &\left\lbrace\int_{\mathbb{S}^2_s(\vec{0}\,)}\right. \left[\left(g_{ij,i} - g_{ii,j}\right)\frac{y_l y^j}{s} -\left(g_{il}\frac{y^i}{s}-g_{ii}\frac{y_l}{s}\right)\right]\, d\mu^\delta\\
&\quad\left.-\int_{\mathbb{S}^2_s(\vec{0}\,)}\right.\left((N^2T_{,i}T_{,j})_{,i} - (N^2T_{,i}T_{,i})_{,j}\right)\frac{y_l y^j}{s}\, d\mu^\delta\\
&\quad\left.+\int_{\mathbb{S}^2_s(\vec{0}\,)}\left(N^2T_{,i}T_{,l}\frac{y^i}{s}-N^2T_{,i}T_{,i}\frac{y_l}{s}\right)\, d\mu^\delta\right\rbrace\\
=\frac{1}{16\pi m} &\left\lbrace\int_{\mathbb{S}^2_s(\vec{0}\,)}\right.\left[-s\Delta_\delta T\, T_{,j}\eta^j\eta_l+s\nabla^2_\delta T(\operatorname{grad}_\delta T,\vec{\eta})\eta_l\right]d\mu^\delta\\
&\quad\left.+\int_{\mathbb{S}^2_s(\vec{0}\,)}\left[T_{,i}\eta^i\,T_{,l}-\vert dT\vert_\delta^2\,\eta_l\right]d\mu^\delta\right\rbrace+O(s^{-1}).
\end{split}
\end{align*}

As in \cite[Section 6]{CN}\footnote{In fact, we are using Schwarzschild coordinates, here, while in \cite{CN}, isotropic coordinates are used. This allows us to treat the case of $m<0$ as well and does not affect the asymptotic computations.} , we pick a fixed vector $\vec{0}\neq\vec{u}\in\R^{3}$ and set
\begin{align*}
T\colon\R^3\setminus\overline{B_{2m}(\vec{0}\,)}\to\R:\vec{x}\mapsto\sin\left(\ln r\right)+\frac{\vec{u}\cdot\vec{x}}{r}=O_{\infty}\left(r^{0}\right).
\end{align*}
We point out that this choice of $T$ ensures that $(M^3,g_T,\mu_T\equiv0)$ satisfies the Riemannian $C^2_{\frac{1}{2}+\varepsilon}$-Regge--Teitelboim conditions  so that \cite[Cor.~4.2]{Nerzevo} or \cite[Theorem 6.3]{NerzCMC} apply and ensure that $\vec{C}_\text{CMC}=\vec{C}_\text{B\'OM}$ or that both diverge. One directly computes from the above expression for $\vec{C}_\text{B\'OM}\left(\mathbb{S}^{2}_{s}(\vec{0}\,)\right)$ that
\begin{align*}
\vec{C}_\text{B\'OM}\left(\mathbb{S}^{2}_{s}(\vec{0}\,)\right)=\frac{ \cos(\ln s) }{3 m}\vec{u} +O(s^{-1})
\end{align*}
which diverges as $s\to\infty$. Hence, the B\'OM- and thus also the CMC-coordinate center diverge in this example.

\paragraph*{\emph{Computing the STCMC-coordinate center of mass (via Formula \eqref{eqCenterSTCMC}).}} In order to check whether the STCMC-coordinate center of the $C^2_1$-asymptotically Euclidean initial data set $(M^3_T,g_T,K_T,\mu_T\equiv0,J_T\equiv0)$ converges, one needs to compute the STCMC-leaves $\Sigma_\sigma$ and the coordinate averages $\vec{z}\,(\Sigma_\sigma)$ and check whether they converge as $\sigma\to0$. However, the proof of Theorem \ref{thCoordinateExpression} asserts that $\vec{C}_\text{STCMC}$ converges if and only if the coordinate expression given in \eqref{eqCenterSTCMC} converges, or in other words if and only if
\begin{align*}
\vec{C}_\text{STCMC}\left(\mathbb{S}^{2}_{s}(\vec{0}\,)\right)=\vec{C}_\text{B\'OM}\left(\mathbb{S}^{2}_{s}(\vec{0}\,)\right)+\vec{Z}\left(\mathbb{S}^{2}_{s}(\vec{0}\,)\right)
\end{align*}
converges as $s\to\infty$, where we recall that, using $E=m$ and $(\pi_T)_{kl}=-(K_T)_{ij}+\tr_{g_T}K_T(g_T)_{ij}$, we know that 
\begin{align*}
Z^i\left(\mathbb{S}^{2}_{s}(\vec{0}\,)\right) &= \frac{1}{32\pi m}\int_{\mathbb{S}^{2}_{s}(\vec{0}\,)} \frac{y^i \left((\pi_T)_{kl}\,y^k y^l\right)^2}{s^3} \, d\mu^\delta\\
&=\frac{1}{32\pi m}\int_{\mathbb{S}^{2}_{s}(\vec{0}\,)}s^2\eta^i\left(\Delta_\delta T-\nabla^2_\delta T(\vec{\eta},\vec{\eta})\right)^2 \, d\mu^\delta +O(s^{-1})\\
&=-\frac{\cos(\ln s)}{3m}u^i +O(s^{-1}).
\end{align*}
Our (diverging) spacetime correction term $\vec{Z}$ thus precisely compensates for the divergence occurring in $\vec{C}_\text{CMC}=\vec{C}_\text{B\'OM}$. Hence the STCMC-coordinate center of the considered graphical slice converges to $\vec{0}$ as desired.

%% file: appendix.tex
% !TEX root = main.tex
In this appendix we collect some standard results about closed surfaces in a $C^2_{\sfrac{1}{2}+\varepsilon}$-asymptotically Euclidean initial data set $\mathcal{I}=\IDS$ (see Definition \ref{defAEdata} and Remark \ref{remAERiem}) that are repeatedly used in this paper.

\begin{lemma}\label{lemComparison}
Let $\mathcal{I}=\IDS$ be a $C^2_{\sfrac{1}{2}+\varepsilon}$-asymptotically Euclidean initial data set with asymptotic coordinate chart $\vec{x}\colon M^{3}\setminus \mathcal{B} \to \mathbb{R}^3 \setminus \overline{B_R(0)}$ and let $\Sigma\hookrightarrow M^{3}\setminus \mathcal{B}$ be a closed, oriented $2$-surface. Using the chart $\vec{x}$, we may also view $\Sigma$ as a surface in $\mathbb{R}^3 \setminus \overline{B_R(0)}$ equipped with the Euclidean metric $\delta$, with the induced metric denoted by $\delta^\Sigma$. Then there exist positive constants $c$ and $C$ depending only on $\varepsilon$ and $C_{\mathcal{I}}$ such that the following holds, provided that the Euclidean distance to the coordinate origin $\vert\vec{x}\,\vert$ on $\Sigma$ satisfies $\vert\vec{x}\,\vert \geq c$:
\begin{itemize}
\item The normals $\nu$ and $\nu^\delta$ of $\Sigma$ in the metrics $g$ and $\delta$ satisfy\\[-3ex]
\begin{eqnarray*}
\vert \nu - \nu^\delta\vert  & \leq & C \vert \vec{x}\,\vert ^{-\frac{1}{2}-\varepsilon}, \\
\vert \nabla \nu - \nabla^\delta \nu^\delta\vert  & \leq & C \vert \vec{x}\,\vert ^{-\frac{3}{2}-\varepsilon}.
\end{eqnarray*}
\item The volume elements $d\mu$ and $d\mu^\delta$ satisfy\\[-3ex]
\begin{align*}
d\mu - d\mu^\delta = O(\vert \vec{x}\,\vert ^{-\frac{1}{2}-\varepsilon})\,d\mu.
\end{align*}
\item The respective second fundamental forms $A$ and $A^\delta$ satisfy\\[-3ex]
\begin{align*}
\vert A - A^\delta\vert  \leq C (\vert \vec{x}\,\vert ^{-\frac{3}{2}-\varepsilon} + \vert \vec{x}\,\vert ^{-\frac{1}{2}-\varepsilon} \vert A\vert ),
\end{align*}
and the respective mean curvatures $H$ and $H^\delta$ are related via 
\begin{align*}
\vert H - H^\delta\vert  \leq C (\vert \vec{x}\,\vert ^{-\frac{3}{2}-\varepsilon} + \vert \vec{x}\,\vert ^{-\frac{1}{2}-\varepsilon} \vert A\vert ).
\end{align*}
Further, if $\Vert H\Vert _{L^2(\Sigma)}$ is a priori bounded, then the respective trace-free parts of the second fundamental forms satisfy 
\begin{align*}
\Vert \mathring{A}^\delta\Vert _{L^2(\Sigma,\delta^\Sigma)} \leq C \Vert \mathring{A}\Vert _{L^2(\Sigma,g^\Sigma)} + C\vert \vec{x}\,\vert ^{-\frac{1}{2}-\varepsilon}
\end{align*}
where $C$ also depends on the bound on $\Vert H\Vert _{L^2(\Sigma)}$.
\end{itemize}
\end{lemma}  

\begin{proof}
See \cite[Section 2.4]{MetzgerCMC} or \cite[Section 1.5]{Metzgerthesis}, where similar estimates are proven.
\end{proof}

The following result is a Sobolev Embedding Theorem which holds for a very general class of 
$2$-surfaces. In Section \ref{secHypersurfaces} this result is applied to $\Sigma$ being a large coordinate sphere $\mathbb{S}^2_r(\vec{z})\hookrightarrow M^{3}\setminus\mathcal{B}$ in the asymptotic end of an asymptotically Euclidean initial data set. Note that in this case one can without loss of generality replace the area radius $\sqrt{\sfrac{|\Sigma|}{4\pi}}$ in the formulation of Lemma \ref{lemSobolevEmbed} by the coordinate sphere's radius $r$ as these two radii are uniformly equivalent. In the subsequent sections, this result is applied to $\Sigma\hookrightarrow M^{3}\setminus\mathcal{B}$ being an asymptotically centered closed $2$-surface with constant spacetime mean curvature. Here we are using the fact that \eqref{eqSobolev1} is available for large asymptotically centered surfaces in $M^{3}\setminus\mathcal{B}$ in the form of \cite[Proposition 5.4]{HY}, which applies to asymptotically Euclidean initial data sets with general asymptotics as described in Section \ref{secPrelim}.

\begin{lemma}\label{lemSobolevEmbed}
Let $(\Sigma,g^\Sigma)$ be a closed, oriented 2-surface with area radius $r=\sqrt{\sfrac{|\Sigma|}{4\pi}}$. If there is a constant $C_S$ such that for any Lipschitz continuous function $f$ on $\Sigma$, the so-called first Sobolev Inequality 
\begin{eqnarray}\label{eqSobolev1}
\Vert f \Vert _{L^2\left(\Sigma\right)}  \leq C_S \,r^{-1} \Vert f\Vert _{W^{1,1}\left(\Sigma\right)}
\end{eqnarray}
holds, then we also have the Sobolev Inequality\footnote{Note that the constant in \eqref{eqSobolev2} is not necessarily optimal.} 
\begin{eqnarray}\label{eqSobolev2}
\Vert f\Vert _{L^\infty\left(\Sigma\right)}  \leq 32\, C_S^2\, r^{-1} \Vert f\Vert _{W^{2,2}\left(\Sigma\right)}
\end{eqnarray}
for any $f\in W^{2,2}(\Sigma)$. 
\end{lemma}  

\begin{proof}
With the first Sobolev Inequality \eqref{eqSobolev1} at hand, the following Sobolev Inequalities can be derived for $p>2$:
\begin{eqnarray*}
\|f\|_{L^p(\Sigma)} \leq \frac{C_S }{2} p\, r^{\tfrac{2}{p}-1} \| f\|_{W^{1,2}(\Sigma)}  \qquad \text{ for any }  f\in W^{1,2}(\Sigma), \\
\|f\|_{L^\infty(\Sigma)}\leq 2^{\tfrac{2(p-1)}{p-2}}C_S\, r^{-\tfrac{2}{p}} \| f\|_{W^{1,p}(\Sigma)} \qquad \text{ for any } f\in W^{1,p}(\Sigma), 
\end{eqnarray*}
see e.g. \cite[Proposition II.1.3]{NerzThesis} for details. As a consequence, for any $f\in W^{2,2}(\Sigma)$ we have
\begin{eqnarray*}
\|f\|_{L^\infty(\Sigma)}&\leq& 2^{\tfrac{2(p-1)}{p-2}}C_S\, r^{-\tfrac{2}{p}} \| f\|_{W^{1,p}(\Sigma)} \\
&=& 2^{\tfrac{2(p-1)}{p-2}}C_S\, r^{-\tfrac{2}{p}} \left(\| f\|_{L^{p}(\Sigma)} + r \| \nabla^\Sigma f\|_{L^{p}(\Sigma)}\right) \\
& \leq& 2^{\tfrac{2(p-1)}{p-2}-1}p\,C_S^2\, r^{-1} \left(\| f\|_{W^{1,2}(\Sigma)} + r \| \nabla^\Sigma f\|_{W^{1,2}(\Sigma)}\right) \\
& \leq & 2^{\tfrac{2(p-1)}{p-2}}p\,C_S^2\, r^{-1} \| f\|_{W^{2,2}(\Sigma)}.
\end{eqnarray*}
for any $p>2$. In particular, for $p=4$ we have \eqref{eqSobolev2}.
\end{proof}

The following result is well-known, see e.g. \cite[Corollary 2.10]{Sauter} (adapted from \cite[Chapter 2]{CK}).

\begin{lemma}\label{lemCalderon}
Let $(\Sigma,g^\Sigma)$ be a $2$-surface of spherical topology with Gaussian curvature $K$ satisfying 
\begin{align*}
\tfrac{1}{2}\leq r^2 K \leq 2,
\end{align*} 
where $r=\sqrt{\sfrac{|\Sigma|}{4\pi}}$ is the area radius of $\Sigma$. Then there is a universal constant $C$
such that for any $f \in W^{2,2}(\Sigma)$ we have
\begin{align*}
\Vert f - \overline{f}\Vert_{W^{2,2}(\Sigma)} \leq C r^2 \Vert \Delta^\Sigma f \Vert _{L^2(\Sigma)}, 
\end{align*} 
where $\overline{f}$ denotes the mean value of $f$ on $\Sigma$.
\end{lemma}

%% file: appendix2.tex
% !TEX root = main.tex
Let $\mathcal{I}=(M^{3},g,K,\mu, J)$  be a $C^{2}_{\sfrac{1}{2}+\varepsilon}$-asympto\-tically Euclidean initial data set and let $\Sigma\hookrightarrow M^{3}$ be a closed, oriented $2$-surface. Let $(u^\alpha)$ be coordinates on $\Sigma$ and let $\partial_\alpha$ denote the respective tangent vectors to $\Sigma$, for $\alpha\in\lbrace{1,2\rbrace}$. Here and in the rest of this appendix, we use the convention that Greek indices $\alpha,\beta,\gamma \in\{1,2\}$ refer to coordinate vector fields tangential to $\Sigma$.

In a neighborhood of $\Sigma$, the normal geodesic coordinates $y\colon \Sigma \times (-\xi,\xi) \to M^{3}$ are defined for some $\xi>0$, see Section \ref{subsecStab} for details. In this neighborhood we may write $g=dt^2 + g_t$ where $g_t$ is the induced metric on $\Sigma_t\definedas y(\Sigma, t)$. 

Consider a $2$-surface $S$ given as the graph of a function $f$ with $\vert f\vert  < \xi $ over $\Sigma$, i.e.
\begin{align}
S = \graph f =  \{y(q,f(q)) : q \in \Sigma\}.
\end{align}
Since the vector $-\partial_t+\nabla^{g_t} f$ is normal to $S$ at the point $(q,t)=(q, f(q))$, the vectors $\partial_\alpha+(\partial_\alpha f)\partial_t$ are tangent to $S$ at this point. As a consequence, the induced metric on $S$ has components given by 
\begin{equation*}
(g_S)_{\alpha\beta} = g (\partial_\alpha+(\partial_\alpha f)\partial_t, \partial_\beta+(\partial_\beta f)\partial_t)=(g_t)_{\alpha\beta}+\partial_\alpha f \,\partial_\beta f,
\end{equation*}  
with the components of the inverse given by 
\begin{equation*}
(g_S)^{\alpha\beta}= (g_t)^{\alpha\beta}-\frac{f^\alpha f^\beta}{1+|df|^2_{g_t}}.
\end{equation*}
Here and in what follows, $2$-dimensional indices $\alpha$, $\beta$, $\dots$ are raised with respect to $g_t$, and all quantities are computed at the point $(q,t)=(q, f(q))$, unless stated otherwise. A straightforward computation then shows that the mean curvature of $S=\graph f$ is given~by
\begin{equation*}
\begin{split}
H(S) = \left((g_t)^{\alpha\beta}-\frac{f^\alpha f^\beta}{1+|df|^2_{g_t}}\right)\frac{\hess^{g_t}_{\alpha\beta} f + (A_t)_{\alpha\beta}+2 (A_t)_{\phantom{\alpha} \alpha}^\gamma 
                          (\partial_\beta f) ( \partial_\gamma f)}{\sqrt{1+|df|^2_{g_t}}},
\end{split} 
\end{equation*}
where $A_t$ is the second fundamental form of $\Sigma_t$. We also have
\begin{equation}\label{eqPNormal}
                   P(S)=\tr_S K  =  \left( (g_t)^{\alpha\beta} - \frac{f^\alpha f^\beta}{1+|df|_{g_t}^2} \right) (K_{\alpha\beta}+2(\partial_\alpha f)K_{t \beta}+(\partial_\alpha f)(\partial                                     _\beta f) K_{t t}).
\end{equation}
\begin{proposition}
If $S=\graph f$ is a surface of constant spacetime mean curvature $\mathcal{H}(S)\equiv\sfrac{2}{\sigma}$ then $f$ satisfies the equation
\begin{equation}\label{eqJangHeight}
a^{\alpha\beta}\partial_\alpha \partial_\beta f + b^\alpha \partial_\alpha f = F,
\end{equation}
where, with $P$ given by \eqref{eqPNormal}, we use the shorthands
\begin{eqnarray*}
a^{\alpha\beta} & \definedas & \left((g_t)^{\alpha\beta}-\frac{f^\alpha f^\beta}{1+|df|_{g_t}^2}\right)\frac{1}{\sqrt{1+|df|_{g_t}^2}},\\
  b^\gamma  & \definedas & - \left((g_t)^{\alpha\beta}-\frac{f^\alpha f^\beta}{1+|df|_{g_t}^2}\right)\frac{(\Gamma^{g_t})^\gamma_{\alpha\beta}}{\sqrt{1+|df|_{g_t}^2}},\\  
	   F & \definedas & - \left((g_t)^{\alpha\beta}-\frac{f^\alpha f^\beta}{1+|df|_{g_t}^2}\right)\frac{(A_t)_{\alpha\beta}+2 (A_t)_{\phantom{\alpha} \alpha}^\gamma (\partial_\beta f) ( \partial_\gamma f)}
		         {\sqrt{1+|df|^2_{g_t}}} + \sqrt{P^2 + \frac{4}{ \sigma^2}}.
\end{eqnarray*}
\end{proposition}